\documentclass[journal,twoside,web]{ieeecolor}
\usepackage{generic}
\pdfoutput=1
\usepackage{cite}
\usepackage{amsmath,amssymb,amsfonts}
\usepackage{algorithmic}
\usepackage{graphicx}
\usepackage{textcomp}
\def\BibTeX{{\rm B\kern-.05em{\sc i\kern-.025em b}\kern-.08em
	T\kern-.1667em\lower.7ex\hbox{E}\kern-.125emX}}
\makeatletter
\def\fnum@figure{\textcolor{subsectioncolor}{\sf Fig.~\thefigure}}
\def\fnum@table{\textcolor{subsectioncolor}{\sf TABLE~\thetable}}
\makeatother
\usepackage{mathtools}%
\usepackage{ulem}%
\usepackage{bm}%
\usepackage[hidelinks]{hyperref}%
\usepackage{tabularray}

\newtheorem{theorem}{Theorem}
\newtheorem{lemma}{Lemma}
\newtheorem{remark}{Remark}

\newcommand{\remarkEnd}{\hfill$\triangleleft$}
\newcommand{\proofEnd}{\hfill$\QED$}
\newcommand{\col}{\operatorname{col}}
\newcommand{\rank}{\operatorname{rank}}

\newcommand{\diag}{\operatorname{diag}}

\renewcommand{\d}{\mathrm{d}}
\newcommand{\e}{\mathrm{e}}
\usepackage{xcolor}
\definecolor{matlabBlue}{rgb}{0.00000,0.44700,0.74100}%
\definecolor{matlabOrange}{rgb}{0.92900,0.69400,0.12500}%
\definecolor{matlabGreen}{rgb}{0.4660,0.6740,0.1880}%
\definecolor{myRed}{rgb}{0.80000,0.15000,0.05000}
\definecolor{myNiceRed}{rgb}{0.85000,0.32500,0.09800}%
\definecolor{myYellow}{rgb}{0.92900,0.69400,0.12500}%
\definecolor{myMagenta}{rgb}{0.49400,0.18400,0.55600}%
\definecolor{myGreen}{rgb}{0.46600,0.67400,0.18800}%
\definecolor{myDarkGreen}{rgb}{0.2330, 0.420, 0.0940}%
\definecolor{viridis1of4}{rgb}{0.25497,0.26253,0.52898}%
\definecolor{viridis2of4}{rgb}{0.16393,0.47039,0.55815}%
\definecolor{viridis3of4}{rgb}{0.13523,0.65936,0.51723}%
\definecolor{viridis4of4}{rgb}{0.48323,0.82294,0.31468}%
\definecolor{colA0}{rgb}{0.8500,0.3250,0.0980}%
\definecolor{colA1}{rgb}{0.0000,0.4470,0.7410}%
\definecolor{colA2}{rgb}{0.9290,0.6940,0.1250}%
\definecolor{colA3}{rgb}{0.4660,0.6740,0.1880}%
\definecolor{colA4}{rgb}{0.4940,0.1840,0.5560}%
\definecolor{ghostColor}{rgb}{0.8,0.8,0.8}%
\definecolor{outerFrameColor}{rgb}{0.3,0.3,0.3}%
\definecolor{someDarkGreen}{rgb}{0.1660,0.540,0.1380}%
\definecolor{changeColor}{rgb}{0.02,0.200,0.900}
\definecolor{leaderColor}{rgb}{0.80000,0.15000,0.05000}%
\definecolor{followerColor}{rgb}{0.49400,0.22,0.55600}%
\definecolor{followerColor1}{rgb}{0.7,0.22,0.556}%
\definecolor{followerColor2}{rgb}{0.5,0.22,0.7}%
\definecolor{followerColor3}{rgb}{0.3,0.75,0.5}%
\definecolor{followerColor4}{rgb}{0.5,0.75,0.2}%
\definecolor{communicationColor}{rgb}{0.00000,0.44700,0.74100}%
\newcommand{\Otimes}{\!\otimes\!}
\renewcommand{\emph}[1]{\textit{#1}}

\newcommand{\mybar}[3]{%
	\mathrlap{\hspace{#2}\overline{\scalebox{#1}[1]{\phantom{\ensuremath{#3}}}}}\ensuremath{#3}
}
\newcommand{\mytilde}[3]{%
	\mathrlap{\hspace{#2}\tilde{\scalebox{#1}[0.85]{\phantom{\ensuremath{#3}}}}}\ensuremath{#3}
}
\usepackage{pgfplots,pgfplotstable}

\newcommand{\pfeilTextAbstand}{-0.75mm}

\newcommand{\blockTextSize}{\small}

\newcommand{\myTikzSet}{
	\tikzset{
		frame/.style={ 
			rectangle,
			minimum height=8mm,
			minimum width=10mm,
			thick,
			inner xsep=3pt,
			inner ysep=1pt,
			draw,
		},
		>=latex,
	};}
\RequirePackage{pic/pgf-extrect}

\newcommand{\frameColor}{black!25}
\newcommand{\stableColor}{someDarkGreen}
\newcommand{\subText}[1]{\\[-0mm]{\footnotesize#1}}

\usetikzlibrary{positioning,calc,fit,spy,decorations,decorations.markings,shapes,arrows.meta,intersections,tikzmark,petri,automata}
\pgfdeclarelayer{back}
\pgfdeclarelayer{front}
\pgfsetlayers{back,main,front}
\usepackage{tabularx}
\allowdisplaybreaks[3]
\pgfplotsset{plotOpts/.style=
	{	width=0.92\linewidth,
		height=0.275\linewidth,
		at={(0\linewidth,0\linewidth)},
		scale only axis,
		every axis title/.append style={at={(current axis.north)}, yshift=-2ex, font=\footnotesize},
		every axis x label/.append style={font=\footnotesize, yshift=3mm},
		every axis y label/.append style={rotate=-90, font=\footnotesize},
		every tick label/.append style={font=\footnotesize},
		every x tick label/.append style={yshift=0.5mm},
		every y tick label/.append style={xshift=0.75mm},,
		every axis plot/.append style={line width=1pt,opacity=1},
		major tick length = 2pt,
		baseline,
		trim axis left, trim axis right
}}%
\tikzset{
	ropeStyle1/.pic={
		\node at (-0.4, 3.3) {#1};
		\draw [line width = 1] plot [smooth] coordinates {(0.5, 3.8) (0.4, 1.8) (0, 0)};
		\node [circle, scale=0.425, fill] at (0, 0) {};
		\coordinate (-left) at (-2mm, 0);
		\coordinate (-right) at (2mm, 0);
		\coordinate (-below) at (0, -0.25);
	},
	ropeStyle2/.pic={
		\node at (-0.4, 3.3) {#1};
		\draw [line width = 1] plot [smooth] coordinates {(0.5, 3) (0.4, 1.3) (0, 0)};
		\node [circle, scale=0.5, fill] at (0, 0) {};
		\coordinate (-left) at (-2mm, 0);
		\coordinate (-right) at (2mm, 0);
		\coordinate (-below) at (0, -0.25);
		\coordinate (-top-left) at (-1.5, 4.2);
		\coordinate (-top-right) at (1, 4.2);
	},
	ropeStyle3/.pic={
		\node at (-0.0, 3) {#1};
		\draw [line width = 1] plot [smooth] coordinates {(1, 5) (0.7, 2.3) (0, 0)};
		\node [circle, scale=0.65, fill] at (0, 0) {};
		\coordinate (-left) at (-2mm, 0);
		\coordinate (-right) at (2mm, 0);
		\coordinate (-below) at (0, -0.25);
		\coordinate (-top-left) at (-1.25, 5.5);
		\coordinate (-top-right) at (1.5, 5.5);
	},
	ropeStyle4/.pic={
		\node at (-0.0, 3) {#1};
		\draw [line width = 1] plot [smooth] coordinates {(1, 6) (0.7, 2.9) (0, 0)};
		\node [circle, scale=0.75, fill] at (0, 0) {};
		\coordinate (-left) at (-2mm, 0);
		\coordinate (-right) at (2mm, 0);
		\coordinate (-below) at (0, -0.25);
	},
	antennaStyle/.pic={
		\node at (0mm, 6mm) {#1};
		\draw [line width = 1] (-1mm, 0mm) -- (0mm, 0mm) -- (0mm, 4mm) -- (0mm, 0mm) -- (1mm, 0mm);
		\node [circle, scale=0.4, fill] at (0mm, 4mm) {};
		\coordinate (-left) at (-1mm, 4mm);
		\coordinate (-communicate-left) at (-2.5mm, 4mm);
		\coordinate (-communicate-right) at (2.5mm, 4mm);
	},
	communicationRight/.pic={
		\draw (-1mm, -1mm) -- (-1mm, 1mm);
		\draw (0mm, -1.5mm) -- (0mm, 1.5mm);
		\draw (1mm, -2mm) -- (1mm, 2mm);
	},
	communicationLeft/.pic={
		\draw (-1mm, -2mm) -- (-1mm, 2mm);
		\draw (0mm, -1.5mm) -- (0mm, 1.5mm);
		\draw (1mm, -1mm) -- (1mm, 1mm);
	},
}%
\newcommand{\copyrightnotice}{\begin{tikzpicture}[remember picture,overlay]
		\node[anchor=south,yshift=0pt] at (current page.south) {\parbox{\textwidth}{
			\centering\footnotesize\linespread{0}\textcopyright 2023 IEEE. Personal use of this material is permitted. Permission from IEEE must be obtained for all other uses, in any current or future media, including 	reprinting/republishing this material for advertising or promotional purposes, creating new collective works, for resale or redistribution to servers or lists, or reuse of any copyrighted
			component of this work in other works.
			\textbf{DOI: \href{https://doi.org/10.1109/TAC.2023.3272871}{10.1109/TAC.2023.3272871}}
	}};\end{tikzpicture}}%
\newcommand{\acceptancenotice}{\begin{tikzpicture}[remember picture,overlay]
		\node[anchor=north,yshift=-0pt] at (current page.north) {\parbox{\textwidth}{
			\centering\footnotesize\linespread{0} This article has been accepted for publication in IEEE Transactions on Automatic Control. This is the author's version which has not been fully edited and content may change prior to final publication. Citation information: \textbf{DOI: \href{https://doi.org/10.1109/TAC.2023.3272871}{10.1109/TAC.2023.3272871}}
		}};
	\end{tikzpicture}}%
\begin{document}
\title{Robust Cooperative Output Regulation for Networks of Hyperbolic PIDE--ODE Systems}
\author{Jakob~Gabriel and Joachim~Deutscher,~\IEEEmembership{Member,~IEEE} 
\thanks{J. Gabriel and J. Deutscher are with the Institute of Measurement, Control and Microtechnology, Ulm University,	Albert-Einstein-Allee 41, 89081 Ulm, Germany. (e-mail: \{jakob.gabriel,joachim.deutscher\}@uni-ulm.de)}}

\maketitle
%
%
%
%
\begin{abstract}
	In this paper the robust cooperative output regulation problem for multi-agent systems (MAS) with general heterodirectional hyperbolic PIDE--ODE agents is considered. 
	This setup also covers networks of ODEs with arbitrarily long input and output delays.
	The output of the agents can be defined at all boundaries, in-domain and may depend on the ODE state, while disturbances act on the agents in-domain, at the boundaries, the output and the ODE.
	The communication network is described by a constant digraph and if its Laplacian is reducible, then heterogeneous agents are permitted also in the nominal case.
	The solution is based on the cooperative internal model principle, which requires to include a diffusively driven internal model in the controller.
	The corresponding state feedback regulator design starts with a local backstepping stabilization of the coupled hyperbolic PIDE--ODE systems.
	It is shown that the remaining simultaneous stabilization of the MAS can be traced back to the simultaneous stabilization of the finite-dimensional cooperative internal model.
	Solvability conditions in terms of the network topology and the agents transfer behavior are presented.
	The new design method is applied to the formation control of a platoon of uncertain heavy ropes carrying loads to verify its applicability.
	Simulations confirm the synchronization performance achieved by the resulting networked controller.
\end{abstract}
%
%
%
%
%
%
\section{Introduction}\label{sec:intro}
\subsubsection{Background and Motivation}{\acceptancenotice\copyrightnotice}%
Networked controllers for MAS are utilized to achieve common control goals by the cooperation of multiple autonomous systems, so-called agents, although their exchange of information is restricted to a communication network.
In the last two decades, new methods for such networks of systems governed by linear and nonlinear ODEs have been developed (see, e.g., \cite{Lewis2013cooperative,Bullo2019lectures,Liang2019cooperative}).
Recently, many researchers recognized the need to extend the existing methods to agents with distributed-parameter dynamics, i.e., PDE agents.
Besides the theoretical appeal of this research problem, it is well-founded by applications.
Important examples are
networks of Lithium-Ion cells in battery management (see \cite{Ouyang16, Tang2017}),
synchronization of the attitude of flexible spacecrafts (see \cite{Chen2019distributed}),
networks of heaters in industrial furnaces (see \cite{Cap14}),
climate control in buildings (see \cite{Sangi2017platform, Borggaard2009control}) or
in environmental applications (see \cite{Tricaud2011optimal}).
Therefore, the networked control of PDE agents is an active but still emerging research topic.
Recent contributions are \cite{Demetriou2013, PilloniPisano2015, XiaScardovi2021Synchronization}, which investigate the synchronization problem for networks of parabolic PDE agents.
Furthermore, the synchronization for a MAS of wave equations is presented in \cite{AguilarOrlovPisano2021} and agents with beam dynamics are addressed in \cite{Chen2019distributed}.
In \cite{Demetriou2015TAC,SinghNatarajan2022} the synchronization of general infinite-dimensional agents described in Hilbert spaces are considered.
This is solved in \cite{Demetriou2015TAC} by optimizing the gains of a static state feedback, and, alternatively, by an adaptive state feedback.
Thereby, also the weights of the undirected networked communication are considered as a degree of freedom and utilized in the design.
Later, this approach was extended to the output feedback design for the synchronization of positive real systems in \cite{Demetriou2018IJACSP} as well as in \cite{Demetriou2019IJC} for second order infinite-dimensional systems, which model, for instance, flexible structures and wave equations with boundary dynamics.
In practice, the influence of the environment, which is represented by disturbances, needs to be dealt with.
Moreover, variations during manufacturing as well as wear and tear cause parameter differences.
Both these aspects were not considered intensively in the previous references for PDE agents.
So far, disturbances affecting only the actuated boundary are considered in \cite{Chen2021leader,Chen2021bipartite} for networks of wave PDEs and in \cite{Tang2021syncBeam} for agents with beam dynamics.
Moreover, in \cite{Demetriou2018} synchronization of more general infinite-dimensional agents that are subject to constant disturbances and some structured perturbations which are present at the boundaries or in-domain is achieved.
Recently, \cite{Deu22robustMas} showed that results for \textit{robust cooperative output regulation} of ODE agents (see \cite{SuCoopImp2013}) can be systematically extended to agents described by scalar parabolic PDEs.
Therein, output synchronization in the presence of parameter perturbations that do not destabilize the MAS is achieved.
It should be noted that PDE theory can also be successfully applied for the deployment of MAS with ODE agents on the basis of continuum models.
Results in this direction can be found in \cite{Frihauf2010leader,Freudenthaler2020pde,Terushkin2021network}.

As far as MAS with PDE agents are concerned, networks of hyperbolic PIDE--ODE systems are very interesting both from the sides of applications and of control theory.
For example, the cooperative load transport \cite{Irscheid2019flatness2,KnWo14} by means of heavy ropes carrying a mass directly leads to cooperative control problems for hyperbolic PIDE--ODE systems.
This system class is also of paramount importance to systematically deal with MAS consisting of ODE agents subject to arbitrarily long actuator and sensor delays (see \cite{ZhuKrstic2020delay,Kr08a} for single ODEs).
To the best knowledge of the authors, the robust cooperative output regulation problem has not been solved for networks of hyperbolic PIDE--ODE systems so far.
This also relates to the robust cooperative output regulation problem for MAS with ODE agents subject to delays since disturbances and model uncertainty were not considered in the recent contributions \cite{Huang2020, Huang2021}.

\subsubsection{Contributions}{\acceptancenotice\copyrightnotice}%
This paper addresses the robust cooperative output regulation, i.e., the leader-follower output synchronization despite disturbances and unknown parameter uncertainties, for networks of boundary controlled general heterodirectional hyperbolic PIDE--ODE agents on a one-dimensional spatial domain.
The outputs to be synchronized are defined at the boundaries, in-domain and may also depend on the ODE states.
All agents are subject to unknown parameter uncertainties and disturbances.
The latter and the dynamics of the leader are modeled by the solution of ODEs.
The communication between the agents uses a time-invariant network described by a directed graph.
If the Laplacian of the communication network is reducible, then distinct nominal dynamics for strongly connected groups of agents can be taken into account allowing to consider also a heterogeneous MAS in the nominal case.
It seems that even for ODE agents this setup has not been studied in the literature.
The regulator design is based on the so-called \textit{cooperative internal model principle}, which was introduced in \cite{SuCoopImp2013} for ODE agents and was extended in \cite{Deu22robustMas} to parabolic PDE agents.
Then, robust cooperative output regulation is achieved if the MAS of the agents augmented by the cooperative internal model is stabilized.
This gives rise to new and interesting challenges in the cooperative regulator design.
In particular, it requires to extend the backstepping method (see, e.g., \cite{Kr08} for an overview) for general hyperbolic PIDE--ODE systems to deal with the restricted communication topology.
Different from \cite{Deu22robustMas}, where only parabolic PDE systems in the SISO case are considered, the paper extents these results to MIMO PIDE--ODE systems consisting of general heterodirectional hyperbolic PIDEs coupled with ODEs at the boundary.
With this, the advantages of the backstepping approach are inherited for the networked controller design, for example, the stabilization for any magnitude of the coefficients or the finite-time stabilization (see \cite{Hu15a} for a discussion).
Since the solution of the robust cooperative output regulation problem amounts to solve a simultaneous stabilization problem, this has to be achieved for the network of hyperbolic PIDE--ODE systems.
To this end, the agents are stabilized locally by utilizing backstepping.
This generalizes the result \cite{DeuGehKern2018} to the larger class of hyperbolic PIDE--ODE systems.
In the next step the extension of the backstepping method is presented to take the limited communication topology into account.
In particular, a cooperative decoupling transformation is proposed so that the simultaneous stabilization problem can be traced back to the simultaneous stabilization of the finite-dimensional cooperative internal model.
This significantly simplifies the design because systematic results for the solution of the latter are available in the literature (see, e.g., \cite{Lewis2013cooperative}).
The computation of the state feedback controller is scalable, i.e., it is independent of the number of agents, all feedback gains can be calculated offline and, even though the agents are infinite-dimensional, only a communication of lumped signals from neighbors is required.
It is shown that the cooperative state feedback controller, which is designed based on the known nominal agent dynamics, ensures robust cooperative output regulation also in presence of unknown parameter uncertainties as long as the latter do not destabilize the networked controlled MAS.
This requires to verify the solvability of the so-called extended regulator equations for the cooperative control problem.
Since hyperbolic PIDEs coupled with ODEs are considered, this is a very challenging problem.
Recent results concerning the classical robust output regulation problem in \cite{Deu17a} are thereby not directly applicable and thus new ideas are introduced to verify robust cooperative output regulation.

Compared to the current literature on networked controlled PDEs, a larger and more general system class as well as more general disturbances, outputs and parameter uncertainties are considered.
Different to, e.g., \cite{Demetriou2015TAC,Demetriou2019IJC,SinghNatarajan2022}, a specific model of boundary actuated hyperbolic PIDE--ODE agents is considered which allows to employ the backstepping method and to state the design equations explicitly.
Moreover, the approach allows a transparent tuning of the closed-loop dynamics both for the local stabilization as well as for the simultaneous stabilization.
In addition, the corresponding results are directly applicable to solve the robust cooperative output regulation problem for MAS with uncertain ODE agents and large uncertain actuation and sensor delays.
Thereby, all delays may be distinct.
This extends the results in \cite{Huang2020}, where only one input and one output delay is considered.
Finally, it is worth noticing, that this paper also includes the classical robust output regulation problem for non-networked hyperbolic PIDE--ODE systems as a special case, which has not been solved in the literature so far.

\subsubsection{Organization}%
In the following Section \ref{sec:problem} the MAS in question is introduced and the robust cooperative output regulation problem is formulated.
Then, in Section \ref{sec:controller.design} the networked controller is designed by means of a local stabilization and a subsequent simultaneous stabilization.
Moreover, the nominal stability of the networked controlled MAS is shown.
Section~\ref{sec:robustCooperativeOutputRegulation} verifies that cooperative output regulation is achieved also for the uncertain MAS.
An example in Section~\ref{sec:example} demonstrates the applicability and usefulness of the new cooperative controller for the formation control of a platoon of four heavy ropes carrying loads with uncertain rope lengths and load masses.

\subsubsection{Notation}\label{sec:notation}
Variables with the double superscript $\hspace{0pt}^{i \gamma}$ are often written boldface without these indices, for instance $\bm{x} = x^{i \gamma}$.
The Kronecker product of two matrices $A = [a_{k l}] \in \mathbb{C}^{n_A \times m_A}, B \in \mathbb{C}^{n_B \times m_B}$ is $A \otimes B = [a_{k l} B] \in \mathbb{C}^{n_A n_B \times m_A m_B}$.
The vec operator is defined by $\text{vec}(A) = \col(A e_{1}, \ldots, A e_{m_{A}}) \in \mathbb{R}^{n_{A} m_{A}}$, i.e., the columns of $A$ are stacked to obtain a vector. With the inverse operator $\text{vec}^{-1}$ the original matrix $A = \text{vec}^{-1}(\text{vec}(A))$ is recovered.
Stacking the variables $x^{i \gamma} = \bm{x} \in \mathbb{R}^{n^{i} \times m^{i}}$, $i = 1, \ldots, g$, $\gamma = 1, \ldots, N^{i}$, column-wise is denoted as $\col(\bm{x}) = \col(x^{1 1}, x^{1 2}, \ldots, x^{g N^{g}})$.
Moreover, the block-diagonal matrix with $\bm{x}$ on its block-diagonal is $\diag(\bm{x}) = \diag( x^{1 1}, x^{1 2}, \ldots, x^{g N^{g}})$.
Derivatives are written using Lagrange's notation $f^\prime(z) = \frac{\d}{\d z}f(z)$ or Leibniz's notation $\d_z f(z) = \frac{\d}{\d z}f(z)$.
The matrices
\begin{align}\label{eq:Epm}
	E_{-}^{i} = \begin{bmatrix*} I_{n_{-}^{i}} \\ 0 \end{bmatrix*} \in \mathbb{R}^{n^{i} \times n_{-}^{i}}
	, &&
	E_{+}^{i} = \begin{bmatrix*} 0 \\ I_{n_{+}^{i}} \end{bmatrix*} \in \mathbb{R}^{n^{i} \times n_{+}^{i}}
\end{align}
with $n^{i} = n_{-}^{i} + n_{+}^{i}$ are frequently used.
For example, $x_{-}^{i} = E_{-}^{i \top} x^{i}$ and $x_{+}^{i} = E_{+}^{i \top} x^{i}$ with $x^{i} \in \mathbb{R}^{n^{i}}$.
%
%
%
%
{\acceptancenotice\copyrightnotice}%
\section{Problem Formulation}\label{sec:problem}%
\subsubsection{Communication Network}\label{sec:network}
The communication network is described by the \textit{digraph} $\mathcal{G} = \{ \mathcal{V}, \mathcal{E}, A_{\mathcal{G}} \}$ with the node set
\begin{align}\label{eq:agent.set}
	\mathcal{V} = \{\nu^{0}, \underbrace{\nu^{1 1}, \ldots, \nu^{1 N^{1}}}_{\in \mathcal{V}^{1}}, \ldots,  \underbrace{\nu^{g 1}, \ldots, \nu^{g N^{g}}}_{\in \mathcal{V}^{g}}\}
\end{align}
of $N+1$ vertices.
The first vertex $\nu^{0}$ represents the \textit{leader}, which generates a reference trajectory to be tracked by the outputs of the $N$ \textit{agents}
\begin{align}
	\nu^{i \gamma} := \text{$\gamma$-th agent in the group $\mathcal{V}^{i}$},
\end{align}
$i = 1, \ldots, g$, $\gamma = 1, \ldots, N^{i}$,
which are also referred to as \textit{followers}.
The latter are grouped in the $g$ sets
\begin{align}
	\mathcal{V}^{i} &= \{\nu^{i 1}, \ldots, \nu^{i N^{i}}\}, &&i=1, \ldots, g, &&\sum_{i=1}^{g} N^{i} = N.
\end{align}
Within each group $\mathcal{V}^{i}$, $i=1, \ldots, g$, the agents $\nu^{i \gamma}$, $\gamma = 1, \ldots, N^{i}$, are required to be identical in the nominal case.
In contrast, for agents of different subsets $\mathcal{V}^{i}$, $\mathcal{V}^{j}$, $i\neq j$, different nominal dynamics are allowed, including a distinct system order.
The network permits communication only along the edges in $\mathcal{E}$ and the transmitted information is weighted by the related element of the adjacency matrix $A_{\mathcal{G}}$.
In accordance with the grouping of the agents, it is assumed that the adjacency matrix $A_{\mathcal{G}}$ is partitioned and has the lower block-triangular form
\begin{align}\label{eq:Ag.frobenius}
	A_{\mathcal{G}} &= \begin{bmatrix}
		0 & 0^{\top} & \cdots &0^{\top} \\[-0.25ex]
		A_{\mathcal{G}}^{10} &A_{\mathcal{G}}^{11} &\cdots &0\hphantom{{}^\top} \\[-0.8ex]
		\vdots &\vdots &\ddots & \vdots~\,\, \\[-0.5ex]
		A_{\mathcal{G}}^{g 0} &A_{\mathcal{G}}^{g 1} & \cdots &A_{\mathcal{G}}^{g g}
	\end{bmatrix}
	\in \mathbb{R}^{(N+1) \times (N+1)}
\end{align}
with the blocks
$A_{\mathcal{G}}^{i j} = [a^{i j}_{\gamma \delta}] \in \mathbb{R}^{N^{i} \times N^{j}}$ and
$A_{\mathcal{G}}^{i 0} = [a_{\gamma}^{i 0}] \in \mathbb{R}^{N^{i}}$.
Therein, $a^{i j}_{\gamma \delta} \geq 0$ is the weight for the information received by the agent $\nu^{i \gamma}$ from the neighboring agent $\nu^{j \delta}$.
Self-loops are excluded, i.e., $a^{i i}_{\gamma \gamma} = 0$.
Moreover, $\nu^{i \gamma}$ is an \textit{informed agent} only if $a_{\gamma}^{i 0} > 0$.
In this case, the agent $\nu^{i \gamma}$ has direct access to the reference trajectory provided by the leader $\nu^{0}$ through the network.
If $a_{\gamma}^{i 0} = 0$, then $\nu^{i \gamma}$ is an \textit{uninformed agent} as it can obtain information about the reference trajectory only indirectly by communicating with other agents $\nu^{i \gamma}$.
Note that the first row in \eqref{eq:Ag.frobenius} indicates that the leader is the only root of $\mathcal{G}$ and thus determines the synchronization behavior.

The diagonal blocks $A_{\mathcal{G}}^{i i} \in \mathbb{R}^{N^{i} \times N^{i}}$, $i = 1, \ldots, g$, in \eqref{eq:Ag.frobenius} are supposed to be irreducible, hence $A_{\mathcal{G}}$ is in a \textit{Frobenius canonical form}.
Note that a square matrix $M$ is called \textit{reducible} if there is a similarity transformation $\tilde{M} = P M P^\top$ with a permutation matrix $P$, $P^{-1} = P^\top$, which yields a lower block-triangular matrix $\tilde{M}$.
Matrices that are not reducible are \textit{irreducible} (see \cite[Ch. 2.8.2]{Lewis2013cooperative}).
The Frobenius canonical form \eqref{eq:Ag.frobenius} means no loss of generality because any adjacency matrix can be mapped into \eqref{eq:Ag.frobenius} by a similarity transformation with a permutation matrix.
Furthermore, for $g=1$ the usual form of $A_{\mathcal{G}}$ results.
This setup allows to take networks of heterogeneous agents in the nominal case into account without assuming a cycle-free graph.
Note that a cycle-free condition would lead to a cascade connection of the agents without bidirectional communication.
Hence, the absence of a cycle-free condition ensures full design flexibility for the approach without further limitations.

The \textit{Laplacian} of $\mathcal{G}$ is
\begin{align}
	L_\mathcal{G} = D_\mathcal{G} - A_\mathcal{G}
		= \begin{bmatrix}
			0 & 0^{\top} \\ l_{0} & H
		\end{bmatrix}
	\in \mathbb{R}^{(N+1) \times (N+1)}
\end{align}
with the \textit{in-degree matrix}
$D_\mathcal{G} = \diag(0, d_{\mathcal{G}}^{1 1}, d_{\mathcal{G}}^{1 2}, \ldots, d_{\mathcal{G}}^{g N^{g}})$,
$d_{\mathcal{G}}^{i \gamma} = a_{\gamma}^{i 0} + \sum_{j=1}^{g} \sum_{\delta = 1}^{N^{j}} a_{\gamma \delta}^{i j}$
and $l_{0} = - \col(A_{\mathcal{G}}^{1 0}, \ldots, A_{\mathcal{G}}^{g 0})$.
Therein, $H \in \mathbb{R}^{N \times N}$ is the \textit{leader-follower matrix}, which is also lower block-triangular with the blocks $H^{i j} \in \mathbb{R}^{N^{i} \times N^{j}}$ due to \eqref{eq:Ag.frobenius}.
A useful property of $H$ is stated in the following lemma (see \cite[Lem. 3.3]{Lewis2013cooperative}).
\begin{lemma}[Leader-Follower Matrix]\label{lem:H}
	If the leader $\nu^{0}$ is the root of $\mathcal{G}$, then $\text{Re}_{\lambda \in \sigma(H)} \lambda > 0$.
\end{lemma}

For more background on graph theory, see e.g., \cite{Lewis2013cooperative}.

\subsubsection{Infinite-dimensional PIDE--ODE Agents}\label{sec:agentDynamics.pide-ode}
The dynamics of the \textit{uncertain agents} $\nu^{i \gamma}$, $i = 1, \ldots, g$, $\gamma = 1, \ldots, N^{i}$, are described by the linear heterodirectional hyperbolic PIDE--ODE system
\begin{subequations}\label{eq:agent}
	\begin{align}
		\!\partial_t \bm{x}(z, t) &\!=\! \bar{\bm{\Lambda}}(z) \partial_z \bm{x}(z, t) \!+\! \mybar{0.6}{4pt}{\bm{\mathcal{A}}}[\bm{x}(t),\bm{w}(t)](z)
		&&\hspace*{-4ex}+\!\bm{G}(z)\bm{d}(t)\hspace*{-1ex}\label{eq:agent.pde}\\
		\bm{x}_{+}(0, t) &\!=\! \bar{\bm{Q}}_{0} \bm{x}_{-}(0, t) + \bar{\bm{C}}_{0} \bm{w}(t) + \bm{G}_{0} \bm{d}(t)\label{eq:agent.bc0}, && t>0\\
		\bm{x}_{-}(1, t) &\!=\! \bar{\bm{Q}}_{1} \bm{x}_{+}(1, t) + \bm{u}(t) + \bm{G}_1 \bm{d}(t), && t>0\label{eq:agent.bc1}\\
		\dot{\bm{w}}(t) &\!=\! \bar{\bm{F}}_{w} \bm{w}(t) + \bar{\bm{B}}_{w} \bm{x}_{-}(0, t) + \bm{G}_w \bm{d}(t), && t>0\label{eq:agent.ode}\\
		\bm{y}(t) &\!=\! \bar{\bm{\mathcal{C}}}_{x} [\bm{x}(t)] + \bar{\bm{C}}_{w} \bm{w}(t) + \bm{G}_y \bm{d}(t), && t\geq0\label{eq:agent.output}
	\end{align}
\end{subequations}
with the formal operator
\begin{align}\label{eq:agent.pde.operator}
	&\mybar{0.6}{4pt}{\bm{\mathcal{A}}}[\bm{x}(t), \bm{w}(t)](z) = \bar{\bm{A}}(z) \bm{x}(z, t) + \bar{\bm{A}}_{0}(z) \bm{x}_{-}(0, t)\nonumber\\
	&\qquad + \int_{0}^{z} \bar{\bm{F}}(z, \zeta) \bm{x}(\zeta, t) \d\zeta + \bar{\bm{C}}(z) \bm{w}(t)
\end{align}
(for the boldface notation see Section \ref{sec:intro}.4).
The \textit{transport PIDE} \eqref{eq:agent.pde} with the state $x^{i \gamma}(z, t) = \bm{x}(z, t) \in \mathbb{R}^{n^{i}}$ is defined on $(z, t) \in (0, 1) \times \mathbb{R}^{+}$ and the initial condition (IC) is $\bm{x}(z, 0) = \bm{x}_{0}(z)$.
The matrix $\bar{\Lambda}^{i \gamma} = \bar{\bm{\Lambda}} = \Lambda^{i} + \Delta \bm{\Lambda} = \diag(\bar{\bm{\lambda}}_{1}, \ldots, \bar{\bm{\lambda}}_{n^{i}}) \in (C^{1}[0, 1])^{n^{i} \times n^{i}}$ contains the transport velocities $\bar{\bm{\lambda}}_{j}$.
Therein, $\Delta \bm{\Lambda} = \Delta \Lambda^{i \gamma}$ means an unknown uncertainty for each agent $\nu^{i \gamma}$ and therefore $\Lambda^{i}$ is the known nominal parameter common for the group $\mathcal{V}^{i}$.
The same convention using over-line $\bar{(\cdot)}$ and $\Delta$ is also used for the other model uncertainties.
The diagonal elements of $\bar{\bm{\Lambda}}$ are sorted in descending order, i.e.,
\begin{align}\label{eq:lambda}
	\bar{\bm{\lambda}}_{1}(z) \geq\! \ldots \!\geq \bar{\bm{\lambda}}_{n_{-}^{i}}(z) > 0 > \bar{\bm{\lambda}}_{n_{-}^{i}+1}(z) \geq\! \ldots \!\geq \bar{\bm{\lambda}}_{n^{i}}(z)\!
\end{align}
holds for all $z \in [0, 1]$.
It is assumed that if some $\bar{\bm{\lambda}}_{j}, \bar{\bm{\lambda}}_{j+1}, \ldots, \bar{\bm{\lambda}}_{j+k}$ of the positive or negative transport velocities are equal, then they are equal for all $z$. Otherwise, non-identical transport velocities $\bar{\bm{\lambda}}_{j}, \bar{\bm{\lambda}}_{j+1}$ must be distinct for all $z$.
Furthermore, $\bar{\bm{A}} = A^{i} + \Delta \bm{A} \in (C^{1}[0, 1])^{n^{i} \times n^{i}}$, $\bar{\bm{A}}_{0} = A_0^{i} + \Delta \bm{A}_{0} \in (C^{1}[0, 1])^{n^{i} \times n_{-}^{i}}$ and $\bar{\bm{F}} = F^{i} + \Delta \bm{F} \in (C^{1}([0, 1]^2))^{n^{i} \times n^{i}}$ holds, where the elements of $\bar{\bm{A}}(z) = [\bar{\bm{a}}_{j l}(z)]$ satisfy $\bar{\bm{a}}_{j j}(z) = 0$.
Moreover, for all $k$ equal transport velocities $\bar{\bm{\lambda}}_{j} \equiv \ldots \equiv \bar{\bm{\lambda}}_{j+k}$, the elements of $\bar{\bm{A}}(z)$ have to additionally verify $\bar{\bm{a}}_{j (j+1)}(z) = \ldots = \bar{\bm{a}}_{j (j+k)}(z) = 0$ and $\bar{\bm{a}}_{(j+1) j}(z) = \ldots = \bar{\bm{a}}_{(j+k) j}(z) = 0$.
This means no loss of generality since this can be ensured by the transformation given in \cite[Rem. 6]{Hu15a}.
In view of \eqref{eq:lambda}, the state can be composed according to $\bm{x}(z, t) = \col(\bm{x}_{-}(z, t), \bm{x}_{+}(z, t))$ with
$\bm{x}_{-} = x_{-}^{i \gamma} = E_{-}^{i \top} x^{i \gamma}$,
$\bm{x}_{+} = x_{+}^{i \gamma} = E_{+}^{i \top} x^{i \gamma}$ (see \eqref{eq:Epm}).
The states $\bm{x}_{-}(z, t) \in \mathbb{R}^{n_{-}^{i}}$ are related to the PIDEs describing the transport in negative $z$--direction and $\bm{x}_{+}(z, t) \in \mathbb{R}^{n_{+}^{i}}$ relates to the propagation in the other direction.
All assumptions posed on the uncertain parameters $\bar{\bm{\Lambda}}$, $\bar{\bm{A}}$ and $\bar{\bm{F}}$ are also fulfilled for the nominal values $\Lambda^{i}$, $A^{i}$ and $F^{i}$.
In \eqref{eq:agent.bc0} and \eqref{eq:agent.bc1} the matrices $\bar{\bm{Q}}_{0} = Q^{i}_{0} + \Delta \bm{Q}_{0} \in \mathbb{R}^{n_{+}^{i} \times n_{-}^{i}}$ and $\bar{\bm{Q}}_{1} = Q^{i}_{1} + \Delta \bm{Q}_{1} \in \mathbb{R}^{n_{-}^{i} \times n_{+}^{i}}$ describe the boundary coupling between the PIDEs \eqref{eq:agent.pde}.
The control input is $u^{i \gamma}(t) = \bm{u}(t) \in \mathbb{R}^{n_{-}^{i}}$.
Moreover, the agents are influenced by unknown disturbances $d^{i \gamma}(t) = \bm{d}(t) \in \mathbb{R}^{q^{i \gamma}}$, which affect the PIDE subsystem through the unknown disturbance input matrices $\bm{G}(z) \in \mathbb{R}^{n^{i} \times q^{i \gamma}}$ with piecewise continuous elements, $\bm{G}_{0} \in \mathbb{R}^{n_{+}^{i} \times q^{i \gamma}}$ and $\bm{G}_1 \in \mathbb{R}^{n_{-}^{i} \times q^{i \gamma}}$.
The state of the ODE subsystem \eqref{eq:agent.ode} is $w^{i \gamma}(t) = \bm{w}(t) \in \mathbb{R}^{n_{w}^{i}}$ with the IC $\bm{w}(0) = \bm{w}_{0}$. Its dynamics are described by $\bar{\bm{F}}_{w} = F_{w}^{i} + \Delta \bm{F}_{w}$ and $\bar{\bm{B}}_{w} = B_{w}^{i} + \Delta \bm{B}_{w}$ with $(F_{w}^{i}, B_{w}^{i})$ controllable.
Furthermore, the ODEs unknown disturbance input matrix is $\bm{G}_w \in \mathbb{R}^{n_{w}^{i} \times q^{i \gamma}}$.
The ODE state acts on the PIDE subsystem according to $\bar{\bm{C}}(z) = C^{i}(z) + \Delta \bm{C}(z) \in \mathbb{R}^{n^{i} \times n_{w}^{i}}$ with piecewise continuous elements and $\bar{\bm{C}}_{0} = C^{i}_{0} + \Delta \bm{C}_{0} \in \mathbb{R}^{n_{+}^{i} \times n_{w}^{i}}$.
The output $y^{i \gamma}(t) = \bm{y}(t) \in \mathbb{R}^{\bar{n}_{-}}$ to be synchronized with $0 < \bar{n}_{-} \leq \min_{i=1, \ldots, g} n_{-}^{i}$ depends on the PIDE state and the ODE state.
In particular, the formal output operator in \eqref{eq:agent.output} is
\begin{align}\label{eq:agent.output.operator}
	\bar{\bm{\mathcal{C}}}_{x}[h] &= \bar{\bm{C}}_{x,0} h(\bar{\bm{z}}_{0}) + \bar{\bm{C}}_{x,1} h(\bar{\bm{z}}_{1}) + \int_{0}^{1} \bar{\bm{C}}_{x}(\zeta) h(\zeta) \d \zeta
\end{align}
for $h(z) \in \mathbb{C}^{n^{i}}$.
Distributed and pointwise in-domain outputs are taken into account by defining
\begin{align}
	\bar{\bm{C}}_{x}(z) = \bar{\bm{C}}_{x,d}(z) + \sum_{k = 2}^{l^{i}} \bar{\bm{C}}_{x,k} \delta(z - \bar{\bm{z}}_{k})
\end{align}
with $\bar{\bm{C}}_{x,d}(z) = C_{x,d}^{i}(z) + \Delta \bm{C}_{x,d}(z) \in \mathbb{R}^{\bar{n}_{-} \times n^{i}}$ piecewise continuous.
For pointwise evaluation at the boundaries and in-domain $\bar{\bm{C}}_{x,k} = C_{x,k}^{i} + \Delta \bm{C}_{x,k} \in\mathbb{R}^{\bar{n}_{-} \times n^{i}}$ as well as $\bar{\bm{z}}_{k} = z_{k}^{i} + \Delta \bm{z}_{k} \in [0, 1]$, $k=0, \ldots, l^{i}$, are introduced with $z_{0}^{i}=0$, $z_{1}^{i} = 1$ and $z_{k}^{i} \in (0, 1)$, $k = 2, \ldots, l^{i}$.
The ODE state is considered in \eqref{eq:agent.output} by $\bar{\bm{C}}_{w} = C_{w}^{i} + \Delta \bm{C}_{w} \in \mathbb{R}^{\bar{n}_{-} \times n_{w, i}}$.
Finally, the disturbance may act at the output $\bm{y}$ due to the unknown matrix $\bm{G}_y \in \mathbb{R}^{\bar{n}_{-} \times q^{i \gamma}}$.

\begin{remark}
	This general hyperbolic PIDE--ODE setup is motivated by systems of coupled wave equations with dynamic boundary conditions at the non-actuated ends.
	In particular, if a load is carried by multiple heavy ropes (see \cite{Irscheid2019flatness2}), a system of coupled wave equations results, which can be mapped into \eqref{eq:agent} (see, e.g., \cite{Deu21waveOde}). Thereby, disturbances can act distributed on the ropes, at the suspensions, at the payload and at the measurements, which leads to the disturbance inputs in \eqref{eq:agent}\remarkEnd
\end{remark}

\subsubsection{ODE Agents with Input and Output Delays}\label{sec:agentDynamics.delay}
The system \eqref{eq:agent} includes the important setup of uncertain finite-dimensional agents subject to arbitrarily large and uncertain input and output delays.
In particular, the agents $\nu^{i \gamma}$, $i = 1, \ldots, g$, $\gamma = 1, \ldots, N^{i}$, take the form
\begin{subequations}\label{eq:agent.delay.original}
	\begin{align}
		\dot{\bm{w}}(t) &= \bar{\bm{F}}_{w} \bm{w}(t) + \bar{\bm{B}}_{w} \bar{\bm{u}}(t) + \bm{G}_{w} \bm{d}(t)\\
		\bar{\bm{y}}(t) &= \bar{\bm{C}}_{0} \bm{w}(t) + \bm{G}_{0} \bm{d}(t) \label{eq:agent.delay.original.output.ode}
	\end{align}
	with the delayed inputs and outputs
	\begin{align}
		\bar{\bm{u}}(t) &= \col (\bm{u}_1(t - \bar{\bm{D}}_1), \ldots, \bm{u}_{n_{-}^{i}\!}(t - \bar{\bm{D}}_{n_{-}^{i}}) )\\
		\bm{y}(t) &= \col (\bar{\bm{y}}_1(t - \bar{\bm{D}}_{y,1}), \ldots, \bar{\bm{y}}_{n_{+}^{i}\!}(t - \bar{\bm{D}}_{y,n_{+}^{i}}) ) \label{eq:agent.delay.original.output}
	\end{align}
\end{subequations}
in which $\bar{\bm{D}}_{1} \leq \bar{\bm{D}}_{2} \leq \ldots \leq \bar{\bm{D}}_{n_{-}^{i}}$ are the uncertain input delays and $\bar{\bm{D}}_{y,1} \geq \bar{\bm{D}}_{y,2} \geq \ldots \geq \bar{\bm{D}}_{y,n_{+}^{i}}$ represent the uncertain output delays.
This ordering of the delays means no loss of generality, since it can always be achieved by a suitable numbering of the inputs and outputs (see \cite[Rem. 2]{DeuGab20delay}).
By describing the delays as the solution of transport PDEs, the system \eqref{eq:agent.delay.original} results in the hyperbolic PDE--ODE system
\begin{subequations}\label{eq:agent.delay}
	\begin{align}
		\partial_t \bm{x}(z, t) &= \bar{\bm{\Lambda}}(z) \partial_z \bm{x}(z, t)\label{eq:agent.delay.pde}\\
		\bm{x}_{+}(0, t) &= \bar{\bm{C}}_{0} \bm{w}(t) + \bm{G}_{0} \bm{d}(t) \label{eq:agent.delay.bc0}\\
		\bm{x}_{-}(1, t) &= \bm{u}(t)\label{eq:agent.delay.bc1}\\
		\dot{\bm{w}}(t) &= \bar{\bm{F}}_{w} \bm{w}(t) + \bar{\bm{B}}_{w} \bm{x}_{-}(0, t) + \bm{G}_{w} \bm{d}(t)\label{eq:agent.delay.ode}\\
		\bm{y}(t) &= \bm{x}_{+}(1, t)\label{eq:agent.delay.output}
	\end{align}
\end{subequations}
with
$\bar{\bm{\Lambda}} = \diag(\bar{\bm{D}}_1, \ldots, \bar{\bm{D}}_{n_{-}^{i}}, -\bar{\bm{D}}_{y,1}, \ldots, -\bar{\bm{D}}_{y,n_{+}^{i}})^{-1}$.
More precisely, the PDE--ODE--PDE cascade
\begin{center}
	\begin{tikzpicture}[every path/.style={draw, thick, line  join=bevel}, node distance = 9mm and 9mm]
		\myTikzSet
		\blockTextSize
		%
		%
		\node(uIn) at (0, 0){};
		\node(PDE_in)[frame, right=of uIn, align=center]{PDE\subText{\eqref{eq:agent.delay.pde}, \eqref{eq:agent.delay.bc1}}};
		\node(ODE)[frame, right=10mm of PDE_in, align=center]{ODE\subText{\eqref{eq:agent.delay.ode}}};
		\node(PDE_out)[frame, right=10mm of ODE, align=center]{PDE\subText{\eqref{eq:agent.delay.pde}, \eqref{eq:agent.delay.bc0}}};
		\node(PDE_in.title)[above=0mm of PDE_in, align=center]{input delay};
		\node(ODE.title)[above=0mm of ODE, align=center]{lumped agent};
		\node(PDE_out.title)[above=0mm of PDE_out, align=center]{output delay};
		\draw[->] (uIn) to node[above=\pfeilTextAbstand, align=center](arrow1){$\bm{u}(t)~~$} (PDE_in.west);
		\draw[->] (PDE_in.east) to node[above=\pfeilTextAbstand, align=center](pide2ode){$\bm{\bar{u}}(t)$} (ODE.west);
		\draw[->] (ODE.east) to node[above=\pfeilTextAbstand, align=center](ode2pide){$\bm{\bar{y}}(t)$} (PDE_out.west);
		\node(yOut) [coordinate, right=of PDE_out]{};
		\draw[->] (PDE_out) to node[above=\pfeilTextAbstand, align=center](arrow1){$\bm{y}(t)$} (yOut);
	\end{tikzpicture}
\end{center}
is obtained.

\subsubsection{Robust Cooperative Output Regulation Problem}{\acceptancenotice\copyrightnotice}
The reference input $r(t) \in \mathbb{R}^{\bar{n}_{-}}$ to be tracked by all agents is the solution of the \textit{global reference model}
\begin{subequations}\label{eq:signalModel.leader}
	\begin{align}
		\dot{v}_{r}(t) &= S_{r} v_{r}(t), && t>0, &&v_{r}(0) = v_{r,0} \in \mathbb{R}^{n_{r}}\\
		r(t) &= P_{r} v_{r}(t),	&& t \geq 0,&&
	\end{align}
\end{subequations}
which is called the \textit{leader} and is treated as agent $\nu^{0}$.
It is assumed that $(P_{r}, S_{r})$ is observable and all eigenvalues in the \textit{spectrum} $\sigma(S_{r})$ lie on the imaginary axis $\text{j} \mathbb{R}$.
This allows to model trigonometric and polynomial reference inputs, for instance, linear combinations of constants, ramps and sinusoidals.
Similarly, the disturbances $d^{i \gamma}$ to be rejected are modeled by the \textit{local disturbance models}
\begin{subequations}\label{eq:signalModel.disturbance}
	\begin{align}
		\dot{v}_{d}^{i \gamma}(t) &= S_{d}^{i \gamma} v_{d}^{i \gamma}(t), && t>0, &&v_{d}^{i \gamma}(0) = v_{d,0}^{i \gamma} \in \mathbb{R}^{n^{i \gamma}_{d}}\\
		d^{i \gamma}(t) &= P_{d}^{i \gamma} v_{d}^{i \gamma}(t),	&& t \geq 0&&
	\end{align}
\end{subequations}
for all $i=1, \ldots, g$, $\gamma=1, \ldots, N^{i}$, with $(P_{d}^{i \gamma}, S_{d}^{i \gamma})$ observable and $\sigma(S_{d}^{i \gamma}) \subset \text{j} \mathbb{R}$. Both signal models \eqref{eq:signalModel.leader} and \eqref{eq:signalModel.disturbance} are merged into the \textit{joint signal model}
\begin{subequations}\label{eq:signalModel.joint}
	\begin{align}
		\dot{v}(t) &= S v(t), && t>0, &&v(0) = v_{0} \in \mathbb{R}^{n_{v}}\\
		r(t) &= \tilde{P}_{r} v(t),	&& t \geq 0&&\\
		d^{i \gamma\!}(t) &= \tilde{P}^{i \gamma}_{d} v(t),	&& t \geq 0&&
	\end{align}
\end{subequations}
such that the same signals can be generated.
Without loss of generality it is assumed that $S \in\mathbb{R}^{n_{v} \times n_{v}}$ contains only the \textit{cyclic part} of $\diag(S_{r}, S_{d}^{1 1}, \ldots, S_{d}^{g N^{g}})$.
This means that only the largest Jordan block of every eigenvalue is considered (see \cite{Fr75, DeuKer19}).
Thus, $S$ fulfills $\sigma(S) \subset \text{j} \mathbb{R}$ and is known, while knowledge of the matrices $\tilde{P}_{r}$, $\tilde{P}_{d}^{i \gamma}$ is not required for the controller design.

For solving the \textit{robust cooperative output regulation problem}, a networked controller must be designed which respects the restricted communication topology described by $\mathcal{G}$ and ensures \textit{reference tracking}
\begin{align}\label{eq:trackingProblem}
	\lim_{t\to\infty} e_{y}^{i \gamma}(t) = \lim_{t\to\infty} (y^{i \gamma}(t) - r(t)) = 0
\end{align}
for all $i = 1, \ldots, g$, $\gamma = 1, \ldots, N^{i}$, for any ICs of the uncertain agents \eqref{eq:agent}, of the joint signal model \eqref{eq:signalModel.joint} and of the networked controller.
%
%
%
\section{Cooperative State Feedback Regulator}\label{sec:controller.design}{\acceptancenotice\copyrightnotice}
The controller design for robust cooperative output regulation is based on the so-called \textit{cooperative internal model principle}, see \cite{SuCoopImp2013} for MAS with lumped agent dynamics and \cite{Deu22robustMas} for parabolic PDE agents.
For every element of the output $y^{i \gamma}(t) \in \mathbb{R}^{\bar{n}_{-}}$ one copy of the signal model \eqref{eq:signalModel.joint}, i.e., a total of $\bar{n}_{-}$-copies, has to be incorporated in the controller of each follower.
Thus, with $S$ cyclic, $\tilde{S} = I_{\bar{n}_{-}} \Otimes S$, $\tilde{B}_{y} = I_{\bar{n}_{-}} \Otimes b_{y}$ one can choose $b_{y} \in \mathbb{R}^{n_{v}}$ such that $(S, b_{y})$ is controllable.
These internal models are driven by a diffusive coupling of the outputs of neighboring agents resulting in the \textit{cooperative internal model} for the $\gamma$-th agent $\nu^{i \gamma}$ in the subset $\mathcal{V}^{i}$
\begin{subequations}\label{eq:controller}
	\begin{align}\label{eq:controller.dynamics}
		\dot{\bar{v}}^{i \gamma}(t) &= \tilde{S} \bar{v}^{i \gamma}(t) + \tilde{B}_{y} \sum_{j=1}^{i} \sum_{\delta=1}^{N^{j}} a_{\gamma \delta}^{i j} (y^{i \gamma}(t) - y^{j \delta}(t))
		\nonumber\\*&\quad
		+ \tilde{B}_{y} a_{\gamma}^{i 0} (y^{i \gamma}(t) - r(t)), \quad t>0 
	\end{align}
	with the IC $\bar{v}^{i \gamma}(0) = \bar{v}_{0}^{i \gamma} \in \mathbb{R}^{n_{\bar{v}}}$, $n_{\bar{v}} = \bar{n}_{-} n_{v}$ and the elements $a_{\gamma \delta}^{i j}$, $a_{\gamma}^{i 0}$ of $A_{\mathcal{G}}$ (see \eqref{eq:Ag.frobenius}).
	Note that in \eqref{eq:controller.dynamics} the agent $\nu^{i \gamma}$ receives information from its neighbors $\nu^{i \delta}$ in the same group $\mathcal{V}^{i}$, as well as from neighbors $\nu^{j \delta}$ of other groups $\mathcal{V}^{j}$, $j\neq i$, ranging from $j=1, \ldots, i$ due to the triangular structure of $A_{\mathcal{G}}$ in \eqref{eq:Ag.frobenius}.
	Similar to the classical internal model principle, robust cooperative output regulation is achieved by stabilizing the closed-loop consisting of the agents augmented by the internal model, which is shown in Section \ref{sec:robustCooperativeOutputRegulation}.
	For this, however, only a communication through the network with the restricted topology described by $\mathcal{G}$ is allowed. This gives rise to the state feedback
	\begin{align}
		u^{i \gamma}(t) &= u_{l}^{i \gamma}(t) + u_{c}^{i \gamma}(t) = \mathcal{K}[\bar{v}^{i \gamma}(t), x(t), w(t)] \label{eq:controller.output}
	\end{align}
\end{subequations}
with the formal feedback operator $\mathcal{K}$, $x = \col(x^{1 1}, x^{1 2}, \ldots, x^{g N^{g}})$ and $w = \col(w^{1 1}, w^{1 2}, \ldots, w^{g N^{g}})$.
By combining the cooperative internal model \eqref{eq:controller.dynamics} with the state feedback \eqref{eq:controller.output}, the \textit{cooperative state feedback regulator} \eqref{eq:controller} results.
Therein, the state feedback \eqref{eq:controller.output} consists of the \textit{local state feedback}
\begin{subequations}\label{eq:controller.u}
	\begin{align}\label{eq:controller.u.local}
		u_{l}^{i \gamma}(t) &= K_{\bar{v}}^{i} \bar{v}^{i \gamma}(t) - \int_{0}^{1}\! K_{l,x}^{i}\!(\zeta) x^{i \gamma}(\zeta, t) \d\zeta - K_{l,1}^{i} x_{+}^{i \gamma}(1, t) \nonumber\\*&\quad- K_{l,w}^{i} w^{i \gamma}(t),
	\end{align}
	which only depends on the states of the individual agents, i.e., local information, and the \textit{cooperative state feedback}
	\begin{align}
		&u_{c}^{i \gamma}(t) \!=\! K_{\bar{v}}^{i} \!\sum_{j=1}^{i} \sum_{\delta=1}^{N^{j}} \!\big( 
			a^{i j}_{\gamma \delta} (\bar{u}^{i \gamma}(t) \!-\! \bar{u}^{j \delta}(t)) + a^{i 0}_{\gamma} \bar{u}^{i \gamma}(t)
			\big)
		\label{eq:cooperative.cooperative.v2.sum}\\
		&\hspace*{-2pt}\bar{u}_{c}^{i \gamma}(t) \!=\! K_{c,w}^{i} w^{i \gamma}(t) + \int_{0}^{1} K_{c,x}^{i}(\zeta) x^{i \gamma}(\zeta, t) \d \zeta \label{eq:cooperative.cooperative.v2.u}
	\end{align}
\end{subequations}
that uses information provided by neighboring agents, i.e., $\bar{u}^{j \delta}(t)$.
The feedback gains $K_{\bar{v}}^{i} \in \mathbb{R}^{n_{-}^{i} \times n_{\bar{v}}}$, $K_{l,x}^{i}(z) \in \mathbb{R}^{n_{-}^{i} \times n^{i}}$, $K_{c,x}^{i}(z) \in \mathbb{R}^{n_{\bar{v}} \times n^{i}}$, $K_{l,1}^{i} \in \mathbb{R}^{n_{-}^{i} \times n_{+}^{i}}$ and $K_{l,w}^{i} \in \mathbb{R}^{n_{-}^{i} \times n_{w}^{i}}$, $K_{c,w}^{i} \in \mathbb{R}^{n_{\bar{v}} \times n_{w}^{i}}$ will be specified in the following design and they determine the feedback operator $\mathcal{K}$ in \eqref{eq:controller.output}.
Note that the controller \eqref{eq:controller} complies with the communication constraints specified by $\mathcal{G}$, because the elements $a^{i j}_{\gamma \delta}$ and $a^{i 0}_{\gamma}$ of the adjacency matrix $A_{\mathcal{G}}$ in \eqref{eq:controller.dynamics} and \eqref{eq:cooperative.cooperative.v2.sum} take the network topology into account.
\begin{remark}
	Fig. \ref{pic:cooperative.control} shows the local realization of the cooperative state feedback regulator for one agent $\nu^{i \gamma}$.
	It can be seen that each follower $\nu^{i \gamma}$ has to implement \eqref{eq:controller} and \eqref{eq:controller.u} locally.
	Although the agents have a distributed state, only the lumped signals $\bar{u}_{c}^{j \delta}(t) \in \mathbb{R}^{n_{\bar{v}}}$, $y^{j \delta}(t) \in \mathbb{R}^{\bar{n}_{-}}$ from neighbors $\nu^{j \delta}$ must be transmitted through the network.
	Hence, only a limited amount of data needs to be communicated.\remarkEnd
\end{remark}
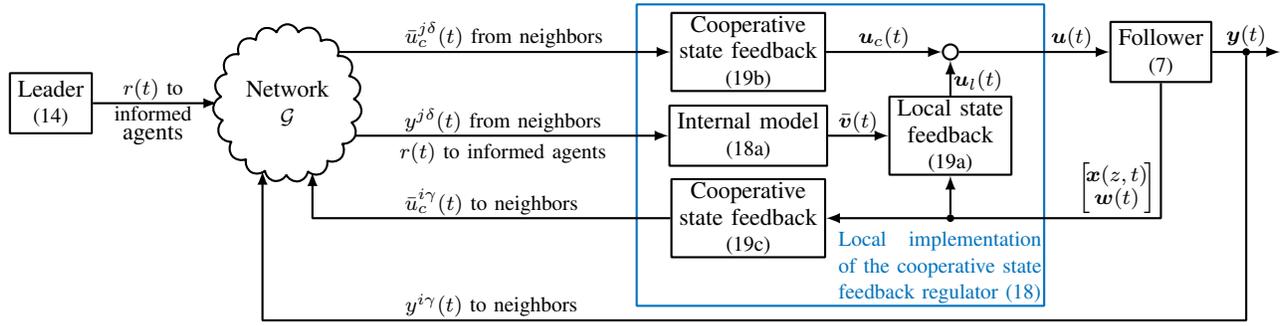
\begin{figure*}\centering
	\begin{tikzpicture}[every path/.style={draw, thick, line  join=bevel, align=center}, node distance = 4mm and 8mm]
		\myTikzSet
		\blockTextSize
		%
		%
		\node(network)[minimum width=13mm, minimum height=22mm] at (0, 0){};
		\begin{pgfonlayer}{front}
			\node[cloud, cloud puffs=19.1, cloud ignores aspect, minimum width=15mm, minimum height=21mm, align=center, fill=white, draw] (cloud) at (0, -1.25mm) {Network\subText{$\mathcal{G}$}};
		\end{pgfonlayer}
		%
		\node(IM)[frame, right=44mm of network.east south east]{Internal model\subText{\eqref{eq:controller.dynamics}}};
		\node(localFeedback)[frame, right= of IM]{Local state\\feedback\subText{\eqref{eq:controller.u.local}}};
		\draw[->] (network.east south east) to %
			node[above=\pfeilTextAbstand] {\footnotesize$y^{j \delta}(t)$ from neighbors}%
			node[below=\pfeilTextAbstand+0.5mm] {\hspace{0em}\footnotesize$r(t)$ to informed agents}%
		(IM);
		\draw[->] (IM) to node[above=\pfeilTextAbstand](v){\footnotesize$\bar{\bm{v}}(t)$} (localFeedback);
		%
		\node(sumFeedback)[circle, minimum size=2mm, inner sep = 0pt, draw=black] at (network.east north east -| localFeedback){};
		\node(coopFeedbackSum)[frame] at (IM |- network.east north east)
			{Cooperative\\ state feedback \subText{\eqref{eq:cooperative.cooperative.v2.sum}}};
		\draw[->] (network.east north east) --%
			node[above=\pfeilTextAbstand] {\footnotesize$\bar{u}_{c}^{j \delta}(t)$ from neighbors}%
			(coopFeedbackSum);
		\draw[->] (coopFeedbackSum) --
			node[above=\pfeilTextAbstand](uc){\footnotesize$\bm{u}_{c}(t)$}
			(sumFeedback);
		\draw[->] (localFeedback) to node[right=\pfeilTextAbstand, very near start](ul){\footnotesize$\bm{u}_{l}(t)$\\~} (sumFeedback);
		%
		\node(cooperativeFeedback)[frame, below=1.5mm of IM.south, anchor=north] 
		{Cooperative\\ state feedback \subText{\eqref{eq:cooperative.cooperative.v2.u}}};
		%
		\node(agent)[frame, right=2 of sumFeedback]{Follower\subText{\eqref{eq:agent}}};
		\node(knotY)[circle, minimum size=1mm, inner sep = 0pt, line width=0pt, fill=black, right=4mm of agent] {};
		\node(yOut)[coordinate,right=4mm of knotY]{};
		\draw[->] (agent) to node[above=\pfeilTextAbstand+1pt](y){\footnotesize$\bm{y}(t)$} (yOut);
		\node(knotState)[circle, minimum size=1mm, inner sep = 0pt, line width=0pt, fill=black] at (localFeedback |- cooperativeFeedback) {};
		\draw[->] (sumFeedback) to node[above=\pfeilTextAbstand, near end](u){\footnotesize$\bm{u}(t)$} (agent);
		\draw[->] (agent) |-  node[above left=\pfeilTextAbstand](states)
			{\footnotesize$\left[\!\begin{matrix}\bm{x}(z,t)\\\bm{w}(t)\end{matrix}\!\right]$}
			(cooperativeFeedback);
		\draw[->] (knotState) -- (localFeedback);
		%
		\node (sendCooperativeFeedbackHelper) [coordinate] at (network.south south east |- cooperativeFeedback){};
		\draw[->] (cooperativeFeedback) --
			node[above=\pfeilTextAbstand] (sendCooperativeFeedback) {\footnotesize$\bar{u}_{c}^{i \gamma}(t)$ to neighbors}%
			(sendCooperativeFeedbackHelper) -- (network.south south east)--++(0,2pt);
		\node(yHelper)[coordinate, below=19mm of network.south south west]{};
		\draw[->] (knotY) |- 
			node[above=\pfeilTextAbstand](yphantom){\phantom{$(y)$}}%
			(yHelper) -- (network.south south west)--++(0,2.5pt);
		\node (yLabel2) at (yphantom -| sendCooperativeFeedback) {\footnotesize$y^{i \gamma}(t)$ to neighbors};
		%
		\begin{pgfonlayer}{back}
			\node(regulatorFrame)[frame, matlabBlue, fit=(coopFeedbackSum.center)(localFeedback)(cooperativeFeedback)(IM), inner xsep=12pt, inner ysep=18pt]{};
		\end{pgfonlayer}
		\node(regulatorName)[above left=-1mm of regulatorFrame.south east, align=justify, matlabBlue, text width = 2.7cm]{\footnotesize Local implementation of the cooperative state feedback regulator \eqref{eq:controller}};
		\node(leader)[frame, left=16mm of cloud]{Leader\subText{\eqref{eq:signalModel.leader}}};
		\draw[->] (leader) -- 
			node[above=\pfeilTextAbstand](r){\footnotesize$r(t)$ to}
			node[below=\pfeilTextAbstand](r){\footnotesize informed\\[-1mm] agents}
			 (cloud.west)--++(1.5pt,0);
	\end{tikzpicture}
	\caption{\unboldmath Structure of the cooperative state feedback regulator \eqref{eq:controller} for the $\gamma$-th agent $\nu^{i \gamma}$ from the subset $\mathcal{V}^{i}$.
		Note that $y^{i \gamma}(t)$ is only used in the diffusive coupling in \eqref{eq:controller.dynamics} giving rise to a relative measurement.}
	\label{pic:cooperative.control}
\end{figure*}

The cooperative state feedback regulator is designed to stabilize the known nominal agent dynamics.
Nevertheless, reference tracking \eqref{eq:trackingProblem} is still achieved in the presence of non-destabilizing model uncertainties due to the robustness provided by the cooperative internal model principle (see Section \ref{sec:robustCooperativeOutputRegulation}).
The solution of this stabilization problem consists of two steps. Firstly, each agent is stabilized locally. Then, the remaining unstable dynamics of the internal models are stabilized simultaneously in the second step.
This is shown in the next subsections.

\subsection{Local Stabilization of the Nominal Agents}\label{sec:controller.design.localStabilization}%
For the design of the locally stabilizing part of the state feedback \eqref{eq:controller.u.local}, the known nominal agents $\nu^{i \gamma}$, $i = 1, \ldots, g$, $\gamma = 1, \ldots, N^{i}$,
\begin{center}
	\begin{tikzpicture}[every path/.style={draw, thick, line  join=bevel}, node distance = 10mm and 10mm]
		\myTikzSet
		\blockTextSize
		%
		%
		\node(uIn) at (0, 0){};
		\node(PIDE)[frame, right=of uIn, align=center]{PIDE\subText{\eqref{eq:agent.local.nominal.pde}--\eqref{eq:agent.local.nominal.bc1}}};
		\node(ODE)[frame, right=15mm of PIDE, align=center]{ODE\subText{\eqref{eq:agent.local.nominal.ode}}};
		\draw[->] (uIn) to node[above=\pfeilTextAbstand, align=center, near start](arrow1){$\bm{u}(t)$} (PIDE.west);
		\draw[->] (PIDE.east north east) to node[above=\pfeilTextAbstand, align=center](pide2ode){$\bm{x}_{-}(0, t)$} (ODE.west north west);
		\draw[->] (ODE.west south west) to node[below=\pfeilTextAbstand, align=center](ode2pide){$\bm{w}(t)$} (PIDE.east south east);
		\begin{pgfonlayer}{back}
			\node(agent)[frame, \frameColor, fit=(PIDE)(ODE)(pide2ode)(ode2pide), inner xsep=8pt, inner ysep=-1pt]{};
		\end{pgfonlayer}
		\node(yOut) [coordinate, right=of agent.east]{};
		\draw[->] (agent.east) to node[above=\pfeilTextAbstand, align=center](arrow1){$\bm{y}(t)$} (yOut);
	\end{tikzpicture}
\end{center}
described by
\begin{subequations}\label{eq:agent.local.nominal}
	\begin{align}
		\partial_t \bm{x}(z, t) &= \Lambda^{i}(z) \partial_z \bm{x}(z, t) + \mathcal{A}^{i}[\bm{x}(t), \bm{w}(t)](z) \label{eq:agent.local.nominal.pde}\\
		\bm{x}_{+}(0, t) &= Q^{i}_{0} \bm{x}_{-}(0, t) + C^{i}_{0} \bm{w}(t)\label{eq:agent.local.nominal.bc0}\\
		\bm{x}_{-}(1, t) &= Q^{i}_{1} \bm{x}_{+}(1, t) + \bm{u}(t)\label{eq:agent.local.nominal.bc1}\\
		\dot{\bm{w}}(t) &= F_{w}^{i} \bm{w}(t) + B_{w}^{i} \bm{x}_{-}(0, t)\label{eq:agent.local.nominal.ode}\\
		\bm{y}(t) &= \mathcal{C}^{i}_{x} [\bm{x}(t)] + C_{w}^{i} \bm{w}(t)\label{eq:agent.local.nominal.out}
	\end{align}
\end{subequations}
are considered, which follow from \eqref{eq:agent} without the uncertainties and the disturbance inputs. Note that the disturbances have not to be taken into account to solve this stabilization problem.

\subsubsection{Decoupling into a PIDE--ODE Cascade}\label{sec:controller.design.local.decoupling}{\acceptancenotice\copyrightnotice}%
Firstly, the \textit{decoupling transformation}
\begin{align}\label{eq:transformation.local.decoupling}
	\varepsilon^{i \gamma}(z, t) = x^{i \gamma}(z, t) - \Sigma^{i}(z) w^{i \gamma}(t)
\end{align}
is used to map \eqref{eq:agent.local.nominal} into the intermediate target system
\begin{subequations}\label{eq:agent.local.decoupled}
	\begin{align}
		\partial_t \bm{\varepsilon}(z, t) &= \Lambda^{i}(z) \partial_z \bm{\varepsilon}(z, t)  + \mathcal{A}^{i}[\bm{\varepsilon}(t), 0](z) \nonumber\\*&\quad - \Sigma^{i}(z) B_{w}^{i} \bm{\varepsilon}_{-}(0, t) \label{eq:agent.local.decoupled.pde}\\
		\bm{\varepsilon}_{+}(0, t) &= Q^{i}_{0} \bm{\varepsilon}_{-}(0, t)\label{eq:agent.local.decoupled.bc0}\\
		\bm{\varepsilon}_{-}(1, t) &= Q^{i}_{1} \bm{\varepsilon}_{+}(1, t) + \bm{u}_{1}(t)\label{eq:agent.local.decoupled.bc1}\\
		\dot{\bm{w}}(t) &= \tilde{F}_{w}^{i} \bm{w}(t) + B_{w}^{i} \bm{\varepsilon}_{-}(0, t)\label{eq:agent.local.decoupled.ode}\\
		\bm{y}(t) &= \mathcal{C}^{i}_{x} [\bm{\varepsilon}(t)] + \tilde{C}_{w}^{i} \bm{w}(t)\label{eq:agent.local.decoupled.output}
	\end{align}
\end{subequations}
with $\tilde{C}_{w}^{i} = \mathcal{C}^{i}_{x} [\Sigma^{i}] + C_{w}^{i}$.
Note that for each agent $\nu^{i \gamma}$, $\gamma = 1, \ldots, N^{i}$, in the group $\mathcal{V}^{i}$ the same decoupling matrix $\Sigma^{i}(z)$ can be used, because they coincide.
The PIDE subsystem \eqref{eq:agent.local.decoupled.pde}--\eqref{eq:agent.local.decoupled.bc1} is decoupled from the ODE \eqref{eq:agent.local.decoupled.ode}, hence a PIDE--ODE cascade
\begin{center}
	\begin{tikzpicture}[every path/.style={draw, thick, line  join=bevel}, node distance = 10mm and 10mm]
		\myTikzSet
		\blockTextSize
		%
		%
		\node(uIn) at (0, 0){};
		\node(PIDE)[frame, right=of uIn, align=center]{PIDE\subText{\eqref{eq:agent.local.decoupled.pde}--\eqref{eq:agent.local.decoupled.bc1}}};
		\node(ODE)[frame, \stableColor, right=15mm of PIDE, align=center]{ODE\subText{\eqref{eq:agent.local.decoupled.ode}}};
		\draw[->] (uIn) to node[above=\pfeilTextAbstand, align=center, near start](arrow1){$\bm{u}_{1}(t)$} (PIDE.west);
		\draw[->] (PIDE.east) to node[above=\pfeilTextAbstand, align=center](pide2ode){$\bm{\varepsilon}_{-}(0, t)$} (ODE.west);
		\begin{pgfonlayer}{back}
			\node(agent)[frame, \frameColor, fit=(PIDE)(ODE), inner xsep=8pt, inner ysep=3pt]{};
		\end{pgfonlayer}
		\node(yOut) [coordinate, right=of agent.east]{};
		\draw[->] (agent.east) to node[above=\pfeilTextAbstand, align=center](arrow1){$\bm{y}(t)$} (yOut);
	\end{tikzpicture}
\end{center}
is obtained.
To stabilize the ODE subsystem \eqref{eq:agent.local.decoupled.ode}, $K_{w}^{i}$ is chosen such that
$\tilde{F}_{w}^{i} = F_{w}^{i} - B_{w}^{i} K_{w}^{i}$
is Hurwitz.
Since $(F_{w}^{i}, B_{w}^{i})$ is controllable by assumption, this can always be achieved, for instance, by an eigenvalue assignment.

To map \eqref{eq:agent.local.nominal} into \eqref{eq:agent.local.decoupled}, the control law
\begin{align}\label{eq:feedback.local.decoupling}
	u^{i \gamma}(t) &= u^{i \gamma}_{1}(t) + ( E_{-}^{i \top} \!-\! Q^{i}_{1} E_{+}^{i \top}) \Sigma^{i}(1) w^{i \gamma}(t)
\end{align}
has to be applied and $\Sigma^{i}(z)$ in \eqref{eq:transformation.local.decoupling} has to satisfy the \textit{decoupling equations} given in the lemma below, which follows from similar calculations as in \cite{Deu21waveOde}.
However, different to \cite{Deu21waveOde}, a system of Volterra integro-differential equations has to be solved.
The next lemma asserts their solvability.
\begin{lemma}[Local Decoupling Equations]\label{lem:decouplingEquations.local}
	The \textit{decoupling equations}
	\begin{subequations}\label{eq:decouplingEquations.boundary}
		\begin{align}
			&\Lambda^{i}(z) \d_z \Sigma^{i}(z)  + A^{i}(z) \Sigma^{i}(z) - \Sigma^{i}(z) \tilde{F}_{w}^{i}\label{eq:decouplingEquations.boundary.volterraIDE}\\*
			&\quad+ \int_0^z F^{i}(z, \zeta) \Sigma^{i}(\zeta) \d \zeta = A_0^{i}(z) K_{w}^{i} - C^{i}(z), \quad z\in(0, 1]\nonumber \\
			&\Sigma^{i}(0) = - E_{-}^{i} K_{w}^{i} + E_{+}^{i} (C_{0}^{i} - Q_{0}^{i} K_{w}^{i})\label{{eq:decouplingEquations.boundary.ic}}
		\end{align}
	\end{subequations}
	have a unique continuous solution $\Sigma^{i}(z) \in \mathbb{R}^{n^{i} \times n_{w}^{i}}$, $i=1, \ldots, g$.
\end{lemma}

\begin{proof}
	By premultiplying \eqref{eq:decouplingEquations.boundary.volterraIDE} with $(\Lambda^{i}(z))^{-1}$, applying the vec operator to \eqref{eq:decouplingEquations.boundary} and using $\text{vec}(A B C) = (C^{\top} \otimes A) \text{vec}(B)$ (see \cite[Prop. 7.1.9]{Bern05}), the simpler \textit{Volterra integro-differential equation}
	\begin{align}
		\d_z \sigma^{i}(z) &= \Theta_{1}^{i}(z) \sigma^{i}(z) \!+\! \int_{0}^{z} \Theta_{2}^{i}(z, \zeta) \sigma^{i}(\zeta) \d\zeta \!+\! \theta_{3}^{i}(z) \label{eq:decouplingEquations.boundary.volterraIDE2}
	\end{align}
	results with $\sigma^{i}(z) = \text{vec}(\Sigma^{i}(z))$, $\sigma^{i}(0) = \sigma^{i}_{0}$ and
	\begin{subequations}
		\begin{align}
			\Theta_{1}^{i}(z) &= \tilde{F}_{w}^{i \top} \Otimes (\Lambda^{i}(z))^{-1} \!-\! I_{n_{w}^{i}} \Otimes ((\Lambda^{i}(z))^{-1} A^{i}(z))\!\\
			\Theta_{2}^{i}(z, \zeta) &= -I_{n_{w}^{i}} \otimes ((\Lambda^{i}(z))^{-1} F^{i}(z, \zeta))\\
			\theta_{3}^{i}(z) &= \text{vec}( (\Lambda^{i}(z))^{-1} (A_0^{i}(z) K_{w}^{i} - C^{i}(z)))\\
			\sigma^{i}_{0} &= \text{vec}(- E_{-}^{i} K_{w}^{i} + E_{+}^{i} (C_{0}^{i} - Q_{0}^{i} K_{w}^{i})).
		\end{align}
	\end{subequations}
	Then, a formal integration of \eqref{eq:decouplingEquations.boundary.volterraIDE2} yields
	\begin{align}
		\sigma^{i}\!(z) \!=\!\sigma^{i}_{0} \!+\! \int_{0}^{z}\!\big(\Theta_{1}^{i}(\zeta) \sigma^{i}(\zeta) \!+\! \int_{0}^{\zeta}\! \Theta_{2}^{i}(\zeta, \bar{\zeta}) \sigma^{i}(\bar{\zeta}) \d\bar{\zeta} \!+\! \theta_{3}^{i}(\zeta) \!\big) \d\zeta\!
	\end{align}
	and after changing the order of the integrals, the \textit{Volterra integral equation of the second kind}
	\begin{align}\label{eq:decouplingEquations.boundary.volterra}
		\sigma^{i}(z) &= \theta_{4}^{i}(z) + \int_{0}^{z} \Theta_{5}^{i}(z, \zeta) \sigma^{i}(\zeta) \d\zeta
	\end{align}
	for $\sigma^{i}(z)$ follows with $\theta_{4}^{i}(z) = \sigma^{i}_{0} + \int_{0}^{z} \theta_{3}^{i}(\zeta) \d\zeta$ and the kernel $\Theta_{5}^{i}(z, \zeta) = \Theta_{1}^{i}(\zeta) + \int_{\zeta}^{z} \Theta_{2}^{i}(\bar{\zeta}, \zeta) \d\bar{\zeta}$. As $\theta_{4}^{i}$, $\Theta_{5}^{i}$ have continuous elements, \eqref{eq:decouplingEquations.boundary.volterra} and equivalently \eqref{eq:decouplingEquations.boundary} have a unique continuous solution, which can be shown by using successive approximations (see \cite[Th. 2.1.1]{Burton2005volterra}).
\end{proof}

\subsubsection{Backstepping Transformation}\label{sec:localStabilization.backstepping}
In order to stabilize the PIDE subsystem \eqref{eq:agent.local.decoupled.pde}--\eqref{eq:agent.local.decoupled.bc1}, the invertible \textit{backstepping transformation}
\begin{align}\label{eq:transformation.backstepping}
	\tilde{\varepsilon}^{i \gamma}(z, t) &= \varepsilon^{i \gamma}(z, t) - \int_{0}^{z} K^{i}(z, \zeta) \varepsilon^{i \gamma}(\zeta, t) \d\zeta = \mathcal{T}_{i}[\varepsilon^{i \gamma}(t)](z)
\end{align}
is applied to transform \eqref{eq:agent.local.decoupled} into the target system
\begin{subequations}\label{eq:agent.local.backstepping}
	\begin{align}
		\partial_t \tilde{\bm{\varepsilon}}(z, t) &= \Lambda^{i}(z) \partial_z \tilde{\bm{\varepsilon}}(z, t)  + \tilde{A}_{0}^{i}(z) \tilde{\bm{\varepsilon}}_{-}(0, t)\label{eq:agent.local.backstepping.pde}\\
		\tilde{\bm{\varepsilon}}_{+}(0, t) &= Q^{i}_{0} \tilde{\bm{\varepsilon}}_{-}(0, t)\label{eq:agent.local.backstepping.bc0}\\
		\tilde{\bm{\varepsilon}}_{-}(1, t) &= \bm{u}_{2}(t)\label{eq:agent.local.backstepping.bc1}\\
		\dot{\bm{w}}(t) &= \tilde{F}_{w}^{i} \bm{w}(t) + B_{w}^{i} \tilde{\bm{\varepsilon}}_{-}(0, t)\label{eq:agent.local.backstepping.ode}\\
		\bm{y}(t) &= \tilde{\mathcal{C}}^{i}_{x} [\tilde{\bm{\varepsilon}}(t)] + \tilde{C}_{w}^{i} \bm{w}(t)\label{eq:agent.local.backstepping.output}
	\end{align}
\end{subequations}
with $\tilde{A}_{0}^{i}(z) = \col(\tilde{A}_{0-}^{i}(z), \tilde{A}_{0+}^{i}(z))$ and a strictly lower triangular matrix $\tilde{A}_{0-}^{i}(z) \in \mathbb{R}^{n_{-}^{i} \times n_{-}^{i}}$.
Similar to the decoupling transformation \eqref{eq:transformation.local.decoupling}, the same backstepping transformation \eqref{eq:transformation.backstepping} is considered for all agents $\nu^{i \gamma}$, $\gamma = 1, \ldots, N^{i}$, in the group $\mathcal{V}^{i}$.
The PDE subsystem \eqref{eq:agent.local.backstepping.pde}--\eqref{eq:agent.local.backstepping.bc1} is a cascade of scalar transport PDEs so that it becomes finite-time stable with the settling time
$t_{f}^{i} = \sum_{k=1}^{n_{-}^{i} + 1} \int_{0}^{1} {|\lambda_{k}^{i}(\zeta)|}^{-1} \d\zeta$
in the final target system (see \cite{Hu19}). Thus, the PDE--ODE cascade
\begin{center}
	\begin{tikzpicture}[every path/.style={draw, thick, line  join=bevel}, node distance = 10mm and 10mm]
		\myTikzSet
		\blockTextSize
		%
		%
		\node(uIn) at (0, 0){};
		\node(PIDE)[frame, \stableColor, right=of uIn, align=center]{PDE\subText{\eqref{eq:agent.local.backstepping.pde}--\eqref{eq:agent.local.backstepping.bc1}}};
		\node(ODE)[frame, \stableColor, right=15mm of PIDE, align=center]{ODE\subText{\eqref{eq:agent.local.backstepping.ode}}};
		\draw[->] (uIn) to node[above=\pfeilTextAbstand, align=center, near start](arrow1){$\bm{u}_{2}(t)$} (PIDE.west);
		\draw[->] (PIDE.east) to node[above=\pfeilTextAbstand, align=center](pide2ode){$\tilde{\bm{\varepsilon}}_{-}(0, t)$} (ODE.west);
		\begin{pgfonlayer}{back}
			\node(agent)[frame, \stableColor, fit=(PIDE)(ODE), inner xsep=8pt, inner ysep=3pt]{};
		\end{pgfonlayer}
		\node(yOut) [coordinate, right=of agent.east]{};
		\draw[->] (agent.east) to node[above=\pfeilTextAbstand, align=center](arrow1){$\bm{y}(t)$} (yOut);
	\end{tikzpicture}
\end{center}
is asymptotically stable.

In order to determine the backstepping transformation \eqref{eq:transformation.backstepping}, the \textit{kernel} $K^{i}(z, \zeta)\in \mathbb{R}^{n^{i} \times n^{i}}$, $i=1, \ldots, g$, $(z, \zeta) \in \bar{\Omega} = \{(z, \zeta) \, | \, 0 \leq \zeta \leq z \leq 1\}$, must solve the \textit{kernel equations}
\begin{subequations}\label{eq:kernel.backstepping}
	\begin{align}
		&\!\Lambda^{i}(z) \partial_z K^{i}(z, \zeta) + \partial_\zeta (K^{i}(z, \zeta) \Lambda^{i}(\zeta))\label{eq:kernel.backstepping.pde}\\*&\quad= K^{i}(z, \zeta) A^{i}(\zeta) - F^{i}(z, \zeta) + \int_{\zeta}^{z} K^{i}(z, \eta) F^{i}(\eta, \zeta) \d\eta\nonumber\\
		&\!\Lambda^{i}(z) K^{i}(z, z) - K^{i}(z, z) \Lambda^{i}(z) = - A^{i}(z)\label{eq:kernel.backstepping.zEqualZeta}\\
		&\!K^{i}(z, 0) \Lambda^{i}(0) ( E_{-}^{i} \!+\! E_{+}^{i} Q_{0}^{i} ) \!=\! \tilde{A}_{0}^{i}(z) \!-\! \mathcal{T}_{i}[A_{0}^{i} \!-\! \Sigma^{i} B_{w}^{i}](z)\!\label{eq:kernel.backstepping.zetaEqual0}
	\end{align}
\end{subequations}
and the control input
\begin{align}\label{eq:feedback.local.backstepping}
	u^{i \gamma}_{1}(t) &\!=\! u^{i \gamma}_{2}(t) \!-\! Q^{i}_{1} \varepsilon_{+}^{i \gamma}(1, t) \!+\! \int_{0}^{1} \! E_{-}^{i \top} K^{i}(1, \zeta) \varepsilon^{i \gamma}(\zeta, t) \d\zeta\!
\end{align}
has to be applied (see \cite{Hu19, Deu17a}).
The non-zero elements of $\tilde{A}_{0}^{i}(z)$ follow from \eqref{eq:kernel.backstepping.zetaEqual0} and thus are determined by the kernel.
In \cite{Hu19,Hu15a} it is shown that the kernel equations without the integral term in \eqref{eq:kernel.backstepping.pde} have a unique, piecewise differentiable solution provided that the artificial boundary conditions suggested in \cite{Hu19} (see also \cite[Rem. 6]{Hu15a}) are imposed.
Thereby, it is no obstacle to take this integral in \eqref{eq:kernel.backstepping.pde} into account, which can be shown by the approach of \cite{Kr08a}.
The kernel $K^{i}_{I}(z, \zeta)\in\mathbb{R}^{n^{i}\times n^{i}}$ of the \textit{inverse backstepping transformation}
\begin{align}\label{eq:transformation.backstepping.inverse}
	\varepsilon^{i \gamma}(z, t) \!&=\! \tilde{\varepsilon}^{i \gamma}(z, t) \!+\! \int_{0\!}^{z\!} \! K^{i}_{I}(z, \zeta) \tilde{\varepsilon}^{i \gamma}(\zeta, t) \d\zeta \!=\! \mathcal{T}_{i}^{-1}[\tilde{\varepsilon}^{i \gamma}(t)](z)
\end{align}
can be calculated by using the reciprocity relation (see \cite{DeuGehKe19}).
The output \eqref{eq:agent.local.backstepping.output} in the backstepping coordinates follows by inserting \eqref{eq:transformation.backstepping.inverse} into \eqref{eq:agent.local.decoupled.output} to obtain \eqref{eq:agent.local.backstepping.output} with
\begin{subequations}
	\begin{align}
		\tilde{\mathcal{C}}^{i}_{x}[h] &\!=\! C_{x,0}^{i} h(0) + C_{x,1}^{i} h(1) + \int_{0}^{1} \tilde{C}_{x}^{i}(\zeta) h(\zeta) \d \zeta \label{eq:agent.backstepping.output.operator}\\
		\tilde{C}_{x}^{i}(z) &\!=\! C_{x}^{i}(z) \!+\! \int_{z}^{1} \! C_{x}^{i}(\zeta)  K^{i}_{I}(\zeta, z) \d\zeta \!+\! C_{x,1}^{i} K^{i}_{I}(1, z)\!
	\end{align}
\end{subequations}
after changing the order of integrals and considering \eqref{eq:agent.output.operator}.
In view of \eqref{eq:feedback.local.decoupling}, \eqref{eq:feedback.local.backstepping} and the decoupling transformation \eqref{eq:transformation.local.decoupling}, the feedback gains of the local state feedback \eqref{eq:controller.u.local} are
\begin{subequations}\label{eq:controller.gain.local}
	\begin{align}
		K_{l,x}^{i}(z) &= -E_{-}^{i \top} K^{i}(1, z),
		\quad K_{l,1}^{i} = Q_{1}^{i} \\
		K_{l,w}^{i} &= - E_{-}^{i \top} \Sigma^{i}(1) + \int_{0}^{1} \! E_{-}^{i \top} K^{i}(1, \zeta) \Sigma^{i}(\zeta) \d \zeta
		.
	\end{align}
\end{subequations}

\subsection{Simultaneous Stabilization}\label{sec:controller.simultaneousStabilization}
\subsubsection{Cooperative Decoupling of the Agents}\label{sec:controller.simultaneousStabilization.decoupling}{\acceptancenotice\copyrightnotice}
To achieve robust cooperative output regulation, the extended system consisting of the cooperative internal model \eqref{eq:controller.dynamics} and the follower agents has to be stabilized.
For this purpose, the cooperative state feedback regulators \eqref{eq:controller} and the transformed nominal agents \eqref{eq:agent.local.backstepping} are merged separately for each group $\mathcal{V}^{i}$, $i=1, \ldots, g$, (see \eqref{eq:agent.set}).
In particular,
$\bar{v}^{i} = \col(\bar{v}^{i1}, \ldots, \bar{v}^{iN^{i}})$,
$\tilde{\varepsilon}^{i} = \col(\tilde{\varepsilon}^{i 1}, \ldots, \tilde{\varepsilon}^{i N^{i}})$,
$w^{i} = \col(w^{i1}, \ldots, w^{i N^{i}})$,
$y^{i} = \col(y^{i1}, \ldots, y^{i N^{i}})$,
$u_{2}^{i} = \col(u_{2}^{i1}, \ldots, u_{2}^{i N^{i}})$
is considered.
This leads to the $g$ groups of the nominal networked controlled MAS in the target coordinates
\begin{subequations}\label{eq:group}
	\begin{align}
		\dot{\bar{v}}^{i}(t) &= (I_{N^{i}} \otimes \tilde{S}) \bar{v}^{i}(t) + \sum_{j=1}^{i} (H^{i j} \otimes \tilde{B}_{y}) y^{j}(t)\label{eq:group.regulator.dynamics}\\
		\partial_t \tilde{\varepsilon}^{i}(z, t) &= \underline{\Lambda}^{i}(z) \partial_z \tilde{\varepsilon}^{i}(z, t) + \underline{\tilde{A}}_{0}^{i}(z) \tilde{\varepsilon}^{i}_{-}(0, t) \label{eq:group.pde}	\\
		\tilde{\varepsilon}^{i}_{+}(0, t) &= \underline{Q}_{0}^{i} \tilde{\varepsilon}^{i}_{-}(0, t)\label{eq:group.bc0}\\
		\tilde{\varepsilon}^{i}_{-}(1, t) &= u_{2}^{i}(t) \label{eq:group.bc1}\\
		\dot{w}^{i}(t) &= \underline{\tilde{F}}_{w}^{i} w^{i}(t) + \underline{B}_{w}^{i} \tilde{\varepsilon}^{i}_{-}(0, t)\label{eq:group.ode}\\
		y^{i}(t) &= \tilde{\underline{\mathcal{C}}}_{x}^{i} [\tilde{\varepsilon}^{i}(t)] + \underline{\tilde{C}}_{w}^{i} w^{i}(t)\label{eq:group.output}
	\end{align}
\end{subequations}
with the aggregated nominal parameters
$\underline{\Lambda}^{i} = I_{N^{i}} \Otimes \Lambda^{i}$,
$\underline{\tilde{A}}_{0}^{i} = I_{N^{i}} \Otimes \tilde{A}_{0}^{i}$,
$\underline{Q}^{i}_{0} = I_{N^{i}} \Otimes  Q^{i}_{0}$,
$\underline{\tilde{F}}_{w}^{i} = I_{N^{i}} \Otimes \tilde{F}_{w}^{i}$,
$\underline{B}_{w}^{i} = I_{N^{i}} \Otimes B_{w}^{i}$,
$\underline{\tilde{C}}_{w}^{i} = I_{N^{i}} \Otimes \tilde{C}_{w}^{i}$,
$\tilde{\underline{\mathcal{C}}}_{x}^{i} = I_{N^{i}} \Otimes \tilde{\mathcal{C}}_{x}^{i}$
for $i=1, \ldots, g$.
The PDE--ODE--ODE systems \eqref{eq:group} can be viewed as\vspace{-3ex}
\begin{center}
	\begin{tikzpicture}[every path/.style={draw, thick, line  join=bevel}, node distance = 10mm and 10mm]
		\myTikzSet
		\blockTextSize
		%
		%
		\node(uIn) at (0, 0){};
		\node(PIDE)[frame, \stableColor, right=of uIn, align=center]{PDE\subText{\eqref{eq:group.pde}--\eqref{eq:group.bc1}}};
		\node(ODE)[frame, \stableColor, right=15mm of PIDE, align=center]{ODE\subText{\eqref{eq:group.ode}}};
		\draw[->] (uIn) to node[above=\pfeilTextAbstand, align=center, near start](arrow1){$u_{2}^{i}(t)$} (PIDE.west);
		\draw[->] (PIDE.east) to node[above=\pfeilTextAbstand, align=center](pide2ode){$\tilde{\varepsilon}_{-}^{i}(0, t)$} (ODE.west);
		\begin{pgfonlayer}{back}
			\node(agent)[frame, \stableColor, fit=(PIDE)(ODE), inner xsep=8pt, inner ysep=3pt]{};
		\end{pgfonlayer}
		%
		%
		\node(IM)[frame, right=15mm of ODE, align=center]{ODE\subText{\eqref{eq:group.regulator.dynamics}}};
		\draw[->] (agent.east) to node[above=\pfeilTextAbstand, align=center](yArrow){$y^{i}(t)$} (IM);
		\node(yIn) [coordinate, above=8mm of IM.north north west]{};
		\draw[->] (yIn) to node[right=-2mm, near start](yInArrow){{\footnotesize$\begin{array}{l}y^{1}(t)\\[-2pt]{~\tikz{\draw[-,densely dotted] (0,0)--(0,0.185);}}\\[-1mm]y^{i-1}(t)\end{array}$}} (IM.north north west);
		\node(vOut) [coordinate, right=of IM]{};
		\draw[->] (IM) to node[above=\pfeilTextAbstand, align=center](vOut){$\bar{v}^{i}(t)$} (vOut);
	\end{tikzpicture}
\end{center}
in which the cooperative internal model \eqref{eq:group.regulator.dynamics} is driven only from the outputs of the groups $\mathcal{V}^{j}$, $j = 1, \ldots, i$, because of the block triangular structure of the leader-follower matrix $H$, which results from \eqref{eq:Ag.frobenius}.
Note that the reference input $r$ in \eqref{eq:controller.dynamics} does not influence the stabilization and thus is neglected.
The \textit{cooperative decoupling transformation}
\begin{align}\label{eq:transformation.cooperative.decoupling}
	e_{\bar{v}}^{i}(t) &= \bar{v}^{i}(t)
	- \sum_{j=1}^{i} (H^{i j} \Otimes \Pi_{w}^{j}) w^{j}(t)
	\nonumber\\&\quad
	- \sum_{j=1}^{i} \int_{0}^{1} (H^{i j} \Otimes \Pi_{x}^{j}(\zeta)) \tilde{\varepsilon}^{j}(\zeta, t) \d\zeta
\end{align}
and the control input
\begin{align}\label{eq:cooperative.feedback.e}
	u^{i}_{2}(t) &= (I_{N^{i}} \otimes K_{\bar{v}}^{i}) e_{\bar{v}}^{i}(t)
\end{align}
are utilized, in order to decouple and stabilize the internal model \eqref{eq:group.regulator.dynamics}.
Thereby, only the groups $\mathcal{V}^{j}$, $j=1, \ldots, i$, have to be considered for each group $\mathcal{V}^{i}$ due to the triangular coupling with previous groups (see the block diagram above).
This leads to the ODE--PDE--ODE cascade
\begin{subequations}\label{eq:group.decoupled}
	\begin{align}
		\dot{e}_{\bar{v}}^{i}(t) &= \underline{F}_{e}^{i} e_{\bar{v}}^{i}(t) - \sum_{j=1}^{i-1} (H^{i j} \Otimes B_{e}^{j} K_{\bar{v}}^{j}) e_{\bar{v}}^{j}(t)\label{eq:group.decoupled.regulator}\\
		\partial_t \tilde{\varepsilon}^{i}(z, t) &= \underline{\Lambda}^{i}(z) \partial_z \tilde{\varepsilon}^{i}(z, t) + \underline{\tilde{A}}_{0}^{i}(z) \tilde{\varepsilon}^{i}_{-}(0, t) \label{eq:group.decoupled.pde}	\\
		\tilde{\varepsilon}^{i}_{+}(0, t) &= \underline{Q}_{0}^{i} \tilde{\varepsilon}^{i}_{-}(0, t)\label{eq:group.decoupled.bc0}\\
		\tilde{\varepsilon}^{i}_{-}(1, t) &= (I_{N^{i}} \Otimes K_{\bar{v}}^{i}) e_{\bar{v}}^{i}(t)\label{eq:group.decoupled.bc1}\\
		\dot{w}^{i}(t) &= \underline{\tilde{F}}_{w}^{i} w^{i}(t) + \underline{B}_{w}^{i} \tilde{\varepsilon}^{i}_{-}(0, t)\label{eq:group.decoupled.ode}
	\end{align}
\end{subequations}
with
$\underline{F}_{e}^{i} \!=\! I_{N^{i}} \Otimes \tilde{S} \!-\! H^{i i} \Otimes B_{e}^{i} K_{\bar{v}}^{i}$,
$B_{e}^{i} \!=\! (\Pi_{x}^{i}(1) \Lambda^{i}(1) \!-\! \tilde{B}_{y} C_{x,1}^{i}) E_{-}^{i}$.
Each group is thus represented by
\begin{center}
	\begin{tikzpicture}[every path/.style={draw, thick, line  join=bevel}, node distance = 10mm and 10mm]
		\myTikzSet
		\blockTextSize
		%
		%
		\node(uIn) at (0, 0){};
		\node(PIDE)[frame, \stableColor, right=of uIn, align=center]{PDE\subText{\eqref{eq:group.decoupled.pde}--\eqref{eq:group.decoupled.bc1}}};
		\node(ODE)[frame, \stableColor, right=15mm of PIDE, align=center]{ODE\subText{\eqref{eq:group.decoupled.ode}}};
		\draw[->] (PIDE.east) to node[above=\pfeilTextAbstand, align=center](pide2ode){$\tilde{\varepsilon}_{-}^{i}(0, t)$} (ODE.west);
		\begin{pgfonlayer}{back}
			\node(agent)[frame, \stableColor, fit=(PIDE)(ODE), inner xsep=8pt, inner ysep=3pt]{};
		\end{pgfonlayer}
		%
		%
		\node(IM)[frame, \stableColor, left=12mm of PIDE, align=center]{ODE\subText{\eqref{eq:group.decoupled.regulator}}};
		\node(yIn) [coordinate, left=12mm of IM]{};
		\draw[->] (yIn) to node[above=-0.5mm, anchor=south, near start](yInArrow){{\footnotesize$\begin{array}{l} e_{\bar{v}}^{1}(t)\\[-2pt]{~\tikz{\draw[-,densely dotted] (0,0)--(0,0.185);}}\\[-1mm]e_{\bar{v}}^{i-1}(t)\end{array}$}} (IM);
		\draw[->] (IM) to node[above=\pfeilTextAbstand, align=center, near start](ev){~~~$e_{\bar{v}}^{i}(t)$} (PIDE);
	\end{tikzpicture}
\end{center}
in which the ODEs \eqref{eq:group.decoupled.regulator} are driven by the ODE states of the previous groups. Therein, the PDE--ODE subsystem \eqref{eq:group.decoupled.pde}--\eqref{eq:group.decoupled.ode} is stable (see Section \ref{sec:controller.design.localStabilization}).
Hence, the simultaneous stabilization of the ODE subsystem \eqref{eq:group.decoupled.regulator} for $i = 1, \ldots, g$, will lead to a stable closed-loop system \eqref{eq:group.decoupled}.
\begin{remark}
	Note that the cooperative decoupling transformation \eqref{eq:transformation.cooperative.decoupling} is a crucial step in the networked controller design.
	With this, the simultaneous stabilization of the infinite-dimensional MAS is traced back to a simultaneous stabilization of finite-dimensional systems, which can be systematically solved.
	So far, the available simultaneous stabilization results for distributed-parameter systems are limited to infinite-dimensional agents with finite-dimensional unstable dynamics (see \cite{SinghNatarajan2022}), while an extension of these results seems extremely technically hard.\remarkEnd
\end{remark}

To establish \eqref{eq:group.decoupled.regulator}, differentiate \eqref{eq:transformation.cooperative.decoupling} w.r.t. time and insert \eqref{eq:group.regulator.dynamics}, \eqref{eq:group.pde}, \eqref{eq:group.ode}. Then, eliminating $\bar{v}$ by using \eqref{eq:transformation.cooperative.decoupling}, integration by parts, inserting \eqref{eq:group.output}, \eqref{eq:group.bc1}, \eqref{eq:group.bc0}, \eqref{eq:cooperative.feedback.e} as well as utilizing $\underline{E}_{+}^{i} = I_{N^{i}} \Otimes E_{+}^{i}$, $\underline{E}_{-}^{i} = I_{N^{i}} \Otimes E_{-}^{i}$ yields
\begin{align}
	&\dot{e}_{\bar{v}}^{i}(t) 
	= \underline{F}_{e}^{i} e_{\bar{v}}^{i}(t) + \sum_{j=1}^{i-1} - (H^{i j} \Otimes B_{e}^{j} K_{\bar{v}}^{j}) e_{\bar{v}}^{j}(t)
	\nonumber\\&
	+ \big(\!(H^{i j} \Otimes \tilde{S} \Pi_{w}^{j})
	+ (H^{i j} \Otimes \tilde{B}_{y})\underline{\tilde{C}}_{w}^{j}
	- (H^{i j} \Otimes \Pi_{w}^{j}) \underline{\tilde{F}}_{w}^{j}\big) w^{j}(t)
	\nonumber\\&
	+ \int_{0}^{1} \big(\!((H^{i j} \Otimes \Pi_{x}^{j}(\zeta)) \underline{\Lambda}^{j}(\zeta))' + (H^{i j} \Otimes \tilde{S} \Pi_{x}^{j}(\zeta))
	\nonumber\\&\qquad
	+ (H^{i j} \Otimes \tilde{B}_{y} \tilde{C}_{x}^{j}(\zeta))\big) \tilde{\varepsilon}^{j}(\zeta, t) \d \zeta
	\nonumber\\&
	+\! \big(\!(H^{i j} \Otimes \tilde{B}_{y} C_{x,1}^{j}) \underline{E}_{+}^{j} \!-\! (H^{i j} \Otimes \Pi_{x}^{j}(1)) \underline{\Lambda}^{j}(1) \underline{E}_{+}^{j}\big) \tilde{\varepsilon}^{j}_{+}(1, t)
	\nonumber\\&
	+ \big(\!(H^{i j} \Otimes \tilde{B}_{y} C_{x,0}^{j}) (\underline{E}_{+}^{j} \underline{Q}_{0}^{j} + \underline{E}_{-}^{j})
	- (H^{i j} \Otimes \Pi_{w}^{j}) \underline{B}_{w}^{j} 
	\nonumber\\&\qquad
	+ (H^{i j} \Otimes \Pi_{x}^{j}(0)) \underline{\Lambda}^{j}(0) (\underline{E}_{+}^{j} \underline{Q}_{0}^{j} + \underline{E}_{-}^{j})
	\nonumber\\&\qquad
	- \int_{0}^{1} (H^{i j} \Otimes \Pi_{x}^{j}(\zeta)) \underline{\tilde{A}}_{0}^{j}(\zeta) \d\zeta\big) \tilde{\varepsilon}^{j}_{-}(0, t).\!
\end{align}
After replacing the aggregated parameters by means of their Kronecker product representation and factorizing $H^{i j}\Otimes I_{n_{\bar{v}}}$, this leads to \eqref{eq:group.decoupled} if $\Pi_{x}^{i}(z)$, $\Pi_{w}^{i}$ are the solution of the \textit{decoupling equations} given in the subsequent lemma which states their solvability.
\begin{lemma}[Decoupling Equations of the MAS]\label{lem:decouplingEquations.cooperative}
The decoupling equations
\begin{subequations}\label{eq:decouplingEquations.cooperative}
	\begin{align}
		\tilde{S} \Pi_{w}^{i} - \Pi_{w}^{i} \tilde{F}_{w}^{i} &= - \tilde{B}_{y} \tilde{C}_{w}^{i}\label{eq:decouplingEquations.cooperative.sylvester}\\
		(\Pi_{x}^{i}(z) \Lambda^{i}(z))' + \tilde{S} \Pi_{x}^{i}(z) &= -\tilde{B}_{y} \tilde{C}_{x}^{i}(z), \quad z \in (0, 1) \label{eq:decouplingEquations.cooperative.ode}\\
		\Pi_{x}^{i}(0) \Lambda^{i}(0) (E_{+}^{i} Q_{0}^{i} \!+\! E_{-}^{i}) \label{eq:decouplingEquations.cooperative.bc0}
		&- \int_{0}^{1} \Pi_{x}^{i}(\zeta) \tilde{A}_{0}^{i}(\zeta) \d\zeta\\*
		&= \Pi_{w}^{i} B_{w}^{i} \!-\! \tilde{B}_{y} C_{x,0}^{i} (E_{+}^{i} Q_{0}^{i} \!+\! E_{-}^{i})\nonumber\\
		\Pi_{x}^{i}(1) \Lambda^{i}(1) E_{+}^{i} &= \tilde{B}_{y} C_{x,1}^{i} E_{+}^{i} \label{eq:decouplingEquations.cooperative.bc1}
	\end{align}
\end{subequations}
have a unique solution $\Pi_{x}^{i}(z) \in \mathbb{R}^{n_{\bar{v}} \times n^{i}}$ piecewise $C^{1}$ and $\Pi_{w}^{i} \in \mathbb{R}^{n_{\bar{v}} \times n_{w}^{i}}$, $i = 1, \ldots, g$.
\end{lemma}

For the proof, see Appendix \ref{sec:appendix.decouplingEquations}.
From \eqref{eq:cooperative.feedback.e}, \eqref{eq:transformation.cooperative.decoupling} follows
\begin{align}\label{eq:controller.transformed.output}
	u_{2}^{i}(t) &= (I_{N^{i}} \otimes K_{\bar{v}}^{i}) \Big(\bar{v}^{i}(t)
	- \sum_{j=1}^{i} (H^{i j} \Otimes \Pi_{w}^{j}) w^{j}(t)
	\nonumber\\*&\quad
	- \sum_{j=1}^{i} \int_{0}^{1} (H^{i j} \Otimes \Pi_{x}^{j}(\zeta)) \tilde{\varepsilon}^{j}(\zeta, t) \d\zeta \Big).
\end{align}
This yields the cooperative state feedback \eqref{eq:cooperative.cooperative.v2.sum}, \eqref{eq:cooperative.cooperative.v2.u} after applying the backstepping transformation \eqref{eq:transformation.backstepping}, changing the order of integration and using the local decoupling transformation \eqref{eq:transformation.local.decoupling} to obtain the gains
\begin{subequations}\label{eq:controller.gain.cooperative}
	\begin{align}
		K_{c,x}^{i}(z) &= - \Pi_{x}^{i}(z) + \int_{z}^{1} \Pi_{x}^{i}(\zeta) K^{i}(\zeta, z) \d\zeta \\
		K_{c,w}^{i} &= - \Pi_{w}^{i} + \int_{0}^{1} \Pi_{x}^{i}(\zeta) \mathcal{T}_{i}[\Sigma^{i}](\zeta) \d\zeta .
	\end{align}
\end{subequations}

\subsubsection{Simultaneous Stabilization of the ODE}\label{sec:controller.simultaneousStabilization.simultaneousStabilization}{\acceptancenotice\copyrightnotice}%
In the previous subsection, the stabilization of the group-wise aggregated nominal MAS \eqref{eq:group} was traced back to the problem of designing $K_{\bar{v}}^{i}$, $i = 1, \ldots, g$, to stabilize the cascaded ODEs \eqref{eq:group.decoupled.regulator} simultaneously.
The solution of this \textit{simultaneous stabilization problem} is well studied in the literature, see for instance \cite{Lewis2013cooperative}.
A prerequisite for its solvability is the controllability of $(\tilde{S}, B_{e}^{i})$.
The following lemma presents a condition using the followers transfer behavior.
\begin{lemma}[Controllability]\label{lem:controllability}
	The numerator of the transfer matrix $F^{i}(s) = N^{i}(s) (D^{i}(s))^{-1}$, $i = 1, \ldots, g$, from $u^{i \gamma}$ to $y^{i \gamma}$ of the nominal agent \eqref{eq:agent.local.nominal} is
	\begin{align}\label{eq:agent.numerator}
		N^{i}(s) \!=\! \tilde{\mathcal{C}}_{x}^{i}[M^{i}(s)] \det(s I \!-\! \tilde{F}_{w}^{i}) \!+\! \tilde{C}_{w}^{i} \text{adj}(s I \!-\! \tilde{F}_{w}^{i}) B_{w}^{i}\!
	\end{align}
	with 
	$M^{i}(z, s) = \Phi^{i}(z, 0, s) (E^{i}_{-} + E^{i}_{+} Q_{0}^{i}) - \int_{0}^{z} \Phi^{i}(z, \zeta, s) \linebreak\cdot (\Lambda^{i}(\zeta))^{-1} \tilde{A}_{0}^{i}(\zeta)\d\zeta$
	and
	$\Phi^{i}(z, \zeta, s) = \e^{s \int_{\zeta}^{z} (\Lambda^{i}(\eta))^{-1} \d\eta }$.
	If $(S, b_{y})$ is controllable and $\text{rank}\,N^{i}(\mu) = \bar{n}_{-}$, $\forall \mu \in \sigma(S)$, holds, then $(\tilde{S}, B_{e}^{i})$ is controllable.
\end{lemma}

The proof is given in Appendix \ref{sec:appendix.controllability}.
This allows to solve the simultaneous stabilization problem for \eqref{eq:group.decoupled.regulator} by using the next lemma.
\begin{lemma}[Simultaneous Stabilization]\label{lem:simultaneousStabilization}
	If the pair $(\tilde{S}, B_{e}^{i})$, $i = 1, \ldots, g$, is controllable and the leader $\nu^{0}$ is the root of $\mathcal{G}$, then $K_{\bar{v}}^{i} = B_{e}^{i \top} P^{i}$ with the positive definite solution $P^{i} \in \mathbb{R}^{n_{\bar{v}} \times n_{\bar{v}}}$ of the \textit{algebraic Riccati equation (ARE)}
	\begin{align}\label{eq:simultaneousStabilization.are}
		\tilde{S}^{\top} P^{i} + P^{i} \tilde{S} - 2 \kappa^{i} P^{i} B_{e}^{i} B_{e}^{i \top} P^{i} + a^{i} I_{n_{\bar{v}}} = 0,
	\end{align}
	$a^{i} > 0$, $\text{Re}(\lambda_{H}^{i}) \geq \kappa^{i} > 0$, $\forall \lambda_{H}^{i} \in \sigma(H^{i i})$ ensures that $\underline{F}_{e}^{i}$ is a Hurwitz matrix.
\end{lemma}

\begin{proof}
	Since the leader is the root of $\mathcal{G}$, the condition of Lemma \ref{lem:H} is fulfilled and the eigenvalues $\lambda_{H}$ of the leader-follower matrix $H$ have a positive real part. Due to the block-triangular structure of $H$, the same holds for all block-matrices $H^{i i}$ on its diagonal, i.e., $\text{Re}(\lambda_{H}^{i}) > 0, \forall \lambda_{H}^{i} \in \sigma(H^{i i})$. The rest of the proof directly follows from \cite[Th. 3.1]{Lewis2013cooperative}.
\end{proof}

\subsection{Stability of the Networked Controlled Nominal MAS}\label{sec:controller.nominal.stability}%
Next, the stability of the MAS in the nominal case with the previously designed cooperative controller is verified.
By introducing the state $\breve{e}_{\bar{v}}^{i} = \col(e_{\bar{v}}^{1}, \ldots, e_{\bar{v}}^{i})$, $i = 1, \ldots, g$, the cascaded ODEs \eqref{eq:group.decoupled.regulator} are stacked.
Furthermore, the grouped PDE--ODE subsystem \eqref{eq:group.decoupled.pde}--\eqref{eq:group.decoupled.ode} is separated into the individual agents.
Hence, the target system \eqref{eq:group.decoupled} can be viewed as the ODE--PDE--ODE cascades
\begin{subequations}\label{eq:group.decoupled.split}
	\begin{align}
		\dot{\breve{e}}_{\bar{v}}^{i}(t) &= \underline{\breve{F}}_{e}^{i} \breve{e}_{\bar{v}}^{i}(t)\label{eq:group.decoupled.split.regulator}\\
		\partial_t \tilde{\varepsilon}^{i \gamma}(z, t) &= \Lambda^{i}(z) \partial_z \tilde{\varepsilon}^{i \gamma}(z, t) + \tilde{A}_{0}^{i}(z) \tilde{\varepsilon}^{i \gamma}_{-}(0, t) \label{eq:group.decoupled.split.pde}	\\
		\tilde{\varepsilon}^{i \gamma}_{+}(0, t) &= Q_{0}^{i} \tilde{\varepsilon}^{i \gamma}_{-}(0, t)\label{eq:group.decoupled.split.bc0}\\
		\tilde{\varepsilon}^{i \gamma}_{-}(1, t) &= K_{\bar{v}}^{i} ([ 0_{\alpha^{i}}^{\top} \quad e_{\gamma, i}^{\top}] \otimes I_{n_{\bar{v}}}) \breve{e}_{\bar{v}}^{i}(t)\label{eq:group.decoupled.split.bc1}\\
		\dot{w}^{i \gamma}(t) &= \tilde{F}_{w}^{i} w^{i \gamma}(t) + B_{w}^{i} \tilde{\varepsilon}^{i \gamma}_{-}(0, t)\label{eq:group.decoupled.split.ode}
	\end{align}
\end{subequations}
for $i=1, \ldots, g$, $\gamma = 1, \ldots, N^{i}$, illustrated by
\begin{center}
	\begin{tikzpicture}[every path/.style={draw, thick, line  join=bevel}, node distance = 10mm and 10mm]
		\myTikzSet
		\blockTextSize
		%
		%
		\node(uIn) at (0, 0){};
		\node(PIDE)[frame, \stableColor, right=of uIn, align=center]{PDE\subText{\eqref{eq:group.decoupled.split.pde}--\eqref{eq:group.decoupled.split.bc1}}};
		\node(ODE)[frame, \stableColor, right=15mm of PIDE, align=center]{ODE\subText{\eqref{eq:group.decoupled.split.ode}}};
		\draw[->] (PIDE.east) to node[above=\pfeilTextAbstand, align=center](pide2ode){$\tilde{\varepsilon}^{i \gamma}(0, t)$} (ODE.west);
		\begin{pgfonlayer}{back}
			\node(agent)[frame, \stableColor, fit=(PIDE)(ODE), inner xsep=8pt, inner ysep=3pt]{};
		\end{pgfonlayer}
		%
		%
		\node(IM)[frame, \stableColor, left=12mm of PIDE, align=center]{ODE\subText{\eqref{eq:group.decoupled.split.regulator}}};
		\draw[->] (IM) to node[above=\pfeilTextAbstand, align=center, near start](ev){~~~$\breve{e}_{\bar{v}}^{i}(t)$} (PIDE);
	\end{tikzpicture}
\end{center}
To obtain \eqref{eq:group.decoupled.split.bc1} from \eqref{eq:group.decoupled.bc1},
$e_{\bar{v}}^{i \gamma}(t) = (e_{\gamma, i}^{\top} \otimes I_{n_{\bar{v}}}) e_{\bar{v}}^{i}(t) \in \mathbb{R}^{n_{\bar{v}}}$
is utilized, where $e_{\gamma, i} \!\in\! \mathbb{R}^{N^{i}}$ is the $\gamma$-th unit vector.
Then, it follows from the partition
$\breve{e}_{\bar{v}}^{i} = \col(e_{\bar{v}}^{1}, \ldots, e_{\bar{v}}^{i})$
that $e_{\bar{v}}^{i \gamma}(t) = ([ 0_{\alpha^{i}}^{\top} ~~ e_{\gamma, i}^{\top}] \otimes I_{n_{\bar{v}}}) \breve{e}_{\bar{v}}^{i}(t)$
with $\alpha^{i} = \sum_{j=1}^{i-1} N^{j}$
and
$0_{\alpha^{i}} \in \mathbb{R}^{\alpha^{i}}$.
%
Moreover, the matrix 
\begin{align}\label{eq:group.decoupled.split.Fe}
	\underline{\breve{F}}_{e}^{i} = \begin{bmatrix*}
		\underline{F}_{e}^{1} & \cdots & 0\\[-0.8ex]
		\vdots & \ddots & \vdots \\[-0.3ex]
		-H^{i 1} \Otimes B_{e}^{1} K_{\bar{v}}^{1} & \cdots & \underline{F}_{e}^{i}
	\end{bmatrix*}
\end{align}
in \eqref{eq:group.decoupled.split.regulator} is lower block-triangular because of the sum in \eqref{eq:group.decoupled.regulator}.
After these preparations, the following theorem states the stability result for the nominal case.

\begin{theorem}[Nominal Stability] \label{th:nominalStability}
	Assume that $\tilde{F}_{w}^{i}$ and $\underline{F}_{e}^{i}$, $i = 1, \ldots, g$, are Hurwitz matrices.
	Then, the \textit{nominal networked controlled MAS} consisting of the nominal agents \eqref{eq:agent.local.nominal} and the cooperative state feedback regulator \eqref{eq:controller} with \eqref{eq:controller.u}, \eqref{eq:controller.gain.local} and \eqref{eq:controller.gain.cooperative} is asymptotically stable pointwise in space for piecewise continuous ICs $\bm{x}(z, 0) = \bm{x}_{0}(z) \in \mathbb{R}^{n^{i}}$ and $\bm{w}(0) \in \mathbb{R}^{n_{w}^{i}}$, $\bar{\bm{v}}(0) \in \mathbb{R}^{n_{\bar{v}}}$, for all $i = 1, \ldots, g$, $\gamma = 1, \ldots, N^{i}$.
\end{theorem}
\begin{proof}
	The block-triangular form of $\underline{\breve{F}}_{e}^{i}$ in \eqref{eq:group.decoupled.split.Fe} implies $||\breve{e}_{\bar{v}}^{i}(t)|| \leq M_{\breve{e}} \e^{\alpha_{\breve{e}}^{i} t} ||\breve{e}_{\bar{v}}^{i}(0)||$, $t\geq0$, with $\alpha_{\breve{e}}^{i} = \max_{\lambda_{\breve{e}} \in \sigma(\underline{F}_{e}^{1}) \cup \cdots \cup \sigma(\underline{F}_{e}^{i})} \text{Re} ( \lambda_{\breve{e}}) < 0$, $M_{\breve{e}} \geq 1$.
	Then, by making use of the method of characteristics (see \cite[Prop. 2.1]{Hu19}) the solution $\tilde{\varepsilon}^{i \gamma}(z, t)$ of the PDE subsystem can be computed in terms of the $\breve{e}_{\bar{v}}^{i}(t)$ and the  IC $\tilde{\varepsilon}^{i \gamma}(z, 0)$.
	The latter IC is piecewise continuous in view of the backstepping transformation \eqref{eq:transformation.backstepping} and the IC $\bm{x}(z, 0) = x^{i \gamma}(z, 0)$ being piecewise continuous.
	Hence, due to the cascade structure of the PDEs caused by the strictly lower triangular matrix $\tilde{A}_{0-}^{i}(z)$, the solution $\tilde{\varepsilon}^{i \gamma}(z, t)$ depends only for a finite time on the IC and is thus continuous.
	Then, it is smooth because of the boundary coupling to the ODE \eqref{eq:group.decoupled.split.regulator}, which decays exponentially.
	Thus, it follows that $\lim_{t\to\infty} \tilde{\varepsilon}^{i \gamma}(z, t) = 0$, $z \in [0, 1]$, pointwise in space.
	Consequently, $\lim_{t\to\infty} \tilde{\varepsilon}^{i \gamma}_{-}(0, t) = 0$ in \eqref{eq:group.decoupled.split.ode} results and, since $\tilde{F}_{w}^{i}$ is Hurwitz, one has $\lim_{t\to\infty} \bm{w}(t) = 0$.
	Then, by going through the boundedly invertible transformations \eqref{eq:transformation.cooperative.decoupling}, \eqref{eq:transformation.backstepping} and \eqref{eq:transformation.local.decoupling}, the asymptotic stability of the nominal networked controlled MAS \eqref{eq:agent.local.nominal}, \eqref{eq:controller} is shown.
\end{proof}
\begin{remark}
	By a suitable design for $\underline{F}_{e}^{i}$ and $\tilde{F}_{w}^{i}$ the convergence of the nominal networked controlled MAS can be systematically specified.\remarkEnd
\end{remark}
\begin{remark}
	Note that the computation of the local state feedback according to Section \ref{sec:controller.design.localStabilization} as well as the cooperative decoupling equations \eqref{eq:decouplingEquations.cooperative} are independent of the number $N^{i}$ of followers within each subgroup $\mathcal{V}^{i}$, $i = 1, \ldots, g$, hence \textit{scalability} of the networked controller design is ensured.\remarkEnd
\end{remark}
\begin{remark}
	Similar calculations as in the proof of Theorem \ref{th:nominalStability} yield the explicit unique solution of the target system \eqref{eq:group.decoupled.split} for piecewise continuous ICs, which varies continuously with the ICs.
	This asserts for the target system the well-posedness (see \cite{Cu95}).
	Then, by going through the boundedly invertible transformations, the well-posedness also for the nominal networked controlled MAS is verified.\remarkEnd
\end{remark}
%
%
%
\section{Robust Cooperative Output Regulation}\label{sec:robustCooperativeOutputRegulation}{\acceptancenotice\copyrightnotice}%
In this section it is shown that cooperative output regulation \eqref{eq:trackingProblem} is achieved both in the nominal case and even if the followers are subject to model uncertainties which do not destabilize the closed-loop system. This is the main result of this section and is stated next.
\begin{theorem}[Robust Cooperative Output Regulation]\label{th:robustOutputRegulation}
	Assume that the uncertain networked controlled MAS consisting of the followers \eqref{eq:agent} with unknown model uncertainties and the networked controller \eqref{eq:controller}, \eqref{eq:controller.u} is asymptotically stable pointwise in space.
	Then, the cooperative state feedback regulator \eqref{eq:controller}, \eqref{eq:controller.u} achieves \textit{robust cooperative output regulation} \eqref{eq:trackingProblem}, i.e., $\lim_{t\to\infty} e_{y}^{i \gamma}(t) = 0$ for any disturbance input matrices $\bm{G}, \bm{G}_{1}, \bm{G}_{0}, \bm{G}_{w}, \bm{G}_{y}$, for all reference and disturbance inputs modeled by \eqref{eq:signalModel.joint} and for any piecewise continuous ICs of the agents and the controller.
\end{theorem}

For the proof of this theorem, the transformations of Section \ref{sec:controller.design.localStabilization} are utilized.
In the first step, apply
\begin{subequations}\label{eq:transformation.extendedRegulatorEquations}
	\begin{align}
		\varepsilon^{i \gamma}(z, t) &= x^{i \gamma}(z, t) - \bar{\Sigma}^{i \gamma}(z) w^{i \gamma}(t)\\
		\tilde{\varepsilon}^{i \gamma}(z, t) &= \varepsilon^{i \gamma}(z, t) - \int_{0}^{z} \! \bar{K}^{i \gamma}(z, \zeta) \varepsilon^{i \gamma}(\zeta, t) \d\zeta\nonumber\\
			&= \bar{\mathcal{T}}^{i \gamma}[\varepsilon^{i \gamma}(t)](z)
	\end{align}
\end{subequations}
to the uncertain agents \eqref{eq:agent}, which yields
\begin{subequations}\label{eq:agent.transformed}
	\begin{align}
		\!\partial_t \tilde{\bm{\varepsilon}}(z, t) &= \bar{\bm{\Lambda}}(z) \partial_z \tilde{\bm{\varepsilon}}(z, t) \!+\! \mytilde{0.6}{3pt}{\mybar{0.5}{3.3pt}{\bm{A}}}_{0}(z) \tilde{\bm{\varepsilon}}_{-}(0, t) \!+\! \tilde{\bm{G}}(z) \bm{d}(t)\!\label{eq:agent.transformed.pde}\\
		\!\tilde{\bm{\varepsilon}}_{+}(0, t) &= \bar{\bm{Q}}_{0} \tilde{\bm{\varepsilon}}_{-}(0, t) + \bm{G}_{0} \bm{d}(t) \label{eq:agent.transformed.bc0}\\
		\!\tilde{\bm{\varepsilon}}_{-}(1, t) &= \bm{u}(t) + \tilde{\bm{\mathcal{K}}}[\tilde{\bm{\varepsilon}}(t)] + \tilde{\bm{K}}_{w} \bm{w}(t) + \bm{G}_1 \bm{d}(t) \label{eq:agent.transformed.bc1}\\
		\dot{\bm{w}}(t) &= \tilde{\bar{\bm{F}}}_{w} \bm{w}(t) + \bar{\bm{B}}_{w} \tilde{\bm{\varepsilon}}_{-}(0, t) + \bm{G}_w \bm{d}(t) \label{eq:agent.transformed.ode}\\
		\bm{y}(t) &= \bar{\bm{\mathcal{C}}}_{x} \bar{\bm{\mathcal{T}}}^{-1} [\tilde{\bm{\varepsilon}}(t)] + \tilde{\bar{\bm{C}}}_{w} \bm{w}(t) + \bm{G}_{y} \bm{d}(t)  \label{eq:agent.transformed.output}
	\end{align}
\end{subequations}
by similar calculations as in Section \ref{sec:controller.design.localStabilization}.
Thereby, $\bar{\bm{\Sigma}}$, $\bar{\bm{K}}$, $\mytilde{0.6}{3pt}{\mybar{0.5}{3.3pt}{\bm{A}}}_{0}$, $\tilde{\bar{\bm{F}}}_{w}$ follow from \eqref{eq:decouplingEquations.boundary}, \eqref{eq:kernel.backstepping} and the definition of $\tilde{F}_{w}^{i}$ in Section \ref{sec:controller.design.local.decoupling} by replacing the nominal parameters with the related uncertain parameters.
The other parameters are given in the Appendix \ref{sec:appendix.parameters}.
Next, applying the networked controller \eqref{eq:controller} to \eqref{eq:agent.transformed} as well as using the fact that the disturbance and reference inputs are the solution of the joint signal model \eqref{eq:signalModel.joint} yields the transformed aggregated networked controlled MAS subject to the joint signal model
\begin{subequations}\label{eq:extendedSystem.3}
	\begin{align}
		\dot{v}(t) &= S v(t)\label{eq:extendedSystem.3.signalModel}\\
		%
		%
		%
		\dot{\bar{v}}(t) &= \underline{S} \bar{v}(t) + (H \otimes \tilde{B}_{y}) e_{y}(t) \label{eq:extendedSystem.3.internalModel}\\
		%
		%
		\partial_t \tilde{\varepsilon}(z, t) &= \underline{\bar{\Lambda}}(z) \partial_z \tilde{\varepsilon}(z, t) \!+\! \underline{\bar{A}}_{0}(z) \tilde{\varepsilon}_{-}(0, t) \!+\! \underline{G}(z) v(t) \label{eq:extendedSystem.3.pde}\\
		\tilde{\varepsilon}_{+}(0, t) &= \underline{\bar{Q}}_{0} \tilde{\varepsilon}_{-}(0, t) + \underline{G}_{0} v(t) \label{eq:extendedSystem.3.bc0}\\
		\tilde{\varepsilon}_{-}(1, t) &= \underline{K}_{\bar{v}} \bar{v}(t) + \underline{\mathcal{K}}[\tilde{\varepsilon}(t)] + \underline{\tilde{K}}_{w} w(t) + \underline{G}_{1} v(t) \label{eq:extendedSystem.3.bc1}\\
		\dot{w}(t) &= \underline{\tilde{\bar{F}}}_{w} w(t) + \underline{\bar{B}}_{w} \tilde{\varepsilon}_{-}(0, t) + \underline{G}_{w} v(t) \label{eq:extendedSystem.3.ode}\\
		e_{y}(t) &= \underline{\bar{\mathcal{C}}}_{x}[\tilde{\varepsilon}(t)] + \underline{\tilde{\bar{C}}}_{w} w(t) + \underline{G}_{y} v(t) \label{eq:extendedSystem.3.output}
	\end{align}
\end{subequations}
with
$\bar{v} \!=\! \col(\bar{v}^{1 1}\!, ..., \bar{v}^{g N^{g}})$,
$\tilde{\varepsilon} \!=\! \col(\tilde{\varepsilon}^{1 1}\!, ..., \tilde{\varepsilon}^{g N^{g}})$,
$w \!=\! \col(w^{1 1}\!, ..., w^{g N^{g}})$,
$e_{y} \!=\! \col(e_{y}^{1 1}\!, ..., e_{y}^{g N^{g}})$
and the parameters given in Appendix \ref{sec:appendix.parameters}.
Then, the decoupling transformation\hspace*{-1pt}
\begin{subequations}\label{eq:decoupling.extendedRegulatorEquations}
	\begin{align}
		\xi_{\bar{v}}(t) &= \bar{v}(t) - \Upsilon\!_{\bar{v}} v(t) \label{eq:decoupling.extendedRegulatorEquations.v} \\
		\xi(z, t) &= \tilde{\varepsilon}(z, t) - \Upsilon\!_{\varepsilon}(z) v(t) \label{eq:decoupling.extendedRegulatorEquations.x}\\
		\xi_{w}(t) &= w(t) - \Upsilon\!_{w} v(t) \label{eq:decoupling.extendedRegulatorEquations.w}
	\end{align}
\end{subequations}
leads to the error system
\begin{subequations}\label{eq:extendedSystem.4}
	\begin{align}
		\dot{v}(t) &= S v(t)\label{eq:extendedSystem.4.signalModel}\\
		%
		%
		%
		\dot{\xi}_{\bar{v}}(t) &= \underline{S} \xi_{\bar{v}}(t) + (H \otimes \tilde{B}_{y}) e_{y}(t) \label{eq:extendedSystem.4.internalModel}\\
		%
		%
		\partial_t \xi(z, t) &= \underline{\bar{\Lambda}}(z) \partial_z \xi(z, t) \!+\! \underline{\bar{A}}_{0}(z) \xi_{-}(0, t) \label{eq:extendedSystem.4.pde}\\
		\xi_{+}(0, t) &= \underline{\bar{Q}}_{0} \xi_{-}(0, t) \label{eq:extendedSystem.4.bc0}\\
		\xi_{-}(1, t) &= \underline{K}_{\bar{v}} \xi_{\bar{v}}(t) + \underline{\mathcal{K}}[\xi(t)] + \underline{\tilde{K}}_{w} \xi_{w}(t) \label{eq:extendedSystem.4.bc1}\\
		\dot{\xi}_{w}(t) &= \underline{\tilde{\bar{F}}}_{w} \xi_{w}(t) + \underline{\bar{B}}_{w} \xi_{-}(0, t) \label{eq:extendedSystem.4.ode}\\
		e_{y}(t) &= \underline{\bar{\mathcal{C}}}_{x}[\xi(t)] + \underline{\tilde{\bar{C}}}_{w} \xi_{w}(t) \label{eq:extendedSystem.4.output}.
	\end{align}
\end{subequations}
If the subsystem \eqref{eq:extendedSystem.4.internalModel}--\eqref{eq:extendedSystem.4.output} is asymptotically stable, then $\lim_{t\to\infty} e_{y}(t) = 0$ follows from \eqref{eq:extendedSystem.4.output}, i.e., robust cooperative output regulation is achieved.
A straightforward calculation which includes taking the derivative of \eqref{eq:decoupling.extendedRegulatorEquations} w.r.t. time $t$, evaluating \eqref{eq:decoupling.extendedRegulatorEquations.x} at the boundaries and inserting the result in \eqref{eq:extendedSystem.3} as well as replacing $\bar{v}(t), \tilde{\varepsilon}(z, t), w(t)$ using \eqref{eq:decoupling.extendedRegulatorEquations} shows that \eqref{eq:extendedSystem.4} is obtained if $\Upsilon\!_{\bar{v}}, \Upsilon\!_{\varepsilon}, \Upsilon\!_{w}$ are the solution of the so-called \textit{extended regulator equations}. The following lemma states their solvability.

\begin{lemma}[Extended Regulator Equations]\label{lem:extendedRegulatorEquations}
	The \textit{extended regulator equations}
	\begin{subequations}\label{eq:extendedRegulatorEquations}
		\begin{align}
			\Upsilon\!_{\bar{v}} S - \underline{S} \Upsilon\!_{\bar{v}}  &= 0\label{eq:extendedRegulatorEquations.internalModel}\\
			\Upsilon\!_{w} S - \underline{\tilde{\bar{F}}}_{w} \Upsilon\!_{w} &= \underline{\bar{B}}_{w} \underline{E}_{-}^{\top} \Upsilon\!_{\varepsilon}(0) + \underline{G}_{w} \\
			\Upsilon\!_{\varepsilon}(z) S - \underline{\bar{\Lambda}}(z) \Upsilon\!_{\varepsilon}^{\,\prime}(z) &= \underline{\bar{A}}_{0}(z) \underline{E}_{-}^{\top} \Upsilon\!_{\varepsilon}(0) + \underline{G}(z)\\
			(\underline{E}_{+}^{\top} - \underline{\bar{Q}}_{0} \underline{E}_{-}^{\top}) \Upsilon\!_{\varepsilon}(0) &= \underline{G}_{0}\\
			\underline{E}_{-}^{\top} \Upsilon\!_{\varepsilon}(1) - \underline{\mathcal{K}}[\Upsilon\!_{\varepsilon}] &= \underline{K}_{\bar{v}} \Upsilon\!_{\bar{v}} + \underline{\tilde{K}}_{w} \Upsilon\!_{w} + \underline{G}_{1}\\
			\underline{\bar{\mathcal{C}}}_{x}[\Upsilon\!_{\varepsilon}] &= - \underline{\tilde{\bar{C}}}_{w} \Upsilon\!_{w} - \underline{G}_{y}\label{eq:extendedRegulatorEquations.output}
		\end{align}
	\end{subequations}
	on $z \in [0, 1)$ have a unique solution
	$\Upsilon\!_{\bar{v}} \in \mathbb{R}^{N n_{\bar{v}} \times n_{v}}$,
	$\Upsilon\!_{\varepsilon}(z) \in \mathbb{R}^{\sum_{i=1}^{g} N^{i} n^{i} \times n_{v}}$ and
	$\Upsilon\!_{w} \in \mathbb{R}^{\sum_{i=1}^{g} N^{i} n_{w}^{i} \times n_{v}}$
	with $\Upsilon\!_{\varepsilon}(z)$ piecewise $C^{1}$ if the subsystem \eqref{eq:extendedSystem.3.internalModel}--\eqref{eq:extendedSystem.3.output} is asymptotically stable.
\end{lemma}

The proof is given in the Appendix \ref{sec:appendix.extendedRegulatorEquations}.
By using the previous results, Theorem \ref{th:robustOutputRegulation} is verified in the following.

{\textit{Proof of Theorem \ref{th:robustOutputRegulation}: }
	Since the networked controlled MAS \eqref{eq:agent}, \eqref{eq:controller} is asymptotically stable by assumption and the map \eqref{eq:transformation.extendedRegulatorEquations} into \eqref{eq:extendedSystem.3.internalModel}--\eqref{eq:extendedSystem.3.output} is boundedly invertible, the latter system is also asymptotically stable. 
	Then, by Lemma \ref{lem:extendedRegulatorEquations} the mapping \eqref{eq:decoupling.extendedRegulatorEquations} into \eqref{eq:extendedSystem.4} exists.
	As \eqref{eq:extendedSystem.3.internalModel}--\eqref{eq:extendedSystem.3.output} for $v = 0$ coincides with \eqref{eq:extendedSystem.4.internalModel}--\eqref{eq:extendedSystem.4.output}, the asymptotic convergence $\lim_{t\to\infty} (\xi(z, t), \xi_{w}(t), \xi_{\bar{v}}(t)) = (0, 0, 0)$ is evident and due to \eqref{eq:extendedSystem.4.output} $\lim_{t\to\infty} e_{y}(t) = 0$ and thus robust output regulation is verified.
	\proofEnd
}

Obviously, cooperative output regulation in the nominal case is directly implied by the proof of Theorem \ref{th:robustOutputRegulation}.

\begin{remark}
	With $\tilde{e}_{y} = (H \otimes I_{\bar{n}_{-}}) e_{y}$ and $\underline{B}_{y} = I_{N} \otimes \tilde{B}_{y}$ one can rewrite \eqref{eq:extendedSystem.3.internalModel} in the form $\dot{\bar{v}} = \underline{S} \bar{v} + \underline{B}_{y} \tilde{e}_{y}$ of a classical internal model.
	Consequently, signals that are modeled by \eqref{eq:signalModel.joint} are rejected in the redefined tracking error $\tilde{e}_{y}$ in the steady state.
	Note that $\det (H \otimes I_{\bar{n}_{-}}) \neq 0$ if the graph $\mathcal{G}$ has the leader as its root (c.f. Lemma \ref{lem:H}).
	Hence, $\tilde{e}_{y} \equiv 0 \Rightarrow e_{y} \equiv 0$ follows.
	Thus, Theorem \ref{th:robustOutputRegulation} is reminiscent to the classical internal model principle.\remarkEnd
\end{remark}
%
%
%
\section{Example}\label{sec:example}{\acceptancenotice\copyrightnotice}%
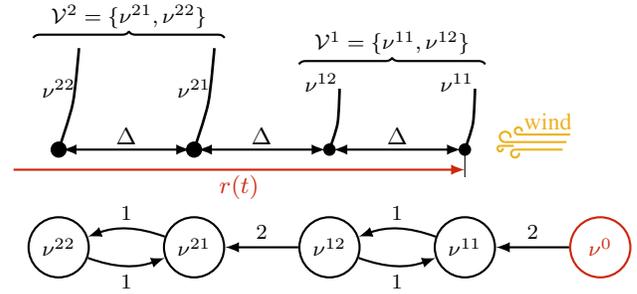
\begin{figure}[t]\centering
	\DeclareDocumentCommand \tikzunderbrace{O{0.5\textwidth} O{0pt} m m m}{
		\tikz[overlay,remember picture]{
			\node[fit=(#3.south west) (#4.south east),inner ysep =#2, inner xsep=0] (gesamt#3) {};
			\draw[decorate, decoration=brace,thick] (gesamt#3.south east)--(gesamt#3.south west);
			\path (gesamt#3.south) node[below,yshift=-2mm]{\parbox{#1}{#5}};
		}	
	}
	\begin{tikzpicture}[arrow/.style={-{latex}},
		every path/.style={draw, thick, line  join=bevel},
		node distance = 7mm and 6mm,
		every state/.style={draw,minimum size=8mm}]
		
		\newcommand{\xDistance}{1.8}
		\newcommand{\scalePics}{0.27}
		\newcommand{\yRope}{1.3}
		
		\pic (rope1) [scale=\scalePics] at (-1*\xDistance, \yRope) {ropeStyle2={\footnotesize $\nu^{1 1}$}};
		\pic (rope2) [scale=\scalePics] at (-2*\xDistance, \yRope) {ropeStyle2={\footnotesize $\nu^{1 2}$}};
		\pic (rope3) [scale=\scalePics] at (-3*\xDistance, \yRope) {ropeStyle3={\footnotesize $\nu^{2 1}$}};
		\pic (rope4) [scale=\scalePics] at (-4*\xDistance, \yRope) {ropeStyle3={\footnotesize $\nu^{2 2}$}};
		\draw[decorate, decoration=brace, thick] (rope2-top-left) -- node[above] {\footnotesize$\mathcal{V}^{1} = \{\nu^{11}, \nu^{12} \}$} (rope1-top-right);
		\draw[decorate, decoration=brace, thick] (rope4-top-left) -- node[above] {\footnotesize$\mathcal{V}^{2} = \{\nu^{21}, \nu^{22} \}$} (rope3-top-right);
		
		\newcommand{\yDelta}{0}\small
		\draw [{latex}-{latex}] (rope1-left) -- node[above=-0.5mm]{$\Delta$} (rope2-right);
		\draw [{latex}-{latex}] (rope2-left) -- node[above=-0.5mm]{$\Delta$} (rope3-right);
		\draw [{latex}-{latex}] (rope3-left) -- node[above=-0.5mm]{$\Delta$} (rope4-right);
		\draw [thin] (rope1-below) --++ (0, -0.3) node (rHelper) {};
		\node [coordinate, below=2mm of rope1-below] (rHelper) {};
		\node [coordinate, left=60mm of rHelper] (rIn) {};
		\draw [{latex}-, leaderColor] (rHelper) -- node[below=-0.5mm]{$\footnotesize r(t)$} (rIn);
		
		\newcommand{\windX}{-1.40}
		\newcommand{\windY}{\yRope}
		\draw [matlabOrange, thick, -{Hooks[right,arc=280]}] 			(\windX+1.00, \windY+0.2) 	-- (\windX+0.25, \windY+0.2) node[above=-0.5mm]{\hspace{8mm} wind};
		\draw [matlabOrange, thick, -{Hooks[scale=1.4,right,arc=280]}] 	(\windX+0.95, \windY+0.1) -- (\windX+0.05, \windY+0.1);
		\draw [matlabOrange, thick, -{Hooks[scale=0.8,left,arc=280]}] 	(\windX+0.25, \windY+0.05) 	-- (\windX+0.02, \windY+0.05);
		\draw [matlabOrange, thick, -{Hooks[scale=1.1,left,arc=280]}] 	(\windX+0.90, \windY) 		-- (\windX+0.20, \windY);
		
		\node[state, leaderColor] (a0) at (-0*\xDistance, 0) {\footnotesize$\nu^{0}$};
		\node[state] (a1) at (-1*\xDistance, 0) {\footnotesize$\nu^{1 1}$};
		\node[state] (a2) at (-2*\xDistance, 0) {\footnotesize$\nu^{1 2}$};
		\node[state] (a3) at (-3*\xDistance, 0) {\footnotesize$\nu^{2 1}$};
		\node[state] (a4) at (-4*\xDistance, 0) {\footnotesize$\nu^{2 2}$};
		\draw[arrow] 							(a0) to node[above] {\footnotesize$2$} (a1);
		\draw[arrow, bend angle=20, bend right] (a1) to node[above] {\footnotesize$1$} (a2);
		\draw[arrow] 							(a2) to node[above] {\footnotesize$2$} (a3);
		\draw[arrow, bend angle=20, bend right] (a3) to node[above] {\footnotesize$1$} (a4);
		\draw[arrow, bend angle=20, bend right] (a4) to node[below] {\footnotesize$1$} (a3);
		\draw[arrow, bend angle=20, bend right] (a2) to node[below] {\footnotesize$1$} (a1);
	\end{tikzpicture}
	\caption{\unboldmath Platoon of the nominal heavy ropes, which are the followers $\nu^{1 1}, \nu^{1 2}, \nu^{2 1}, \nu^{2 2}$ with the constant distances $\Delta$ specifying the formation and the reference input $r(t)$ determining the platoons position.
		Below, the communication graph $\mathcal{G}$ with the leader agent $\nu^{0}$, the informed agent $\nu^{1 1}$ and the uninformed agents $\nu^{1 2}, \nu^{2 1}, \nu^{2 2}$ is shown, which can be grouped into $\mathcal{V}^{1} = \{\nu^{1 1}, \nu^{1 2}\}$ and $\mathcal{V}^{2} = \{\nu^{2 1}, \nu^{2 2}\}$.} 
	\label{pic:platoon}
\end{figure}
In this example the results of the paper are applied to a platoon of four heavy ropes for the cooperative transportation of loads, which is depicted in Fig.~\ref{pic:platoon}.
Therein, each rope represents a follower and they should keep constant distances $\Delta$ between their loads, even if there are disturbances like wind and uncertainties in the load masses and the rope lengths.
This formation control problem is significantly impeded by the wave dynamics of the ropes, giving rise to a swinging of the loads.
The leader $\nu^{0}$, which is the finite-dimensional reference model \eqref{eq:signalModel.leader}, specifies the position of the formation in terms of the reference signal $r(t)$.

\subsubsection{Network Topology}
The communication is restricted to the graph $\mathcal{G}$ shown in Fig. \ref{pic:platoon}.
Obviously, the leader $\nu^{0}$ is the only root of $\mathcal{G}$ so that the leader-follower matrix
\begin{align}\arraycolsep=1.3pt
	H \!=\! \left[\!\begin{array}{cc} H^{11} \!&\! 0 \\ H^{21} \!&\! H^{22}\end{array}\!\right]\!,
	H^{11} \!=\! H^{22} \!=\! \left[\!\begin{array}{cc} 3 \!&\! -1\\-1 \!&\! 1\end{array}\!\right]\!,
	H^{21} \!=\! \left[\!\begin{array}{cc} 0 \!&\! -2\\0 \!&\! 0\end{array}\!\right]
\end{align}
is nonsingular (c.f. Lemma \ref{lem:H}).
With this partition, the $N = 4$ followers $\nu^{1 1}, \nu^{1 2}, \nu^{2 1}, \nu^{2 2}$ can be split into $g=2$ strongly connected groups $\mathcal{V}^{1} = \{\nu^{1 1}, \nu^{1 2}\}$ and $\mathcal{V}^{2} = \{\nu^{2 1}, \nu^{2 2}\}$ with $N_{1} = N_{2} = 2$.
Thus, two subgroups of followers with distinct nominal dynamics can be taken into account, i.e., the followers in $\mathcal{V}^{1} = \{\nu^{1 1}, \nu^{1 2}\}$ have shorter ropes and carry lighter loads than the ones in $\mathcal{V}^{2} = \{\nu^{2 1}, \nu^{2 2}\}$ (see Fig.~\ref{pic:platoon}).

\subsubsection{Nominal Agent Dynamics}
Fig.~\ref{pic:rope} shows the model of a linearized rope (see \cite{Pe01}).
\begin{figure}[t]\centering
	\begin{tikzpicture}
		\tikzset{>=latex}
		\usetikzlibrary{decorations}
		\newcommand{\myTextSize}{\small}
		\newcommand{\myGapEqNameEq}{0.5ex}
		\newcommand{\ropeLength}{2.85}
		\newcommand{\xPosContainer}{0.1}
		\newcommand{\yPosContainer}{1}
		\myTextSize
		\newcommand{\axisXPos}{2.8}
		\draw [<-|] (\axisXPos, \ropeLength+0.3) -- (\axisXPos, \yPosContainer);
		\node [left] at (\axisXPos, \ropeLength+0.03) {$s\!=\!l^{i}$};
		\node [left] at (\axisXPos, 1+0.025) {$s\!=\!0$};
		\draw (\axisXPos-0.05, \ropeLength) -- (\axisXPos +0.05, \ropeLength);
		
		\draw [->] (0.75, \ropeLength) -- node[above=-0.5mm] {$\bar{u}^{i \gamma}(t)$} (0, \ropeLength);
		
		\draw [line width = 1] plot [smooth] coordinates {(0, \ropeLength) (0.25, 3.4*\ropeLength/5) (-0.025, 2.2*\ropeLength/5) (\xPosContainer, \yPosContainer)};
		
		\node [circle, draw, fill=black, scale = 0.75, line width = 1] at (\xPosContainer, \yPosContainer) {};
		\node at (\xPosContainer+5mm, \yPosContainer) {\footnotesize$m^{i}$};
		
		\newcommand{\myZero}{\xPosContainer+1.5}
		\draw [thin] (\myZero, \yPosContainer-0.1) -- (\myZero, 2.2);
		\draw [->] (\myZero, 3.4*\ropeLength/5) -- node[above=-0.5mm]{$v^{i \gamma}(s, t)$} (0.25, 3.4*\ropeLength/5);
		
		%
		%
		%
		\node at (-3, \ropeLength){
			\parbox{5cm}{\myTextSize \flushleft
				PDE:\\[\myGapEqNameEq]
				$\quad \rho \partial_{t}^{2} v^{i \gamma}(s, t) \!=\! \partial_{s} ( \tau^{i}(s) \partial_{s} v^{i \gamma}(s, t))$
			}
		};
		
		\node at (-3, \ropeLength - 0.75){
			\parbox{5cm}{\myTextSize \flushleft
				Neumann actuation:\\[\myGapEqNameEq]
				$\quad \partial_{s} v^{i \gamma}(l^{i}, t)\!=\!\bar{u}^{i \gamma}(t)$
			}
		};
		
		\node at (-3, \ropeLength -1.55){
			\parbox{5cm}{\myTextSize\flushleft
				Dynamic BC:\\[\myGapEqNameEq]%
				$\quad m^{i} \partial_{t}^{2} v^{i \gamma}(0, t) \!=\! \tau^{i}(0) \partial_{s} v^{i \gamma}(0, t)$
			}
		};
	\end{tikzpicture}
	\caption{\unboldmath Nominal model of the heavy ropes with a load for both groups $\mathcal{V}^{i}$, $i=1, 2$, with
		the rope lengths $l^{1} = 3\,\text{m}$, $l^{2} = 5\,\text{m}$, the density $\rho = 0.5 \frac{\text{kg}}{\text{m}}$,
		the load masses $m^{1} = 0.2\,\text{kg}$, $m^{2} = 1\,\text{kg}$,
		the gravity constant $g = 9.81 \frac{\text{m}}{\text{s}^2}$
		and the force in the ropes $\tau^{i}(s) = m^{i} g + \rho g s$.} 
	\label{pic:rope}
\end{figure}
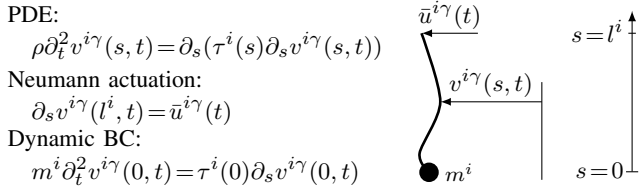
The horizontal displacement $v^{i \gamma}(s, t) \in \mathbb{R}$ of the ropes is governed by a wave equation, while the inertia of the load is taken into account with a dynamic boundary condition.
Moreover, the control input $u^{i \gamma}(t)$ in the Neumann boundary condition at $s=l^{i}$ represents a force that is applied at the suspension point.
The wave dynamics are expressed on the spatial domain $[0, 1]$ by using $z = \frac{s}{l^{i}}$ as well as $\bar{v}^{i \gamma}(z, t) = v^{i \gamma}(z l^{i} , t)$.
Then, with
\begin{align}\label{eq:wave2riemann}
	\!x^{i \gamma\!}(z, t) = \e^{\int_{0}^{z} \frac{g}{l^{i}} \, (2 \epsilon^{i}(\zeta))^{-2} \d\zeta} \begin{bmatrix*} \epsilon^{i}(z) \!&\! 1 \\ -\epsilon^{i}(z) \!&\! 1 \end{bmatrix*} \!\begin{bmatrix*} \partial_{z} \bar{v}^{i \gamma\!}(z, t) \\ \partial_{t} \bar{v}^{i \gamma\!}(z, t) \end{bmatrix*}\!,\!
\end{align}
$\epsilon^{i}(z) = \sqrt{\frac{\tau(l^{i} z)}{(l^{i})^2 \rho}} $ and  $w^{i \gamma}(t) = \col(\bar{v}^{i \gamma}(0, t), \dot{\bar{v}}^{i \gamma}(0, t))$ the wave system is mapped into the nominal heterodirectional hyperbolic PDE--ODE system \eqref{eq:agent.local.nominal} with the parameters
\begin{align}\label{eq:rope.riemann.parameter}
	&\Lambda^{i}(z) \!=\! \begin{bmatrix*}\epsilon^{i}(z)\!& 0\\ 0 &\!-\epsilon^{i}(z)\!\end{bmatrix*}
	\!,~
	A^{i}(z) \!=\! \frac{g}{4 l^{i} \epsilon^{i}(z) }\begin{bmatrix*}0 & \!-1\\ 1 &\!0 \end{bmatrix*}
	\!,\!\\
	&\nonumber
	C_{0}^{i} \!=\! \begin{bmatrix*} 0 &\! 2 \end{bmatrix*}
	\!,
	F_{w}^{i} \!=\! \begin{bmatrix*} 0 &\! 1 \\ 0 &\! -\frac{g}{l^{i} \epsilon^{i}(0)} \end{bmatrix*}
	\!,
	B_{w}^{i} \!=\! \begin{bmatrix*} 0 \\ \frac{g}{l^{i} \epsilon^{i}(0)} \end{bmatrix*}
	\!,
	C_{w}^{i} \!=\! \begin{bmatrix*} 1 &\! 0 \end{bmatrix*}\!
\end{align}
$A_{0}^{i}(z) = 0,  
F^{i}(z, \zeta) = 0,
C^{i}(z) = 0,
Q_{0}^{i} = -1,
Q_{1}^{i} = 1$
and
$u^{i \gamma}(t) = 2 l^{i} \epsilon^{i}(1) \e^{\int_{0}^{1} \frac{g}{2 l^{i}} \, (\epsilon^{i}(\zeta))^{-2} \d\zeta} \bar{u}^{i \gamma}(t)$, $i=1, 2$.

\subsubsection{Cooperative Formation Control Problem}{\acceptancenotice\copyrightnotice}%
The followers are supposed to keep a distance of $\Delta = 2\,\text{m}$ between the loads.
This can be achieved by including the shift $\delta^{i \gamma}$ for each agent $\nu^{i \gamma}$ in the output to be synchronized, i.e.,
$y^{i \gamma}(t) = C_{w}^{i} w^{i \gamma}(t) + \delta^{i \gamma}$.
Then, due to the synchronization $\lim_{t\to\infty} (y^{i \gamma}(t) - r(t)) = 0$, the networked controller ensures that the load position $w^{i \gamma}_{1}(t) = C_{w}^{i} w^{i \gamma}(t)$ keeps the distance $\delta^{i \gamma}$ to the reference signal $r(t)$.
Hence, with $\delta^{1 1} = 0$, $\delta^{1 2} = \Delta$, $\delta^{2 1} = 2 \Delta$, $\delta^{2 2} = 3\Delta$ the desired platoon formation results.
To perform a set-point change, ramp-shaped reference signals $r(t) = \alpha_0 + \alpha_1 t$, $\alpha_0, \alpha_1 \in \mathbb{R}$, are used, while the disturbances $d^{i \gamma}(t) = d^{i \gamma}_{0} \in \mathbb{R}$ are assumed to be constant for considering the influence of the wind.
These signals are modeled by \eqref{eq:signalModel.leader} and \eqref{eq:signalModel.disturbance} with
\begin{align}\label{eq:signalModel.example.parameter}
	S_{r} = \begin{bmatrix*} 0&1 \\ 0&0 \end{bmatrix*}
	,\quad
	S_{d}^{i \gamma} = 0.
\end{align}
Then, for the joint signal model \eqref{eq:signalModel.joint} the matrix $S = S_{r}$ is the cyclic part of $\diag(S_r, S_{d}^{1 1}, \ldots, S_{d}^{2 2})$.
The disturbances are assumed to act only at the loads.
Thus, $G^{i \gamma}_{w} = \begin{bmatrix*}0&1\end{bmatrix*}{\!}^{\top}$, $G^{i \gamma}(z) = 0$, $G^{i \gamma}_{0} = 0$, $G^{i \gamma}_{1} = 0$ and $G^{i \gamma}_{y} = 0$ are the disturbance input matrices for $i=1, 2$, $\gamma=1, 2$.
In practice, the actual parameters differ from the nominal values.
This is taken into account by a perturbation of the nominal masses to
$\bar{m}^{1 1} = 0.25\,\text{kg}$,
$\bar{m}^{1 2} = 0.3\,\text{kg}$,
$\bar{m}^{2 1} = 1.4\,\text{kg}$,
$\bar{m}^{2 2} = 0.8\,\text{kg}$
and uncertainty in the rope lengths according to
$\bar{l}^{1 1} = 3.8\,\text{m}$,
$\bar{l}^{1 2} = 2.2\,\text{m}$,
$\bar{l}^{2 1} = 5.8\,\text{m}$,
$\bar{l}^{2 2} = 6\,\text{m}$.
Then, the uncertain follower dynamics \eqref{eq:agent} result by evaluating \eqref{eq:rope.riemann.parameter} with the latter parameters giving rise to uncertain transport velocities and coupling matrices.
Thus, the networked controller should ensure the formation in the presence of disturbances and has to be robust w.r.t. these model uncertainties.

\subsubsection{Networked Controller Design}
\textbf{(1) Cooperative internal model:}
For the cooperative internal model \eqref{eq:controller.dynamics} $b_{y} = \begin{bmatrix*}0&1\end{bmatrix*}{\!}^{\top}$ is chosen such that $(S, b_{y})$ is controllable.
Due to $\dim y^{i \gamma}(t) = \bar{n}_{-} = 1$, only one copy has to be incorporated, hence $\tilde{S} = S$, $\tilde{B}_{y} = b_{y}$ and $\dim \bar{v}^{i \gamma}(t) = n_{\bar{v}} = 2$ result for the two groups.

\textbf{(2) Local stabilization:}
In order to stabilize the ODE subsystem in \eqref{eq:agent.local.decoupled}, the gain matrices $K_{w}^{i}$, $i=1,2$, are designed by an eigenvalue placement, in order to ensure $\sigma(\tilde{F}_{w}^{i}) = \{-4, -4\}$, $i=1, 2$, (c.f. Section \ref{sec:controller.design.local.decoupling}).
Then, the solution of the local decoupling equations \eqref{eq:decouplingEquations.boundary} is calculated by solving the Volterra integral equation \eqref{eq:decouplingEquations.boundary.volterra} with the method of successive approximations.
If the maximum pointwise difference between two iterations is less then $10^{-9}$, then the computations are stopped, which occurs after $25$ steps.
The backstepping kernels $K^{i}(z, \zeta)$, $i=1, 2$, are calculated by solving the kernel equations \eqref{eq:kernel.backstepping} with the method of characteristics and successive approximations (see \cite{Hu19}).
For this, $18$ iterations are needed in order to obtain a maximum pointwise difference between two iterations of less then $10^{-3}$.

\textbf{(3) Simultaneous stabilization:}
To decouple the internal model, the solution of the cooperative decoupling equations \eqref{eq:decouplingEquations.cooperative} is computed by using the calculations of Appendix \ref{sec:appendix.decouplingEquations}.
With Lemma \ref{lem:controllability} the controllability of the decoupled ODEs, i.e., of the pairs $(\tilde{S}, B_{e}^{i})$, $i=1, 2$, is verified by evaluating the numerators $N^{i}(\mu)$, $i=1, 2$ at $\mu \in \sigma(S) = \{ 0, 0 \}$, i.e., $N^{1}(0) = 4.95$, $N^{2}(0) = 2.21$.
Then, the simultaneous stabilization problem for \eqref{eq:group.decoupled.regulator} is solved with the Matlab function \texttt{are} using the parameters
$\kappa^{1} = \kappa^{2} = 0.585$, $a^{1} = 55$, $a^{2} = 85$
in \eqref{eq:simultaneousStabilization.are}.
In particular, $\text{Re}(\lambda_{H}^{i}) \geq \kappa^{i} > 0$, $\forall \lambda_{H}^{i} \in \sigma(H^{ii}) = \{ 2\pm \sqrt{2} \}$, $a^{i} > 0$, $i=1, 2$, hold.
To utilize the networked controller for the agents with the wave dynamics shown in Fig.~\ref{pic:rope}, the map \eqref{eq:wave2riemann} must be inserted into the state feedback regulator \eqref{eq:controller.u} in order to determine it in terms of the position $v^{i \gamma}(s, t)$, bending $\partial_{s} v^{i \gamma}(s, t)$ and velocity $\partial_{t} v^{i \gamma}(s, t)$ of the ropes.

\subsubsection{Simulations}
\pgfplotsset{snapshotOpts/.style=
	{	width=0.415\linewidth,
		height=0.19\linewidth,
		at={(0\linewidth,0\linewidth)},
		scale only axis,
		every axis title/.append style={at={(current axis.north)}, yshift=-1.9ex, font=\small},
		every axis x label/.append style={font=\footnotesize, yshift=3mm},
		every axis y label/.append style={rotate=-90, font=\footnotesize, xshift=5ex},
		every tick label/.append style={font=\footnotesize},
		every x tick label/.append style={yshift=0.5mm},
		every y tick label/.append style={xshift=0.75mm},,
		every axis plot/.append style={line width=1pt,opacity=1, mark=*, mark indices=1},
		baseline,
		trim axis left, trim axis right,
		unit vector ratio=1 1 1,
		xtick={-8, -6, -4,..., 26},
		ymin=-0.7, ymax=6.3,
		major tick length = 2pt,
}}%
\newcommand{\earlySnap}{
	\,(\tikz[baseline]{
		\draw[ghostColor, line width=1pt] (0mm, 3pt) -- (4.75mm, 3pt);
	})
}
\newcommand{\lateSnap}{
	\,(\tikz[baseline]{
		\draw[colA0, densely dashed, line width=1pt] (0mm, -1pt) -- (4.75mm, -1pt);
		\draw[colA1, line width=1pt] (0mm, 0.75pt) -- (4.75mm, 0.75pt);
		\draw[colA2, line width=1pt] (0mm, 2.5pt) -- (4.75mm, 2.5pt);
		\draw[colA3, line width=1pt] (0mm, 4.25pt) -- (4.75mm, 4.25pt);
		\draw[colA4, line width=1pt] (0mm, 6pt) -- (4.75mm, 6pt);
	})
}
\newcommand{\lateSnapD}{
	\,(\tikz[baseline]{
		\draw[colA1, line width=1pt] (0mm, 0pt) -- (4.75mm, 0pt);
		\draw[colA2, line width=1pt] (0mm, 1.75pt) -- (4.75mm, 1.75pt);
		\draw[colA3, line width=1pt] (0mm, 3.5pt) -- (4.75mm, 3.5pt);
		\draw[colA4, line width=1pt] (0mm, 5.25pt) -- (4.75mm, 5.25pt);
	})
}
\newcommand{\loadSize}{2.2pt}
\newcommand{\myGrid}[2]{
	\addplot [color=colA0!100!black, densely dashed, line width=0.7pt, opacity=1, no markers] coordinates {(#1, -2) (#1, 8)};
}%
\newcommand{\myDelta}[2]{
	\draw [{latex}-{latex}] (axis cs:#1-0*#2-0.2, 0) -- node[above=-0.5mm]{\footnotesize$\Delta$} (axis cs:#1-1*#2+0.23, 0);
	\draw [{latex}-{latex}] (axis cs:#1-1*#2-0.23, 0) -- node[above=-0.5mm]{\footnotesize$\Delta$} (axis cs:#1-2*#2+0.35, 0);
	\draw [{latex}-{latex}] (axis cs:#1-2*#2-0.35, 0) -- node[above=-0.5mm]{\footnotesize$\Delta$} (axis cs:#1-3*#2+0.28, 0);
}%
{\begin{figure}[t]\centering%
		\begin{tikzpicture}
			\begin{axis}[plotOpts, xmin=0, xmax=30, xtick={0, 5, ..., 30}, ymin=-7, ymax=22, xlabel={$t$}]
				\addplot [opacity=0.2, samples=2, thin] coordinates {(10, -10) (10, 30)};
				\addplot [opacity=0.2, samples=2, thin] coordinates {(20, -10) (20, 30)};
				\addplot [color=colA4] table[x=t, y=agent4.w1]{pic/pgf/ropeMas_referenceTracking.dat};\label{pgf:a4}
				\addplot [color=colA3] table[x=t, y=agent3.w1]{pic/pgf/ropeMas_referenceTracking.dat};\label{pgf:a3}
				\addplot [color=colA2] table[x=t, y=agent2.w1]{pic/pgf/ropeMas_referenceTracking.dat};\label{pgf:a2}
				\addplot [color=colA1] table[x=t, y=agent1.w1]{pic/pgf/ropeMas_referenceTracking.dat};\label{pgf:a1}
				\addplot [color=colA0, samples=2, densely dashed] coordinates {(0, 0) (10, 15) (20, 20) (30, 20)};\label{pgf:a0}
				%
				\node[anchor=north] at (axis cs:19, 9) {\parbox{4cm}{\raggedleft\footnotesize\def\arraystretch{1.2}\arraycolsep=1.4pt
						$\begin{array}{rl@{\quad}rl}
							~&~									&\ref{pgf:a0}	&r(t)\\
							\ref{pgf:a1} 	&w_{1}^{1 1}(t)		&\ref{pgf:a2} 	&w_{1}^{1 2}(t)\\
							\ref{pgf:a3} 	&w_{1}^{2 1}(t) 	&\ref{pgf:a4} 	&w_{1}^{2 2}(t)
						\end{array}$}};
			\end{axis}
		\end{tikzpicture}\\[-1ex]%
		{\renewcommand{\arraystretch}{0}
		\setlength{\tabcolsep}{1pt}
		\begin{tabular}{ll}
			\begin{tikzpicture}%
				\begin{axis}[snapshotOpts, title={$t = 1 \earlySnap, t=1.5 \lateSnap$}, ylabel={$s$}]
					\myGrid{2.25}{1.5}
					\addplot [color=ghostColor, mark size=sqrt(0.25)*\loadSize, opacity=1] table[x=agent1.v, y=agent1.s]{pic/pgf/ropeMas_referenceTracking_5.dat};
					\addplot [color=ghostColor, mark size=sqrt(0.30)*\loadSize, opacity=1] table[x=agent2.v, y=agent2.s]{pic/pgf/ropeMas_referenceTracking_5.dat};
					\addplot [color=ghostColor, mark size=sqrt(1.40)*\loadSize, opacity=1] table[x=agent3.v, y=agent3.s]{pic/pgf/ropeMas_referenceTracking_5.dat};
					\addplot [color=ghostColor, mark size=sqrt(0.80)*\loadSize, opacity=1] table[x=agent4.v, y=agent4.s]{pic/pgf/ropeMas_referenceTracking_5.dat};
					\addplot [color=colA1, mark size=sqrt(0.25)*\loadSize] table[x=agent1.v, y=agent1.s]{pic/pgf/ropeMas_referenceTracking_1.dat};
					\addplot [color=colA2, mark size=sqrt(0.30)*\loadSize] table[x=agent2.v, y=agent2.s]{pic/pgf/ropeMas_referenceTracking_1.dat};
					\addplot [color=colA3, mark size=sqrt(1.40)*\loadSize] table[x=agent3.v, y=agent3.s]{pic/pgf/ropeMas_referenceTracking_1.dat};
					\addplot [color=colA4, mark size=sqrt(0.80)*\loadSize] table[x=agent4.v, y=agent4.s]{pic/pgf/ropeMas_referenceTracking_1.dat};
				\end{axis}
			\end{tikzpicture}%
			&
			\begin{tikzpicture}%
				\begin{axis}[snapshotOpts, title={$t = 9.5 \earlySnap, t=10 \lateSnap$}, xtick={5, 7, ..., 21}]
					\myGrid{15}{14.25}
					\addplot [color=ghostColor, mark size=sqrt(0.25)*\loadSize, opacity=1] table[x=agent1.v, y=agent1.s]{pic/pgf/ropeMas_referenceTracking_6.dat};
					\addplot [color=ghostColor, mark size=sqrt(0.30)*\loadSize, opacity=1] table[x=agent2.v, y=agent2.s]{pic/pgf/ropeMas_referenceTracking_6.dat};
					\addplot [color=ghostColor, mark size=sqrt(1.40)*\loadSize, opacity=1] table[x=agent3.v, y=agent3.s]{pic/pgf/ropeMas_referenceTracking_6.dat};
					\addplot [color=ghostColor, mark size=sqrt(0.80)*\loadSize, opacity=1] table[x=agent4.v, y=agent4.s]{pic/pgf/ropeMas_referenceTracking_6.dat};
					\addplot [color=colA1, mark size=sqrt(0.25)*\loadSize] table[x=agent1.v, y=agent1.s]{pic/pgf/ropeMas_referenceTracking_2.dat};
					\addplot [color=colA2, mark size=sqrt(0.30)*\loadSize] table[x=agent2.v, y=agent2.s]{pic/pgf/ropeMas_referenceTracking_2.dat};
					\addplot [color=colA3, mark size=sqrt(1.40)*\loadSize] table[x=agent3.v, y=agent3.s]{pic/pgf/ropeMas_referenceTracking_2.dat};
					\addplot [color=colA4, mark size=sqrt(0.80)*\loadSize] table[x=agent4.v, y=agent4.s]{pic/pgf/ropeMas_referenceTracking_2.dat};
				\end{axis}
			\end{tikzpicture}%
			\\
			\begin{tikzpicture}%
				\begin{axis}[snapshotOpts, title={$t = 20.5 \earlySnap, t=21 \lateSnap$}, xlabel={$v^{i \gamma}(s, t)$}, ylabel={$s$}]
					\myGrid{20}{20}
					\addplot [color=ghostColor, mark size=sqrt(0.25)*\loadSize, opacity=1] table[x=agent1.v, y=agent1.s]{pic/pgf/ropeMas_referenceTracking_7.dat};
					\addplot [color=ghostColor, mark size=sqrt(0.30)*\loadSize, opacity=1] table[x=agent2.v, y=agent2.s]{pic/pgf/ropeMas_referenceTracking_7.dat};
					\addplot [color=ghostColor, mark size=sqrt(1.40)*\loadSize, opacity=1] table[x=agent3.v, y=agent3.s]{pic/pgf/ropeMas_referenceTracking_7.dat};
					\addplot [color=ghostColor, mark size=sqrt(0.80)*\loadSize, opacity=1] table[x=agent4.v, y=agent4.s]{pic/pgf/ropeMas_referenceTracking_7.dat};
					\addplot [color=colA1, mark size=sqrt(0.25)*\loadSize] table[x=agent1.v, y=agent1.s]{pic/pgf/ropeMas_referenceTracking_3.dat};
					\addplot [color=colA2, mark size=sqrt(0.30)*\loadSize] table[x=agent2.v, y=agent2.s]{pic/pgf/ropeMas_referenceTracking_3.dat};
					\addplot [color=colA3, mark size=sqrt(1.40)*\loadSize] table[x=agent3.v, y=agent3.s]{pic/pgf/ropeMas_referenceTracking_3.dat};
					\addplot [color=colA4, mark size=sqrt(0.80)*\loadSize] table[x=agent4.v, y=agent4.s]{pic/pgf/ropeMas_referenceTracking_3.dat};
				\end{axis}
			\end{tikzpicture}%
			&
			\begin{tikzpicture}%
				\begin{axis}[snapshotOpts, title={$t = 29.5 \earlySnap, t=30 \lateSnap$}, xlabel={$v^{i \gamma}(s, t)$}]
					\myGrid{20}{20}
					\addplot [color=ghostColor, mark size=sqrt(0.25)*\loadSize, opacity=1] table[x=agent1.v, y=agent1.s]{pic/pgf/ropeMas_referenceTracking_8.dat};
					\addplot [color=ghostColor, mark size=sqrt(0.30)*\loadSize, opacity=1] table[x=agent2.v, y=agent2.s]{pic/pgf/ropeMas_referenceTracking_8.dat};
					\addplot [color=ghostColor, mark size=sqrt(1.40)*\loadSize, opacity=1] table[x=agent3.v, y=agent3.s]{pic/pgf/ropeMas_referenceTracking_8.dat};
					\addplot [color=ghostColor, mark size=sqrt(0.80)*\loadSize, opacity=1] table[x=agent4.v, y=agent4.s]{pic/pgf/ropeMas_referenceTracking_8.dat};
					\addplot [color=colA1, mark size=sqrt(0.25)*\loadSize] table[x=agent1.v, y=agent1.s]{pic/pgf/ropeMas_referenceTracking_4.dat};
					\addplot [color=colA2, mark size=sqrt(0.30)*\loadSize] table[x=agent2.v, y=agent2.s]{pic/pgf/ropeMas_referenceTracking_4.dat};
					\addplot [color=colA3, mark size=sqrt(1.40)*\loadSize] table[x=agent3.v, y=agent3.s]{pic/pgf/ropeMas_referenceTracking_4.dat};
					\addplot [color=colA4, mark size=sqrt(0.80)*\loadSize] table[x=agent4.v, y=agent4.s]{pic/pgf/ropeMas_referenceTracking_4.dat};
					\myDelta{20}{2}
				\end{axis}
			\end{tikzpicture}%
		\end{tabular}}
		\caption{\unboldmath
			Closed-loop reference behavior of the platoon for the piecewise ramp-like reference signal $r(t) = \alpha_{0} + \alpha_{1} t$.
			\textbf{Top:} The positions $w^{i \gamma}_{1}(t)$ of the loads are plotted.
			\textbf{Bottom:} Snapshots of the rope positions $v^{1 1}(s, t)$ \eqref{pgf:a1}, $v^{1 2}(s, t)$ \eqref{pgf:a2}, $v^{2 1}(s, t)$ \eqref{pgf:a3}, $v^{2 2}(s, t)$ \eqref{pgf:a4} at $t\in\{1, 1.5\}$, $t\in\{9.5, 10\}$, $t\in\{12, 12.5\}$, $t\in\{29.5, 30\}$ with the reference input $r(t)$ \eqref{pgf:a0} as vertical line.}
		\label{pic:referenceTracking.rope}%
\end{figure}
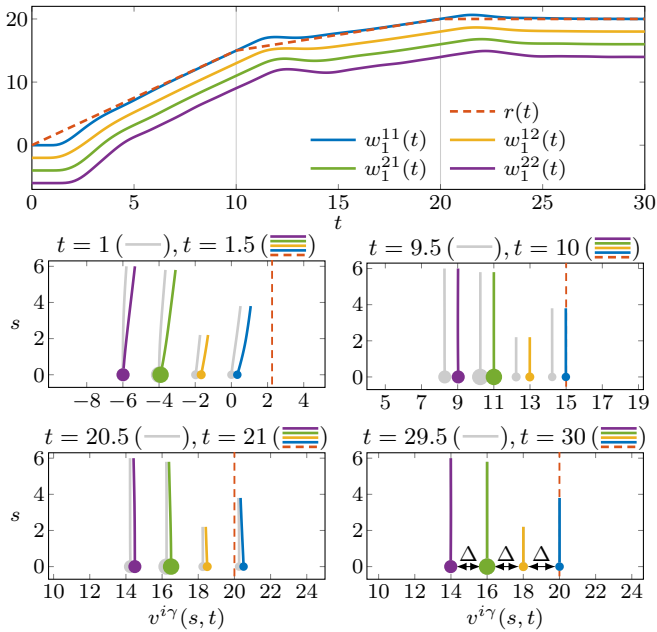}%
{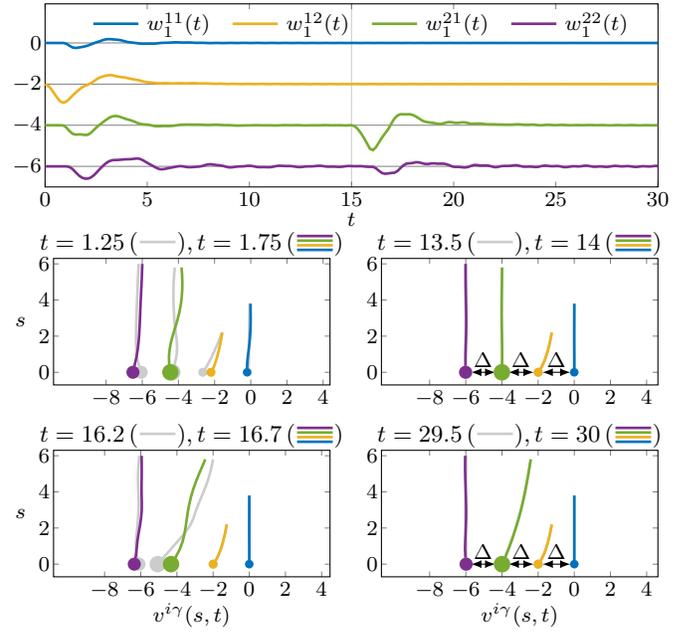
\begin{figure}[t]\centering%
		\begin{tikzpicture}%
			\begin{axis}[plotOpts, xmin=0, xmax=30, xtick={0, 5, ..., 30}, xlabel={$t$}, ymin=-7, ymax=1.9,
				major grid style={line width=.3pt,draw=gray}, ymajorgrids=true,]
				\addplot [opacity=0.2, samples=2, thin] coordinates {(15, -10) (15, 10)};
				\addplot [color=colA4] table[x=t, y=agent4.w1]{pic/pgf/ropeMas_disturbanceRejection.dat};
				\addplot [color=colA3] table[x=t, y=agent3.w1]{pic/pgf/ropeMas_disturbanceRejection.dat};
				\addplot [color=colA2] table[x=t, y=agent2.w1]{pic/pgf/ropeMas_disturbanceRejection.dat};
				\addplot [color=colA1] table[x=t, y=agent1.w1]{pic/pgf/ropeMas_disturbanceRejection.dat};
				%
				\node[anchor=north] at (axis cs:15, 2.05) {\parbox{7cm}{\centering\footnotesize\def\arraystretch{1.2}\arraycolsep=1.4pt
						$\begin{array}{rr@{\quad}rr@{\quad}rr@{\quad}rr}
							\ref{pgf:a1} 	&w_{1}^{1 1}(t)		&\ref{pgf:a2} 	&w_{1}^{1 2}(t)&
							\ref{pgf:a3} 	&w_{1}^{2 1}(t) 	&\ref{pgf:a4} 	&w_{1}^{2 2}(t)
						\end{array}$}};
			\end{axis}
		\end{tikzpicture}\\[-1ex]%
		{\renewcommand{\arraystretch}{0}
		\setlength{\tabcolsep}{1pt}
		\begin{tabular}{ll}
			\begin{tikzpicture}%
				\begin{axis}[snapshotOpts, title={$t = 1.25 \earlySnap, t=1.75 \lateSnapD$}, ylabel={$s$}]
					\addplot [color=ghostColor, mark size=sqrt(0.25)*\loadSize, opacity=1] table[x=agent1.v, y=agent1.s]{pic/pgf/ropeMas_disturbanceRejection_5.dat};
					\addplot [color=ghostColor, mark size=sqrt(0.30)*\loadSize, opacity=1] table[x=agent2.v, y=agent2.s]{pic/pgf/ropeMas_disturbanceRejection_5.dat};
					\addplot [color=ghostColor, mark size=sqrt(1.40)*\loadSize, opacity=1] table[x=agent3.v, y=agent3.s]{pic/pgf/ropeMas_disturbanceRejection_5.dat};
					\addplot [color=ghostColor, mark size=sqrt(0.80)*\loadSize, opacity=1] table[x=agent4.v, y=agent4.s]{pic/pgf/ropeMas_disturbanceRejection_5.dat};
					\addplot [color=colA1, mark size=sqrt(0.25)*\loadSize] table[x=agent1.v, y=agent1.s]{pic/pgf/ropeMas_disturbanceRejection_1.dat};
					\addplot [color=colA2, mark size=sqrt(0.30)*\loadSize] table[x=agent2.v, y=agent2.s]{pic/pgf/ropeMas_disturbanceRejection_1.dat};
					\addplot [color=colA3, mark size=sqrt(1.40)*\loadSize] table[x=agent3.v, y=agent3.s]{pic/pgf/ropeMas_disturbanceRejection_1.dat};
					\addplot [color=colA4, mark size=sqrt(0.80)*\loadSize] table[x=agent4.v, y=agent4.s]{pic/pgf/ropeMas_disturbanceRejection_1.dat};
				\end{axis}
			\end{tikzpicture}%
			&
			\begin{tikzpicture}%
				\begin{axis}[snapshotOpts, title={$t = 13.5 \earlySnap, t=14 \lateSnapD$}]
					\addplot [color=ghostColor, mark size=sqrt(0.25)*\loadSize, opacity=1] table[x=agent1.v, y=agent1.s]{pic/pgf/ropeMas_disturbanceRejection_6.dat};
					\addplot [color=ghostColor, mark size=sqrt(0.30)*\loadSize, opacity=1] table[x=agent2.v, y=agent2.s]{pic/pgf/ropeMas_disturbanceRejection_6.dat};
					\addplot [color=ghostColor, mark size=sqrt(1.40)*\loadSize, opacity=1] table[x=agent3.v, y=agent3.s]{pic/pgf/ropeMas_disturbanceRejection_6.dat};
					\addplot [color=ghostColor, mark size=sqrt(0.80)*\loadSize, opacity=1] table[x=agent4.v, y=agent4.s]{pic/pgf/ropeMas_disturbanceRejection_6.dat};
					\addplot [color=colA1, mark size=sqrt(0.25)*\loadSize] table[x=agent1.v, y=agent1.s]{pic/pgf/ropeMas_disturbanceRejection_2.dat};
					\addplot [color=colA2, mark size=sqrt(0.30)*\loadSize] table[x=agent2.v, y=agent2.s]{pic/pgf/ropeMas_disturbanceRejection_2.dat};
					\addplot [color=colA3, mark size=sqrt(1.40)*\loadSize] table[x=agent3.v, y=agent3.s]{pic/pgf/ropeMas_disturbanceRejection_2.dat};
					\addplot [color=colA4, mark size=sqrt(0.80)*\loadSize] table[x=agent4.v, y=agent4.s]{pic/pgf/ropeMas_disturbanceRejection_2.dat};
					\myDelta{0}{2}
				\end{axis}
			\end{tikzpicture}%
			\\
			\begin{tikzpicture}%
				\begin{axis}[snapshotOpts, title={$t = 16.2 \earlySnap, t=16.7 \lateSnapD$}, xlabel={$v^{i \gamma}(s, t)$}, ylabel={$s$}]
					\addplot [color=ghostColor, mark size=sqrt(0.25)*\loadSize, opacity=1] table[x=agent1.v, y=agent1.s]{pic/pgf/ropeMas_disturbanceRejection_7.dat};
					\addplot [color=ghostColor, mark size=sqrt(0.30)*\loadSize, opacity=1] table[x=agent2.v, y=agent2.s]{pic/pgf/ropeMas_disturbanceRejection_7.dat};
					\addplot [color=ghostColor, mark size=sqrt(1.40)*\loadSize, opacity=1] table[x=agent3.v, y=agent3.s]{pic/pgf/ropeMas_disturbanceRejection_7.dat};
					\addplot [color=ghostColor, mark size=sqrt(0.80)*\loadSize, opacity=1] table[x=agent4.v, y=agent4.s]{pic/pgf/ropeMas_disturbanceRejection_7.dat};
					\addplot [color=colA1, mark size=sqrt(0.25)*\loadSize] table[x=agent1.v, y=agent1.s]{pic/pgf/ropeMas_disturbanceRejection_3.dat};
					\addplot [color=colA2, mark size=sqrt(0.30)*\loadSize] table[x=agent2.v, y=agent2.s]{pic/pgf/ropeMas_disturbanceRejection_3.dat};
					\addplot [color=colA3, mark size=sqrt(1.40)*\loadSize] table[x=agent3.v, y=agent3.s]{pic/pgf/ropeMas_disturbanceRejection_3.dat};
					\addplot [color=colA4, mark size=sqrt(0.80)*\loadSize] table[x=agent4.v, y=agent4.s]{pic/pgf/ropeMas_disturbanceRejection_3.dat};
				\end{axis}
			\end{tikzpicture}%
			&
			\begin{tikzpicture}%
				\begin{axis}[snapshotOpts, title={$t = 29.5 \earlySnap, t=30 \lateSnapD$}, xlabel={$v^{i \gamma}(s, t)$}]
					\addplot [color=ghostColor, mark size=sqrt(0.25)*\loadSize, opacity=1] table[x=agent1.v, y=agent1.s]{pic/pgf/ropeMas_disturbanceRejection_8.dat};
					\addplot [color=ghostColor, mark size=sqrt(0.30)*\loadSize, opacity=1] table[x=agent2.v, y=agent2.s]{pic/pgf/ropeMas_disturbanceRejection_8.dat};
					\addplot [color=ghostColor, mark size=sqrt(1.40)*\loadSize, opacity=1] table[x=agent3.v, y=agent3.s]{pic/pgf/ropeMas_disturbanceRejection_8.dat};
					\addplot [color=ghostColor, mark size=sqrt(0.80)*\loadSize, opacity=1] table[x=agent4.v, y=agent4.s]{pic/pgf/ropeMas_disturbanceRejection_8.dat};
					\addplot [color=colA1, mark size=sqrt(0.25)*\loadSize] table[x=agent1.v, y=agent1.s]{pic/pgf/ropeMas_disturbanceRejection_4.dat};
					\addplot [color=colA2, mark size=sqrt(0.30)*\loadSize] table[x=agent2.v, y=agent2.s]{pic/pgf/ropeMas_disturbanceRejection_4.dat};
					\addplot [color=colA3, mark size=sqrt(1.40)*\loadSize] table[x=agent3.v, y=agent3.s]{pic/pgf/ropeMas_disturbanceRejection_4.dat};
					\addplot [color=colA4, mark size=sqrt(0.80)*\loadSize] table[x=agent4.v, y=agent4.s]{pic/pgf/ropeMas_disturbanceRejection_4.dat};
					\myDelta{0}{2}
				\end{axis}
			\end{tikzpicture}%
		\end{tabular}
		}%
		\caption{\unboldmath Closed-loop disturbance behavior of the platoon for the disturbances $d^{1 2}(t) = -8 \sigma(t)$ and $d^{2 1}(t) = -5 \sigma(t - 15)$ acting on the agents $\nu^{1 2}$ and $\nu^{2 1}$.
			\textbf{Top:} The positions $w^{i \gamma}_{1}(t)$ of the loads are plotted.
			\textbf{Bottom:} Snapshots of the  rope positions $v^{1 1}(s, t)$ \eqref{pgf:a1}, $v^{1 2}(s, t)$ \eqref{pgf:a2}, $v^{2 1}(s, t)$ \eqref{pgf:a3}, $v^{2 2}(s, t)$ \eqref{pgf:a4} at $t\in\{1.25, 1.75\}$, $t\in\{13.5, 14\}$, $t\in\{16.2, 16.7\}$, $t\in\{29.5, 30\}$.}
		\label{pic:disturbanceRejection.rope}%
\end{figure}}%
The wave dynamics of the uncertain followers are approximated in space with $21$ grid points by using differentiation and integration matrices, which are provided by the Chebfun toolbox (see \cite{ChebfunGuide14,Trefethen2000spectral}).
For ease of implementation, the latter is done after normalizing the spatial domain of all ropes to $1$.
To simulate the resulting lumped system in time, the Matlab solver \texttt{lsim} is used with the step size $10^{-2}$.
The ICs of the agents and the networked controller are such that the followers stand still in the formation with $v^{i \gamma}(s, 0) = \delta^{i \gamma}$, $\dot{v}^{i \gamma}(s, 0) = 0$, $i = 1, 2$, $\gamma = 1, 2$.

The reference input $r(t) = 1.5 t$ in the interval $t\in [0, 10)$ results in a fast movement, while the flatter ramp $r(t) = 10 + 0.5 t$, $t \in [10, 20)$, initiates a slow down, so that the new setpoint $r(t) = 20$, $t \geq 20$, is reached without a large overshoot.
The resulting tracking behavior is shown in Fig.~\ref{pic:referenceTracking.rope}.
After an acceleration phase (see the snapshot at $t=1.5$), the agents move with constant speed and the specified distances $\Delta = 2$ (see $t=10$).
As the leader changes the desired velocity to zero at $t=20$, overshooting occurs so that at $t=21$ the followers decelerate and then asymptotically approach the final formation at $t=30$.
Obviously, the networked controller ensures that the informed agent $\nu^{1 1}$ and the uninformed agents $\nu^{1 2}, \nu^{2 1}, \nu^{2 2}$ track the output $r(t)$ of the leader $\nu^{0}$.

The rejection of the disturbances $d^{1 2}(t) = -8$ and $d^{2 1}(t) = 0$, $t \in [0, 15)$, $d^{2 1}(t) = -5$, $t \geq 15$, is illustrated in Fig.~\ref{pic:disturbanceRejection.rope}.
Although these disturbances only act on the agents $\nu^{1 2}$ and $\nu^{2 1}$, all agents are influenced by $d^{1 2}(t)$ and $d^{2 1}(t)$ due to the network communication so the disturbance behavior is cooperative.
The bending of the ropes in the steady state (see Fig.~\ref{pic:disturbanceRejection.rope} for $t=30$) shows the compensation of the disturbances ensuring a vanishing tracking error for the loads.

The simulations demonstrate that robust cooperative output regulation for a heterogeneous MAS consisting of distributed-parameter agents is achieved in the presence of model uncertainties.
Thereby, the presented networked controller allows the formation control of multiple heavy ropes.
This result directly extends to the cooperative transportation using two ropes for each load giving rise to multi-input systems (see \cite{Irscheid2019flatness2}).
The corresponding networked controller can then also be designed by making use of the results of this paper.
%
%
%
\section{Concluding Remarks}{\acceptancenotice\copyrightnotice}%
Similar to \cite{Deu22robustMas}, the results of this paper are also directly applicable for the \textit{leaderless} output synchronization problem.
In order to achieve a cooperative output feedback controller, a state observer can be used.
For this, there are results available in the literature, see \cite{Deu21waveOde,DeuGehKe19,RedaudAuriol2021,Demetriou2004}.
However, an open research topic is the observer design using solely relative measurements.
If the output to be controlled is not available for measurement, then an extension of the cooperative reference observer can be considered.
This requires the generalization of the observer-based feed-forward regulation in \cite{Deu21parabMasOr, su2012cooperative} to hyperbolic PIDE--ODE agents.
However, the resulting tracking is then no longer robust w.r.t. model uncertainty.
Another promising research topic is the combination of the presented results with the delay-adaptive methods in \cite{ZhuKrstic2020delay} in order to achieve robust cooperative output regulation for ODE agents with unknown delays.
Future work could investigate not only the design of the controller but also of the network topology to improve the synchronization dynamics.
Thereby, in particular the choice which agents serves as a leader is of interest.
%
%
%
\appendices
\section{}
\subsubsection{Proof of Lemma \ref{lem:decouplingEquations.cooperative}}\label{sec:appendix.decouplingEquations}{\acceptancenotice\copyrightnotice}%
By recalling $\tilde{S} = I_{\bar{n}_{-}} \otimes S$, $\tilde{B}_{y} = I_{\bar{n}_{-}} \otimes b_{y}$ and using $\Pi_{x,j}^{i}(z) = (e_{j}^{\top} \Otimes I_{n_v}) \Pi_{x}^{i}(z)$ and $\Pi_{w,j}^{i} = (e_{j}^{\top} \Otimes I_{n_v}) \Pi_{w}^{i}$, the decoupling equations \eqref{eq:decouplingEquations.cooperative} are split into
\begin{subequations}\label{eq:decouplingEquations.cooperative.split}
	\begin{align}
		S \Pi_{w,j}^{i} - \Pi_{w,j}^{i} \tilde{F}_{w}^{i} &= - b_{y} e_{j}^\top \tilde{C}_{w}^{i}\label{eq:decouplingEquations.cooperative.split.sylvester}\\
		(\Pi_{x,j}^{i}(z) \Lambda^{i}(z))' + S \Pi_{x,j}^{i}(z) &= -b_{y} e_{j}^\top \tilde{C}_{x}^{i}(z)\label{eq:decouplingEquations.cooperative.split.ode}\\
		\Pi_{x,j}^{i}(0) \Lambda^{i}(0) (E_{+}^{i} Q_{0}^{i} \!+\! E_{-}^{i}) &- \int_{0}^{1} \Pi_{x,j}^{i}(\zeta) \tilde{A}_{0}^{i}(\zeta) \d\zeta\\
		= \Pi_{w,j}^{i} B_{w}^{i} &- b_{y} e_{j}^\top C_{x,0}^{i} (E_{+}^{i} Q_{0}^{i} \!+\! E_{-}^{i})\nonumber\\
		\Pi_{x,j}^{i}(1) \Lambda^{i}(1) E_{+}^{i} &= b_{y} e_{j}^\top C_{x,1}^{i} E_{+}^{i}
	\end{align}
\end{subequations}
with $j=1, \ldots, \bar{n}_{-}$.
Next, the generalized left-eigenvectors $\nu_{k l}^{\top}$ of $S$ related to the eigenvalues $\mu_{k}$, $k=1, \ldots, \bar{k}$, with the dimension $l=1, \ldots, \bar{l}_{k}$ of the corresponding Jordan blocks and
$\nu_{k l}^{\top} S = \nu_{k l}^{\top} \mu_{k} + \nu_{k (l-1)}^\top$,
$\nu_{k 0}^\top = 0$
are utilized.
For brevity, $\pi_{l}^{j \top}\!(z) = \nu_{k l}^{\top} \Pi_{x,j}^{i}(z)$, $\pi_{w,l}^{j \top} = \nu_{k l}^{\top} \Pi_{w,j}^{i}$ are introduced without indices that are irrelevant in this context.
Then, multiplying \eqref{eq:decouplingEquations.cooperative.split} by $\nu_{k l}^{\top}$ from the left side and rearranging \eqref{eq:decouplingEquations.cooperative.split.sylvester}, \eqref{eq:decouplingEquations.cooperative.split.ode} leads to
\begin{subequations}\label{eq:decouplingEquations.cooperative.modal}
	\begin{align}
		&\pi_{w,l}^{j \top} = - (\nu_{k l}^{\top} b_{y} e_{j}^\top \tilde{C}_{w}^{i} + \pi_{w,l-1}^{j \top}) (\mu_{k} I - \tilde{F}_{w}^{i})^{-1} \label{eq:decouplingEquations.cooperative.modal.sylvester}\\
		&\pi_{l}^{\prime \top}\!(z) = \pi_{l}^{j \top}\!(z) M_{0}(z) + c_{0}^{j \top}\!(z) \label{eq:decouplingEquations.cooperative.modal.ode}\\
		&\pi_{l}^{j \top}\!(0) \Lambda^{i}(0) (E_{+}^{i} Q_{0}^{i} \!+\! E_{-}^{i}) \!-\! \int_{0}^{1} \! \pi_{l}^{j \top}\!(\zeta) \tilde{A}_{0}^{i}(\zeta) \d\zeta \!=\! c_{1}^{j \top} \! \label{eq:decouplingEquations.cooperative.modal.bc0}\\
		&\pi_{l}^{j \top}\!(1) \Lambda^{i}(1) E_{+}^{i} = \nu_{k l}^{\top} b_{y} e_{j}^{\top} C_{x,1}^{i}  E_{+}^{i} \label{eq:decouplingEquations.cooperative.modal.bc1}
	\end{align}
\end{subequations}
with $M_{0}(z) = - (\Lambda^{i\prime}(z) + \mu_{k} I) (\Lambda^{i}(z))^{-1}$,
$c_{0}^{j \top}\!(z) = - (\nu_{k l}^{\top} b_{y} e_{j}^\top \tilde{C}_{x}^{i}(z) + \pi_{l-1}^{j \top}(z)) (\Lambda^{i}(z))^{-1}$ and $c_{1}^{j \top} = \pi_{w,l}^{j \top} B_{w}^{i} - \nu_{k l}^{\top} b_{y} e_{j}^\top C_{x,0}^{i} (E_{+}^{i} Q^{i}_{0} \!+\! E_{-}^{i})$.
Note that $\mu_{k} I - \tilde{F}_{w}^{i}$ in \eqref{eq:decouplingEquations.cooperative.modal.sylvester} is a nonsingular matrix because $\tilde{F}_{w}^{i}$ is Hurwitz and $\mu_{k} \notin \sigma(\tilde{F}_{w}^{i})$.
The general solution of the ODE \eqref{eq:decouplingEquations.cooperative.modal.ode} is
$\pi_{l}^{j \top}\!(z) = \pi_{l}^{j \top}\!(1) \Psi_{k}^{i}(z, 1) + \int_{1}^{z} c_{0}^{j \top}\!(\zeta) \Psi_{k}^{i}(z, \zeta) \d\zeta$
with the \textit{fundamental matrix}
$\Psi_{k}^{i}(z, \zeta) = \e^{\int_{\zeta}^{z} M_{0}(\eta) \d\eta} = \Lambda^{i}(\zeta) (\Lambda^{i}(z))^{-1} \e^{-\mu_{k} \! \int_{\zeta}^{z} (\Lambda^{i}(\eta))^{-1} \d\eta}$
since $M_{0}(z)$ is diagonal.
From \eqref{eq:decouplingEquations.cooperative.modal.bc1} follows $\pi_{l}^{j \top}\!(1) E_{+}^{i} = \nu_{k l}^{\top} b_{y} e_{j}^{\top} C_{x,1}^{i} (\Lambda^{i}(1))^{-1} E_{+}^{i}$ and thus $\pi_{l}^{j \top}\!(1) = \pi_{l}^{j \top}\!(1) E_{-}^{i} E_{-}^{i \top} + \nu_{k l}^{\top} b_{y} e_{j}^{\top} C_{x,1}^{i} (\Lambda^{i}(1))^{-1} E_{+}^{i} E_{+}^{i \top}$.
Inserting this into the solution of the ODE \eqref{eq:decouplingEquations.cooperative.modal.ode} gives 
\begin{align}\label{eq:decouplingEquations.cooperative.modal.solution.pi1}
	\pi_{l}^{j \top}\!(z) &= \pi_{l}^{j \top}\!(1) E_{-}^{i} E_{-}^{i \top} \Psi_{k}^{i}(z, 1) + c_{2}^{j \top}\!(z)
\end{align}
with $c_{2}^{j \top}\!(z) = \nu_{k l}^{\top} b_{y} e_{j}^{\top} C_{x,1}^{i} (\Lambda^{i}(1))^{-1} E_{+}^{i} E_{+}^{i \top} \Psi_{k}^{i}(z, 1) + \int_{1}^{z} c_{0}^{j \top}\!(\zeta) \Psi_{k}^{i}(z, \zeta) \d\zeta$.
Making use of \eqref{eq:decouplingEquations.cooperative.modal.solution.pi1} in \eqref{eq:decouplingEquations.cooperative.modal.bc0} yields
$\pi_{l}^{j \top}\!(1) E_{-}^{i} M_{1} = c_{3}^{j \top}$
with
$M_{1} = E_{-}^{i \top} \! \big(\Psi_{k}^{i}(0, 1) \Lambda^{i}(0) E_{-}^{i} \!-\! \int_{0}^{1} \! \Psi_{k}^{i}(\zeta, 1) \tilde{A}_{0}^{i}(\zeta) \d\zeta \big)$,
$c_{3}^{j \top} = c_{1}^{j \top} \!-\! c_{2}^{j \top}\!(0) \Lambda^{i}(0) (E_{+}^{i} Q^{i}_{0}\!+\! E_{-}^{i}) \!+\! \int_{0}^{1}\! c_{2}^{j \top}\!(\zeta) \tilde{A}_{0}^{i}(\zeta) \d\zeta$.
Note that $M_{1}$ is nonsingular and lower triangular because $\Psi_{k}^{i}(z, \zeta)$ and $\Lambda^{i}(z)$ are diagonal as well as nonsingular while $E_{-}^{i \top} \Psi_{k}^{i}(\zeta, 1) \tilde{A}_{0}^{i}(\zeta)$ is strictly lower triangular due to the same property of $\tilde{A}_{0-}^{i}(z)$ (see Section \ref{sec:localStabilization.backstepping}).
Hence, $\pi_{l}^{j \top}\!(1) E_{-}^{i} = c_{3}^{j \top} M_{1}^{-1}$ holds, which leads with \eqref{eq:decouplingEquations.cooperative.modal.solution.pi1} to
$\pi_{l}^{j \top}\!(z) = c_{3}^{j \top} M_{1}^{-1} E_{-}^{i \top} \Psi_{k}^{i}(z, 1) + c_{2}^{j \top}\!(z)$. 
Then, $\Pi_{x,j}^{i}$ and $\Pi_{w,j}^{i}$ follow from the latter and \eqref{eq:decouplingEquations.cooperative.modal.sylvester} by using the inverse of the transformation matrix in the corresponding Jordan canonical form resulting from $\nu_{k l}$.
Composing $\Pi_{x}^{i}(z) = \col(\Pi_{x,1}^{i}(z), \ldots, \Pi_{x,\bar{n}_{-}}(z))$ and $\Pi_{w}^{i} = \col(\Pi_{w,1}^{i}, \ldots, \Pi_{w,\bar{n}_{-}}^{i})$ yields the result.
The elements of $\Pi_{x}^{i}$ are piecewise $C^{1}$ because the integrand of $c_{2}^{j \top}$ can contain $\delta$-functions, which are caused by a pointwise in-domain defined output. Hence, the boundedly invertible transformation \eqref{eq:decoupling.extendedRegulatorEquations} exists.
\proofEnd

\subsubsection{Proof of Lemma \ref{lem:controllability}}\label{sec:appendix.controllability}{\acceptancenotice\copyrightnotice}
First, the transfer matrix from $\bm{u}$ to $\bm{y}$ of \eqref{eq:agent.local.nominal} and its numerator $N^{i}(s)$ are derived. Similar to Section \ref{sec:controller.design.localStabilization}, the local decoupling and backstepping transformations \eqref{eq:transformation.local.decoupling}, \eqref{eq:transformation.backstepping} are applied, but the terms appearing at the actuated boundary condition are not eliminated by state feedback. Then \eqref{eq:agent.local.backstepping} is obtained, but with
$\tilde{\bm{\varepsilon}}_{-}(1, t) = \bm{u}(t) - (E_{-}^{i \top} - Q_{1}^{i} E_{+}^{i \top}) \Sigma^{i}(1) \bm{w}(t) + Q_{1}^{i} E_{+}^{i \top} \mathcal{T}_{i}^{-1}[\tilde{\bm{\varepsilon}}(t)](1) - \int_{0}^{1} E_{+}^{i \top} K^{i}(1, \zeta) \mathcal{T}_{i}^{-1}[\tilde{\bm{\varepsilon}}(t)](\zeta) \d\zeta$
instead of \eqref{eq:agent.local.backstepping.bc1}.
By making use of the solution $\tilde{\bm{\varepsilon}}(z, t) = \check{\tilde{\bm{\varepsilon}}}(z) \e^{s t}$, $\bm{w}(t) = \check{\bm{w}} \e^{s t}$, $\bm{y}(t) = \check{\bm{y}} \e^{s t}$ for $\bm{u}(t) = \check{\bm{u}} \e^{s t}$, $s \in \mathbb{C}$, $t \geq 0$ (see \cite{Zw04}) one obtains
\begin{subequations}\label{eq:agent.local.nofeedback.laplace}
	\begin{align}
		\d_z \check{\tilde{\bm{\varepsilon}}}(z) &= (\Lambda^{i}(z))^{-1} \big(s \check{\tilde{\bm{\varepsilon}}}(z) - \tilde{A}_{0}^{i}(z) \check{\tilde{\bm{\varepsilon}}}_{-}(0) \big) \label{eq:agent.local.nofeedback.laplace.pde}\\
		\check{\tilde{\bm{\varepsilon}}}_{+}(0) &= Q_{0}^{i} \check{\tilde{\bm{\varepsilon}}}_{-}(0)\label{eq:agent.local.nofeedback.laplace.bc0}\\
		\check{\tilde{\bm{\varepsilon}}}_{-}(1) &= \check{\bm{u}} - ( E_{-}^{i \top} \! - \! Q_{1}^{i} E_{+}^{i \top} )\Sigma^{i}(1) \check{\bm{w}} + Q_{1}^{i} E_{+}^{i \top} \mathcal{T}_{i}^{-1}[\check{\tilde{\bm{\varepsilon}}}](1)\nonumber\\
		&\quad - \int_{0}^{1} E_{+}^{i \top} K^{i}(1, \zeta) \mathcal{T}_{i}^{-1}[\check{\tilde{\bm{\varepsilon}}}](\zeta) \d\zeta \label{eq:agent.local.nofeedback.laplace.bc1}\\
		\check{\bm{w}} &= (s I - \tilde{F}_{w}^{i})^{-1} B_{w}^{i} \check{\tilde{\bm{\varepsilon}}}_{-}(0)\label{eq:agent.local.nofeedback.laplace.ode}\\
		\check{\bm{y}} &= \tilde{\mathcal{C}}^{i}_{x} [\check{\tilde{\bm{\varepsilon}}}] + \tilde{C}_{w}^{i} \check{\bm{w}} .\label{eq:agent.local.nofeedback.laplace.output}
	\end{align}
\end{subequations}
The \textit{fundamental matrix} $\Phi^{i}$ related to \eqref{eq:agent.local.nofeedback.laplace.pde} is given in Lemma \ref{lem:controllability}, thus 
$\check{\tilde{\bm{\varepsilon}}}(z) = \Phi^{i}(z, 0, s) \check{\tilde{\bm{\varepsilon}}}(0) - \int_{0}^{z} \Phi^{i}(z, \zeta, s) (\Lambda^{i}(\zeta))^{-1} \tilde{A}_{0}^{i}(\zeta)\d\zeta \check{\tilde{\bm{\varepsilon}}}_{-}(0)$
holds.
This leads with \eqref{eq:agent.local.nofeedback.laplace.bc0} and $M^{i}(z, s)$ given in Lemma \ref{lem:controllability} to
$\check{\tilde{\bm{\varepsilon}}}(z) = M^{i}(z, s) \check{\tilde{\bm{\varepsilon}}}_{-}(0)$.
Inserting the latter and \eqref{eq:agent.local.nofeedback.laplace.ode} into \eqref{eq:agent.local.nofeedback.laplace.bc1} gives
$\check{\tilde{\bm{\varepsilon}}}(z) = M^{i}(z, s) (D^{i}(s))^{-1} \check{\bm{u}}$
with $D^{i}(s) \in \mathbb{C}^{n_{-}^{i} \times n_{-}^{i}}$.
Moreover, plugging $\check{\tilde{\bm{\varepsilon}}}(z) = M^{i}(z, s) \check{\tilde{\bm{\varepsilon}}}_{-}(0)$ into \eqref{eq:agent.local.nofeedback.laplace.ode} with $E_{-}^{i \top} M(0, s) = I$ as well as expressing the matrix inverse by means of the adjugate and the determinant gives
$\check{\bm{w}} = \text{adj}(s I - \tilde{F}_{w}^{i}) (\det(s I - \tilde{F}_{w}^{i}))^{-1} B_{w}^{i} \check{\tilde{\bm{\varepsilon}}}_{-}(0)$.
Inserting the latter and $\check{\tilde{\bm{\varepsilon}}}(z) = M^{i}(z, s) (D^{i}(s))^{-1} \check{\bm{u}}$ into \eqref{eq:agent.local.nofeedback.laplace.output} as well as factoring out $(\det(s I \!-\! \tilde{F}_{w}^{i}))^{-1}$ yields the numerator $N^{i}(s)$ described by \eqref{eq:agent.numerator}.

Next, the controllability of $(\tilde{S}, B_{e}^{i})$ is investigated.
Due to $\tilde{S} = I_{\bar{n}_{-}} \otimes S$, the left-eigenvectors $\bar{\nu}_{k m}^{\top}$, $m = 1, \ldots, \bar{n}_{-}$, of $\tilde{S}$, which are related to the eigenvalues $\mu_{k}$, can be expressed by utilizing the left-eigenvectors $\nu_{k 1}^{\top}$ of $S$ (see Appendix \ref{sec:appendix.decouplingEquations}) as $\bar{\nu}_{k m}^{\top} = \alpha_{k m}^{\top} \otimes \nu_{k 1}^\top$ with $\alpha_{k m}^{\top} = [\alpha_{k m}^{j}] \in \mathbb{R}^{1 \times \bar{n}_{-}}$, $\alpha_{k m}^{\top} \neq 0^{\top}$, $k = 1, \ldots, \bar{k}$.
According to the PBH eigenvector test (see e.g. \cite[Th. 6.2-5]{KL80}) is $(\tilde{S}, B_{e}^{i})$ controllable if $\bar{\nu}_{k m}^{\top} B_{e}^{i} \neq 0$ for all $\bar{\nu}_{k m}^{\top}$.
By making use of the definition of $B_{e}^{i}$ in Section \ref{sec:controller.simultaneousStabilization.decoupling} and recalling $\Pi_{x}^{i}(z) = \col(\Pi_{x,1}^{i}(z), \ldots, \Pi_{x,\bar{n}_{-}}^{i}(z))$, $\pi_{l}^{j \top}\!(z) = \nu_{k l}^{\top} \Pi_{x,j}^{i}(z)$ (see Appendix \ref{sec:appendix.decouplingEquations}) follows
\begin{align}\label{eq:controllability.nub}
	\bar{\nu}_{k m}^{\top} B_{e}^{i} = \sum_{j=1}^{\bar{n}_{-}} \alpha_{k m}^{j} \big( \pi_{1}^{j \top}\!(1) E_{-}^{i} E_{-}^{i \top} \Lambda^{i}(1) E_{-}^{i} \!-\! \nu_{k 1}^{\top} b_{y} C_{x,1}^{i} E_{-}^{i} \big) .
\end{align}
In view of $\pi_{l}^{j \top}\!(1) E_{-}^{i} M_{1} = c_{3}^{j \top}$ (see Appendix \ref{sec:appendix.controllability}) holds $\pi_{l}^{j \top}\!(1) E_{-}^{i} = c_{3}^{j \top} M_{1}^{-1}$.
Note that the relation
$\Psi_{k}^{i}(z, \zeta) = \Lambda^{i}(\zeta) (\Lambda^{i}(z))^{-1} \Phi^{i}(\zeta, z, \mu_{k})$
of the fundamental matrices $\Psi_{k}^{i}$ and $\Phi^{i}$, which are defined in Appendix \ref{sec:appendix.decouplingEquations} and Lemma \ref{lem:controllability}, holds.
With the latter, \eqref{eq:decouplingEquations.cooperative.modal.sylvester}, $l=1$ and the definitions of $c_{3}^{j \top}$, $c_{2}^{j \top}\!(z)$, $c_{1}^{j \top}$, $c_{0}^{j \top}\!(z)$ (see Appendix \ref{sec:appendix.decouplingEquations}) follows
$c_{2}^{j \top}\!(z) = \nu_{k 1}^{\top} b_{y} e_{j}^{\top} \big( C_{x,1}^{i} E_{+}^{i} E_{+}^{i \top} (\Lambda^{i}(z))^{-1} \Phi^{i}(1, z, \mu_{k}) - \int_{1}^{z} \tilde{C}_{x}^{i}(\zeta) (\Lambda^{i}(z))^{-1} \Phi^{i}(\zeta, z, \mu_{k}) \d\zeta \big)$,
$c_{3}^{j \top} \!=\! - \nu_{k 1}^{\top} b_{y} e_{j}^\top \! \big(\tilde{C}_{w}^{i} (\mu_{k} I \!-\! \tilde{F}_{w}^{i})^{-1} B_{w}^{i} + C_{x,0}^{i} (E_{+}^{i} Q^{i}_{0} \!+\! E_{-}^{i}) \big) - c_{2}^{j \top}\!(0) \Lambda^{i}(0) (E_{+}^{i} Q^{i}_{0}\!+\! E_{-}^{i}) + \int_{0}^{1}\! c_{2}^{j \top}\!(\zeta) \tilde{A}_{0}^{i}(\zeta) \d\zeta$.
On the other hand, applying the formal output operator \eqref{eq:agent.backstepping.output.operator} to $M^{i}(z, s)$ (see Lemma \ref{lem:controllability}), using
$C_{x,1}^{i} I M^{i}(1, \mu_{k})= C_{x,1}^{i} (E_{-}^{i} E_{-}^{i \top} \!+\! E_{+}^{i} E_{+}^{i \top}) M^{i}(1, \mu_{k})$
and changing the order of integration leads to
\begin{align}
	&\tilde{\mathcal{C}}^{i}_{x}[M(\mu_{k})] =
	C_{x,0}^{i} (E^{i}_{-} \!+\! E^{i}_{+} Q_{0}^{i}) 
	+ C_{x,1}^{i} E_{-}^{i} E_{-}^{i \top} M^{i}(1, \mu_{k})
	\nonumber\\&~
	+ C_{x,1}^{i} E_{+}^{i} E_{+}^{i \top} \Phi^{i}(1, 0, \mu_{k}) (E^{i}_{-} \!+\! E^{i}_{+} Q_{0}^{i}) 
	\nonumber\\&~
	- \int_{0}^{1} \Big(C_{x,1}^{i} E_{+}^{i} E_{+}^{i \top} \Phi^{i}(1, \zeta, \mu_{k}) (\Lambda^{i}(\zeta))^{-1} \tilde{A}_{0}^{i}(\zeta)
	\nonumber\\&\qquad
	- \tilde{C}_{x}^{i}(\zeta) \Phi^{i}(\zeta, 0, \mu_{k}) (E^{i}_{-} \!+\! E^{i}_{+} Q_{0}^{i})
	\nonumber\\&\qquad
	+ \int_{\zeta}^{1} \tilde{C}_{x}^{i}(\bar{\zeta}) \Phi^{i}(\bar{\zeta}, \zeta, \mu_{k}) (\Lambda^{i}(\zeta))^{-1} \tilde{A}_{0}^{i}(\zeta) \d\bar{\zeta} \Big) \d\zeta . \label{eq:controllability.CopM}
\end{align}
By plugging the previously derived expression for $c_{2}^{j \top}$ into the one of $c_{3}^{j \top}$, exploiting
$\Phi^{i}(\zeta, z, \mu_{k}) (\Lambda^{i}(z))^{-1} = (\Lambda^{i}(z))^{-1} \Phi^{i}(\zeta, z, \mu_{k})$,
which holds as both matrices are diagonal,
and using \eqref{eq:controllability.CopM} one obtains
$c_{3}^{j \top} = - \nu_{k 1}^{\top} b_{y} e_{j}^\top \big( \tilde{C}_{w}^{i} (\mu_{k} I - \tilde{F}_{w}^{i})^{-1} B_{w}^{i} + \tilde{\mathcal{C}}^{i}_{x}[M(\mu_{k})]  \big)
+ \nu_{k 1}^{\top} b_{y} e_{j}^\top C_{x,1}^{i} E_{-}^{i} E_{-}^{i \top} M^{i}(1, \mu_{k})$.
Moreover, one finds
$E_{-}^{i \top} M^{i}(1, \mu_{k}) = (E_{-}^{i \top} \Lambda^{i}(1) E_{-}^{i})^{-1} M_{1}$
in view of  $M^{i}(z, s)$ (see Lemma \ref{lem:controllability}), the definition of $M_{1}$ in Appendix \ref{sec:appendix.decouplingEquations} and
$\Psi_{k}^{i}(z, \zeta) = \Lambda^{i}(\zeta) (\Lambda^{i}(z))^{-1} \Phi^{i}(\zeta, z, \mu_{k})$.
Thus, from \eqref{eq:controllability.nub} follows with the previous preparations
$\bar{\nu}_{k m}^{\top} B_{e}^{i} = - \nu_{k 1}^{\top} b_{y} \sum_{j=1}^{\bar{n}_{-}} \alpha_{k m}^{j} e_{j}^\top \big( \tilde{C}_{w}^{i} (\mu_{k} I - \tilde{F}_{w}^{i})^{-1} B_{w}^{i} 
+ \tilde{\mathcal{C}}^{i}_{x}[M(\mu_{k})]  \big) M_{1}^{-1} E_{-}^{i \top} \Lambda^{i}(1) E_{-}^{i}.$
By considering $\bar{\nu}_{km}^{\top} \tilde{B}_{y} = (\alpha_{k m}^{\top} \otimes \nu_{k 1}^{\top}) (I_{\bar{n}_{-}} \otimes b_{y}) = \nu_{k 1}^{\top} b_{y} \sum_{j=1}^{\bar{n}_{-}} \alpha_{k m}^{j} e_{j}^\top$ and $N^{i}(\mu_{k})$ according to \eqref{eq:agent.numerator} follows
\begin{align}\label{eq:controllability.nub.3}
	\bar{\nu}_{k m}^{\top} B_{e}^{i} &= - \bar{\nu}_{k m}^{\top} \tilde{B}_{y} N^{i}(\mu_{k}) M_{1}^{-1} E_{-}^{i \top} \Lambda^{i}(1) E_{-}^{i}.
\end{align}
Note that $\bar{\nu}_{k m}^{\top} \tilde{B}_{y} \neq 0$ holds due to the PBH eigenvector test because $(S, b_{y})$ is controllable by assumption and thus $(\tilde{S}, \tilde{B}_{y})$ is controllable as $\tilde{S} = I_{\bar{n}_{-}}\otimes S$, $\tilde{B}_{y} = I_{\bar{n}_{-}}\otimes b_{y}$.
Moreover, $E_{-}^{i \top} \Lambda^{i}(1) E_{-}^{i}$ and $M_{1}^{-1}$ are nonsingular (see Appendix \ref{sec:appendix.decouplingEquations}).
Hence, the right-hand side of \eqref{eq:controllability.nub.3} is nonzero if the numerator fulfills $\text{rank}\,N^{i}(\mu_{k}) = \bar{n}_{-}$ for all $\mu_{k} \in \sigma(S)$.
Thus, it follows in this case that $\bar{\nu}_{k m}^{\top} B_{e}^{i} \neq 0$ and due to the PBH eigenvector test the controllability of $(\tilde{S}, B_{e}^{i})$ is ensured.
\proofEnd

\subsubsection{Parameters in Section \ref{sec:robustCooperativeOutputRegulation}}\label{sec:appendix.parameters}{\acceptancenotice\copyrightnotice}%
\textbf{System~(\ref{eq:agent.transformed}): }
$\tilde{\bm{G}}(z) = \bar{\bm{\mathcal{T}}}[\bm{G} - \bar{\bm{\Sigma}} \bm{G}_{w}](z) + \bar{\bm{K}}(z, 0) \bar{\bm{\Lambda}}(0) E_{+}^{i} \bm{G}_{0}$,
$\tilde{\bm{\mathcal{K}}}[\tilde{\bm{\varepsilon}}(t)] = \bar{\bm{Q}}_{1} E_{+}^{i \top} \bar{\bm{\mathcal{T}}}^{-1}[\tilde{\bm{\varepsilon}}(t)](1) - \int_{0}^{1} E_{-}^{i \top} \bar{\bm{K}}(1, \zeta) \bar{\bm{\mathcal{T}}}^{-1}[\tilde{\bm{\varepsilon}}(t)](\zeta) \d\zeta$,
$\tilde{\bm{K}}_{w} = (\bm{Q}_{1} E_{+}^{i \top} - E_{-}^{i \top}) \bar{\bm{\Sigma}}(1)$,
$\tilde{\bar{\bm{C}}}_{w} = \bar{\bm{C}}_{w} + \bar{\bm{\mathcal{C}}}_{x}[\bar{\bm{\Sigma}}]$.

\textbf{System~(\ref{eq:extendedSystem.3}): }
$\underline{S} = I_{N} \otimes \tilde{S}$,
$\underline{\bar{\Lambda}} = \diag(\bar{\bm{\Lambda}})$, 
$\underline{\bar{A}}_{0} = \diag(\mytilde{0.6}{3pt}{\mybar{0.5}{3.3pt}{\bm{A}}}_{0})$,
$\underline{\bar{Q}}_{0} = \diag(\bar{\bm{Q}}_{0})$,
$\underline{E}_{-} = \diag(\underline{E}_{-}^{1}, \ldots, \underline{E}_{-}^{g})$
$\underline{E}_{+} = \diag(\underline{E}_{+}^{1}, \ldots, \underline{E}_{+}^{g})$
$\underline{K}_{\bar{v}} = \diag(I_{N^{1}} \otimes K_{\bar{v}}^{1}, \ldots, I_{N^{g}} \otimes K_{\bar{v}}^{g})$,
$\underline{\mathcal{K}}[\tilde{\varepsilon}(t)] = \diag(\tilde{\bm{\mathcal{K}}}[\tilde{\bm{\varepsilon}}(t)]) + \col(\bm{\mathcal{K}}[0, \underline{\bar{\mathcal{T}}}^{-1}[\tilde{\varepsilon}(t)], 0])$,
$\underline{\tilde{K}}_{w} = \diag(\tilde{\bm{K}}_{w}) + \col(\bm{\mathcal{K}}[0, \underline{\bar{\Sigma}}, 1])$,
$\underline{\bar{\mathcal{T}}} = \diag(\bar{\bm{\mathcal{T}}})$,
$\underline{\bar{\Sigma}} = \diag(\bar{\bm{\Sigma}})$,
$\underline{\tilde{\bar{F}}}_{w} = \diag(\tilde{\bar{\bm{F}}}_{w})$,
$\underline{\bar{B}}_{w} = \diag(\bar{\bm{B}}_{w})$,
$\underline{\bar{\mathcal{C}}}_{x} = \diag(\bar{\bm{\mathcal{C}}}_{x} \bar{\bm{\mathcal{T}}}^{-1})$, 
$\underline{\tilde{\bar{C}}}_{w} = \diag(\tilde{\bar{\bm{C}}}_{w})$,
$\underline{G} = \col(\tilde{\bm{G}} \tilde{\bm{P}}_{d})$, 
$\underline{G}_{0} = \col(\bm{G}_{0} \tilde{\bm{P}}_{d})$,
$\underline{G}_{1} = \col(\bm{G}_{1} \tilde{\bm{P}}_{d})$,
$\underline{G}_{w} = \col(\bm{G}_{w} \tilde{\bm{P}}_{d})$,
$\underline{G}_{y} = \col(\tilde{\bm{G}}_{y} \tilde{\bm{P}}_{d} - \tilde{P}_{r})$.

\subsubsection{Proof of Lemma \ref{lem:extendedRegulatorEquations}}\label{sec:appendix.extendedRegulatorEquations}
The following lemma is anticipated and the proof of Lemma \ref{lem:extendedRegulatorEquations} is stated afterwards.

\begin{lemma}[Transmission Zeros]\label{lem:robustOutputRegulation.rankNumerator}
	The numerator $\underline{N}(s)  \in \mathbb{C}^{N \bar{n}_{-} \times \underline{n}_{-}}$, $\underline{n}_{-} = \sum_{i = 1}^{g} N^{i} n_{-}^{i}$, of the transfer matrix $\underline{F}(s) = \underline{N}(s) \underline{D}^{-1}\!(s)$ from $u = \col(\bm{u}) \in \mathbb{R}^{\underline{n}_{-}}$ to $y = \col(\bm{y}) \in \mathbb{R}^{N \bar{n}_{-}}$ of the uncertain agents \eqref{eq:agent} is
	\begin{align}\label{eq:aggregated.numerator}
		\underline{N}(s) \!=\!
			\underline{\bar{\mathcal{C}}}_{x}[\underline{M}(s)] \!\det(s I \!-\! \underline{\tilde{\bar{F}}}_{w})
			\!+\! \underline{\tilde{\bar{C}}}_{w} \text{adj}(s I \!-\! \underline{\tilde{\bar{F}}}_{w}) \underline{\bar{B}}_{w}\!
	\end{align}
	with
	$\underline{M}(z, s) = \underline{\Phi}(z, 0, s) (\underline{E}_{-} + \underline{E}_{+} \underline{\bar{Q}}_{0}) - \int_{0}^{z} \!\underline{\Phi}(z, \zeta, s) \linebreak \cdot \! \underline{\Lambda}^{-1}\!(\zeta) \underline{\bar{A}}_{0}(\zeta) \d\zeta$
	and
	$\underline{\Phi}(z, \zeta, s) = \e^{s \int_{\zeta}^{z} \underline{\Lambda}\!^{-1}\!(\eta) \d\eta }$.
	Assume that the system \eqref{eq:extendedSystem.3.internalModel}--\eqref{eq:extendedSystem.3.output} is asymptotically stable. Then, for all $\mu \in \sigma(S)$ the numerator \eqref{eq:aggregated.numerator} satisfies $\text{rank}\,\underline{N}(\mu) = N \bar{n}_{-}$.
\end{lemma}

\begin{proof}
	As the transfer behavior from $u(t)$ to $y(t)$ is investigated, $v(t)=0$ is used in the following.
	Since the transformations \eqref{eq:transformation.extendedRegulatorEquations} are boundedly invertible, the transfer behavior of the followers \eqref{eq:agent} is equivalent to the transfer behavior of the system \eqref{eq:agent.transformed}.
	The aggregation of \eqref{eq:agent.transformed} is \eqref{eq:extendedSystem.3.pde}, \eqref{eq:extendedSystem.3.bc0}, \eqref{eq:extendedSystem.3.ode}, \eqref{eq:extendedSystem.3.output} and
	$\tilde{\varepsilon}_{-}(1, t) = u(t) + \underline{\tilde{\mathcal{K}}}[\tilde{\varepsilon}(t)] + \underline{\tilde{K}}_{w} w(t)$
	with $\underline{\tilde{\mathcal{K}}} = \diag(\tilde{\bm{\mathcal{K}}})$, $\underline{\tilde{K}}_{w} = \diag(\tilde{\bm{K}}_{w})$.
	Then, considering the solution $\tilde{\varepsilon}(z, t) = \check{\tilde{\varepsilon}}(z) \e^{s t}$, $w(t) = \check{w} \e^{s t}$ for $u(t) = \check{u} \e^{s t}$, $t \geq 0$, $s \in \mathbb{C}$, yielding $y(t) = \check{y} \e^{s t}$ with the initial values $\check{\tilde{\varepsilon}}(z)$, $z \in [0, 1]$, $\check{w}$ (see \cite{Zw04}), and repeating the calculations from Appendix \ref{sec:appendix.controllability}, shows that $\underline{N}(s)$ in \eqref{eq:aggregated.numerator} is the numerator of the transfer matrix in question.
	Moreover, during the latter calculations also
	\begin{align}\label{eq:aggregated.laplace}
		\check{w} = (sI - \underline{\tilde{\bar{F}}}_{w})^{-1} \underline{\bar{B}}_{w} \check{\tilde{\varepsilon}}_{-}(0), &&\check{\tilde{\varepsilon}}(z) = \underline{M}(z, s) \check{\tilde{\varepsilon}}_{-}(0)
	\end{align}
	result with $\underline{M}(z, s)$ given in the lemma.
	
	Next, the networked controlled MAS \eqref{eq:extendedSystem.3} is examined in order to utilize its stability.
	Note that the equations \eqref{eq:aggregated.laplace} hold independent of the boundary condition at $z=1$, which allows to use them for the networked controlled MAS \eqref{eq:extendedSystem.3}.
	Hence, inserting \eqref{eq:aggregated.laplace} in \eqref{eq:extendedSystem.3.bc1} with $\bar{v}(t) = \check{\bar{v}} \e^{s t}$ and the initial value $\check{\bar{v}}$ gives
	$\underline{K}_{\bar{v}} \check{\bar{v}} = \underline{P}(s) \check{\tilde{\varepsilon}}_{-}(0)$
	with
	\begin{align}\label{eq:laplace.P.K}
		\underline{P}(s) = \underline{E}_{-}^{\top} \underline{M}(1, s) \!-\! \underline{\mathcal{K}}[\underline{M}(s)] \!-\! \underline{\tilde{K}}_{w} (s I \!-\! \underline{\tilde{\bar{F}}}_{w})^{-1} \underline{\bar{B}}_{w}.
	\end{align}
	Inserting \eqref{eq:extendedSystem.3.output} into \eqref{eq:extendedSystem.3.internalModel} and using \eqref{eq:aggregated.laplace} yields 
	$f_{w}(s) (s I - \underline{S}) \check{\bar{v}} = (H \otimes \tilde{B}_{y}) \underline{N}(s) \check{\tilde{\varepsilon}}_{-}(0)$ with $f_{w}(s) = \det(sI - \underline{\tilde{\bar{F}}}_{w})$.
	The latter and $\underline{K}_{\bar{v}} \check{\bar{v}} = \underline{P}(s) \check{\tilde{\varepsilon}}_{-}(0)$ are combined into
	\begin{align}\label{eq:modalSolution.initialValues}
		\!\begin{bmatrix}
			\underline{P}(s) & -\underline{K}_{\bar{v}} \\
			(H \Otimes \tilde{B}_{y}) \underline{N}(s) \!& \!-(s I \!-\! \underline{S}) f_{w}(s)
		\end{bmatrix}\hspace*{-1mm}
		\begin{bmatrix}
			\check{\tilde{\varepsilon}}_{-}(0) \\
			\check{\bar{v}}
		\end{bmatrix}
		\!= 0 .\hspace*{-1mm}
	\end{align}
	Due to the assumed asymptotic stability of the networked controlled MAS, the initial values $\check{\bar{v}} = 0$, $\check{\tilde{\varepsilon}}_{-}(0) = 0$ must be the only admissible ICs for the modes $s = \mu \in \sigma(\underline{S}) \subset \text{j} \mathbb{R}$.
	Otherwise, non-decaying solutions $\bar{v}(t) = \e^{\mu t} \check{\bar{v}}$, $\tilde{\varepsilon}_{-}(0, t) = \e^{\mu t} \check{\tilde{\varepsilon}}_{-}(0)$ would result contradicting the assumed asymptotic stability.
	This means that \eqref{eq:modalSolution.initialValues} with $s=\mu$ has only the trivial solution $\check{\tilde{\varepsilon}}_{-}(0) = 0$, $\check{\bar{v}} = 0$.
	Consequently, the matrix of the dimension $(\underline{n}_{-} + N n_{\bar{v}}) \times (\underline{n}_{-} + N n_{\bar{v}})$ in \eqref{eq:modalSolution.initialValues} with $\underline{n}_{-} = \dim \check{\tilde{\varepsilon}}_{-}(0)$ and $N n_{\bar{v}} = \dim \check{\bar{v}}$ is nonsingular.
	Consider the matrix $\diag(I_{\underline{n}_{-}}, I_{N \bar{n}_{-}} \otimes \rho)$ which contains the eigenvector $\rho$ of $S$ w.r.t. the eigenvalue $\mu$ and has the rank $\underline{n}_{-} + N \bar{n}_{-}$.
	Then, Sylvester's rank inequalities (see \cite[Cor. 2.5.10.]{Bern05}) and $(\underline{S} - \mu I) (I_{N \bar{n}_{-}} \otimes \rho ) = 0$ lead to
	\begin{align}\label{eq:rank.1}
		&\rank(\!\begin{bmatrix}
			\underline{P}(\mu) & -\underline{K}_{\bar{v}} \\
			(H \Otimes \tilde{B}_{y}) \underline{N}(\mu) \!&\! f_{w}(\mu) (\underline{S}\!-\!\mu I)
		\end{bmatrix}
		\hspace*{-1mm}
		\begin{bmatrix}
			I_{\underline{n}_{-}} & 0 \\ 0 & I_{N \bar{n}_{-}} \Otimes \rho
		\end{bmatrix}\!)
		\!=\nonumber\\&
		\rank\! \begin{bmatrix}
			\underline{P}(\mu)\! & \!-\underline{K}_{\bar{v}}(I_{N \bar{n}_{-\!}} \Otimes \rho) \\
			(H \otimes \tilde{B}_{y})\! \underline{N}(\mu) & 0
		\end{bmatrix}
		%
		%
		\!=\! \underline{n}_{-\!} \!+\! N \bar{n}_{-}.\hspace*{-1mm}
	\end{align}
	By utilizing $\rank(H \otimes \tilde{B}_{y}) = \rank(H) \rank(\tilde{B}_{y})$ (c.f. \cite[Fact 7.4.24]{Bern05}), $\rank H =  N$ (see Lemma \ref{lem:H}) and $\rank \tilde{B}_{y} = \rank (I_{\bar{n}_{-}} \otimes b_{y}) = \bar{n}_{-}$ ($b_{y} \neq 0$ because $(S, b_{y})$ is controllable), one obtains $\rank(H \otimes \tilde{B}_{y}) = N \bar{n}_{-}$.
	Thus, $\rank((H \otimes \tilde{B}_{y}) \underline{N}(\mu)) = \rank \underline{N}(\mu)$ results in view of $\underline{N}(\mu) \in \mathbb{C}^{N \bar{n}_{-} \times \underline{n}_{-}}$, $N \bar{n}_{-} \leq \underline{n}_{-}$. With this and \eqref{eq:rank.1} follows that
	\begin{align}\label{eq:rank.2}\rank \begin{bmatrix}
			\underline{P}(\mu) & -\underline{K}_{\bar{v}}(I_{N \bar{n}_{-}} \otimes \rho) \\
			\underline{N}(\mu) & 0
		\end{bmatrix}
		= \underline{n}_{-} + N \bar{n}_{-}
	\end{align}
	has to hold for this matrix of the dimension $(\underline{n}_{-} + N \bar{n}_{-}) \times (\underline{n}_{-} + N \bar{n}_{-})$.
	The latter can only be fulfilled if $\rank \underline{N}(\mu) = N \bar{n}_{-}$, which completes the proof.
\end{proof}

{\textit{Proof of Lemma \ref{lem:extendedRegulatorEquations}: }
	Consider the generalized right eigenvectors $\rho_{k}^{l}$ of $S$ w.r.t. the eigenvalues $\mu_{k}$, $k=1, \ldots, \bar{k}$, with the dimension $l=1, \ldots, \bar{l}_{k}$ of the corresponding Jordan blocks and 
	$S \rho_{k}^{l} = \mu_{k} \rho_{k}^{l} + \rho_{k}^{l-1}$,
	$\rho_{k}^{0} = 0$.
	Then, defining $\gamma_{\bar{v}, k}^{l} = \Upsilon\!_{\bar{v}} \rho_{k}^{l}$, $\gamma_{k}^{l}(z) = \Upsilon\!_{\varepsilon}(z) \rho_{k}^{l}$, $\gamma_{w, k}^{l} = \Upsilon\!_{w} \rho_{k}^{l}$, multiplication of \eqref{eq:extendedRegulatorEquations} with $\rho_{k}^{l}$ from the right side and rearranging some terms leads to
	\begin{subequations}\label{eq:extendedRegulatorEquations.2}
		\begin{align}
			(\underline{S} - \mu_{k} I )\gamma_{\bar{v}, k}^{l} &= \gamma_{\bar{v}, k}^{l-1} \label{eq:extendedRegulatorEquations.2.v}\\
			(\mu_{k} I \!-\! \underline{\tilde{\bar{F}}}_{w}) \gamma_{w, k}^{l} &= \underline{\bar{B}}_{w} \underline{E}_{-}^{\top} \gamma_{k}^{l}(0) \!+\! \underline{G}_{w} \rho_{k}^{l} \!-\! \gamma_{w, k}^{l-1}\!\label{eq:extendedRegulatorEquations.2.w}\\
			\d_{z} \gamma_{k}^{l}(z) &= (\underline{\bar{\Lambda}}(z))^{-1} (\mu_{k} \gamma_{k}^{l}(z) + \gamma_{k}^{l-1}(z)
			\nonumber\\*&~
			-\underline{\bar{A}}_{0}(z) \underline{E}_{-}^{\top} \gamma_{k}^{l}(0) 
			- \underline{G}(z) \rho_{k}^{l} ) \label{eq:extendedRegulatorEquations.2.pde}\\
			(\underline{E}_{+}^{\top} \!-\! \underline{\bar{Q}}_{0} \underline{E}_{-}^{\top}) \gamma_{k}^{l}(0) &= \underline{G}_{0} \rho_{k}^{l} \label{eq:extendedRegulatorEquations.2.bc0}\\
			\underline{E}_{-}^{\top} \gamma_{k}^{l}(1) \!-\! \underline{\mathcal{K}}[\gamma_{k}^{l}] &= \underline{K}_{\bar{v}} \gamma_{\bar{v}, k}^{l} + \underline{\tilde{K}}_{w} \gamma_{w, k}^{l} + \underline{G}_{1} \rho_{k}^{l} \label{eq:extendedRegulatorEquations.2.bc1}\\
			\underline{\bar{\mathcal{C}}}_{x}[\gamma_{k}^{l}] &= - \underline{\tilde{\bar{C}}}_{w} \gamma_{w, k}^{l} - \underline{G}_{y} \rho_{k}^{l} \label{eq:extendedRegulatorEquations.2.output}.
		\end{align}
	\end{subequations}
	Since the eigenvalue $\mu_{k}$ of $S$ is also an eigenvalue of $\underline{S} = I_{N \bar{n}_{-}}\otimes S$, it is evident from \eqref{eq:extendedRegulatorEquations.2.v} that $\gamma_{\bar{v}, k}^{l}$ must be a generalized eigenvector of $\underline{S}$ w.r.t. the eigenvalue $\mu_k$.
	However, $\Upsilon_{\bar{v}}$ is determined by a number of $n_{v} = |\sigma(S)|$ vectors $\gamma_{\bar{v}, k}^{l}$, while the matrix $\underline{S}$ is of dimension $(n_{v} \bar{n}_{-} N) \times (n_{v} \bar{n}_{-} N)$ and its generalized eigenvectors span $\mathbb{C}^{n_{v} \bar{n}_{-} N}$.
	Hence, these remaining degrees of freedom can be used to find $\gamma_{\bar{v}, k}^{l}$ such that \eqref{eq:extendedRegulatorEquations.2.bc1} is satisfied as well. For this, the parametrization
	$\gamma_{\bar{v}, k}^{l} = \beta_{k}^{l} \otimes \rho_{k}^{1} + \Delta_{k}^{l}$,
	$\Delta_{k}^{l} = \sum_{m=1}^{l-1} \beta_{k}^{m} \otimes  \rho_{k}^{l-m+1}$
	is established, which solves \eqref{eq:extendedRegulatorEquations.2.v} without any constraints on the parameters $\beta_{k}^{l} \in \mathbb{R}^{N \bar{n}_{-}}$, $l=1, \ldots, \bar{l}_{k}$.
	The latter expression for $\gamma_{\bar{v}, k}^{l}$ is obtained by first noting that $\gamma_{\bar{v}, k}^{1} = \beta_{k}^{1} \otimes \rho_{k}^{1}$ is an eigenvector of $\underline{S}$.
	From this starting point, the generalized eigenvectors $\gamma_{\bar{v}, k}^{l}$ for $l = 2, \ldots, \bar{l}_{k}$ can be determined iteratively. For this, replace $\gamma_{\bar{v}, k}^{l-1}$ in \eqref{eq:extendedRegulatorEquations.2.v} with $\gamma_{\bar{v}, k}^{1} = \beta_{k}^{1} \otimes \rho_{k}^{1}$ for $l=2$, or otherwise with the result of the previous iteration, which depends only on $\beta_{k}^{m}$, $\rho_{k}^{m}$, $m= 1, \ldots, l-1$. Next, all $\rho_{k}^{m}$ are substituted via $\rho_{k}^{m} = (S -\mu_{k} I) \rho_{k}^{m+1}$ and $0 = \beta_{k}^{l} \otimes (S - \mu_{k} I ) \rho_{k}^{1}$ is added.
	After utilizing $\beta_{k}^{m} \otimes (S - \mu_{k} I ) = (\underline{S} - \mu_{k} I ) ( \beta_{k}^{m} \otimes I_{n_{v}})$, the common factor $\underline{S} - \mu_{k} I$ can be taken out and the result $\gamma_{\bar{v}, k}^{l} = \beta_{k}^{l} \otimes \rho_{k}^{1} + \Delta_{k}^{l}$ is obtained.
	Next, from \eqref{eq:extendedRegulatorEquations.2.pde}, \eqref{eq:extendedRegulatorEquations.2.bc0} follows similarly to \cite[App. 2]{DeuGab20} that
	\begin{align}\label{eq:gamma}
		\gamma_{k}^{l}(z) &= \underline{M}(z, \mu_{k}) \underline{E}_{-}^{\top} \gamma_{k}^{l}(0) + \underline{c}_{k}^{l}(z)
	\end{align}
	with
	$\underline{M}(z, \mu_{k}), \underline{\Phi}(z, \zeta, \mu_{k})$ given in Lemma \ref{lem:robustOutputRegulation.rankNumerator}
	and
	$\underline{c}_{k}^{l}(z) = \underline{\Phi}(z, 0, \mu_{k}) \underline{E}_{+} \underline{G}_{0} \rho_{k}^{l} + \int_{0}^{z} \underline{\Phi}(z, \zeta, \mu_{k}) \underline{\Lambda}^{-1}(\zeta) (\gamma_{k}^{l-1}(\zeta) - \underline{G}(\zeta) \rho_{k}^{l}) \d\zeta$.
	Furthermore, \eqref{eq:extendedRegulatorEquations.2.w} can be solved for $\gamma_{w, k}^{l}$ because $\mu_{k} I - \underline{\tilde{\bar{F}}}_{w}$ is nonsingular since otherwise \eqref{eq:aggregated.laplace} with $s=\mu$ would permit a non-decaying solution $w(t) = \check{w} \e^{\mu t}$ with $(\mu I - \underline{\tilde{\bar{F}}}_{w})\check{w} = 0$, $\check{w} \neq 0$, which contradicts the assumed asymptotic stability.
	Note that in view of $\gamma_{\bar{v}, k}^{l} = \beta_{k}^{l} \otimes \rho_{k}^{1} + \Delta_{k}^{l}$ and $\beta_{k}^{l} \otimes \rho_{k}^{1} = (I_{N \bar{n}_{-}} \otimes \rho_{k}^{1}) \beta_{k}^{l}$, the term $\underline{K}_{\bar{v}} \gamma_{\bar{v}, k}^{l}$ in \eqref{eq:extendedRegulatorEquations.2.bc1} can be written as
	$\underline{K}_{\bar{v}} \gamma_{\bar{v}, k}^{l} = \underline{K}_{\bar{v}} (I_{N \bar{n}_{-}} \otimes \rho_{k}^{1}) \beta_{k}^{l} + \underline{K}_{\bar{v}} \Delta_{k}^{l}$.
	Thus, by stacking \eqref{eq:extendedRegulatorEquations.2.bc1} and \eqref{eq:extendedRegulatorEquations.2.output}, then substituting $\underline{K}_{\bar{v}} \gamma_{\bar{v}, k}^{l}$ as well as $\gamma_{w, k}^{l}$, $\gamma_{k}^{l}(z)$  with \eqref{eq:extendedRegulatorEquations.2.w}, \eqref{eq:gamma} and some simple calculations, one obtains
	\begin{align}\label{eq:gamma.linearMatrixEquation}
		\begin{bmatrix*}
			\underline{P}(\mu_{k}) & - \underline{K}_{\bar{v}} (I_{N \bar{n}_{-}} \Otimes \rho_{k}^{1}) \\
			\underline{N}(\mu_{k}) & 0
		\end{bmatrix*}
		\begin{bmatrix*} \underline{E}_{-}^{\top} \gamma_{k}^{l}(0) \\ \beta_{k}^{l} \end{bmatrix*}
		= r_{k}^{l}
	\end{align}
	with
	$\underline{P}(\mu_{k})$ and $\underline{N}(\mu_{k})$ given in \eqref{eq:laplace.P.K}, \eqref{eq:aggregated.numerator} as well as
	$r_{k}^{l} = \col(r_{1,k}^{l}, r_{2,k}^{l})$,
	$r_{1,k}^{l} = \underline{K}_{\bar{v}} \Delta_{k}^{l} + \underline{\tilde{K}}_{w} (\mu_{k} I \!-\! \underline{\tilde{\bar{F}}}_{w})^{-1} (\underline{G}_{w} \rho_{k}^{l} \!-\! \gamma_{w, k}^{l-1}) + \underline{G}_{1} \rho_{k}^{l} - \underline{E}_{-}^{\top} \underline{c}_{k}^{l}(1) + \underline{\mathcal{K}}[\underline{c}_{k}^{l}]$,
	$r_{2,k}^{l} = - \det(\mu_{k} I \!-\! \underline{\tilde{\bar{F}}}_{w})(\underline{\tilde{\bar{C}}}_{w} (\mu_{k} I \!-\! \underline{\tilde{\bar{F}}}_{w})^{-1} (\underline{G}_{w} \rho_{k}^{l} \!+\! \gamma_{w, k}^{l-1}) + \underline{G}_{y} \rho_{k}^{l} + \underline{\bar{\mathcal{C}}}_{x}[\underline{c}_{k}^{l}])$.
	As the assumptions of Lemma \ref{lem:robustOutputRegulation.rankNumerator} are met by the assumptions of Lemma \ref{lem:extendedRegulatorEquations}, the matrix in \eqref{eq:gamma.linearMatrixEquation} is nonsingular (see \eqref{eq:rank.2}) and a unique solution $\underline{E}_{-}^{\top} \gamma_{k}^{l}(0)$, $\beta_{k}^{l}$ of \eqref{eq:gamma.linearMatrixEquation} can be calculated.
	Then, $\gamma_{w, k}^{l}$, $\gamma_{k}^{l}(z)$, $\gamma_{\bar{v}, k}^{l}$ follow by an evaluation of \eqref{eq:extendedRegulatorEquations.2.w}, \eqref{eq:gamma} and $\gamma_{\bar{v}, k}^{l} = \beta_{k}^{l} \otimes \rho_{k}^{1} + \Delta_{k}^{l}$.
	Together with the inverse of the transformation matrix in the corresponding Jordan canonical form resulting from $\rho_{k}^{l}$, they yield the solution $\Upsilon\!_{\bar{v}}, \Upsilon\!_{\varepsilon}, \Upsilon\!_{w}$ of the extended regulator equations \eqref{eq:extendedRegulatorEquations}.
	Note that $\underline{c}_{k}^{l}(z)$ in \eqref{eq:gamma} contains $\underline{G} = \col(\tilde{\bm{G}})$.
	Since $\tilde{\bm{G}}$ is piecewise continuous, the elements of $\gamma_{k}^{l}(z)$ and $\Upsilon\!_{\varepsilon}(z)$ are piecewise $C^{1}$.\proofEnd
}{\acceptancenotice\copyrightnotice}%

\bibliographystyle{IEEEtranS}
\bibliography{mybib}%
{\vspace*{-1pt}\acceptancenotice\copyrightnotice}%
\begin{IEEEbiography}[{\includegraphics[width=1in,height=1.25in,clip,keepaspectratio]{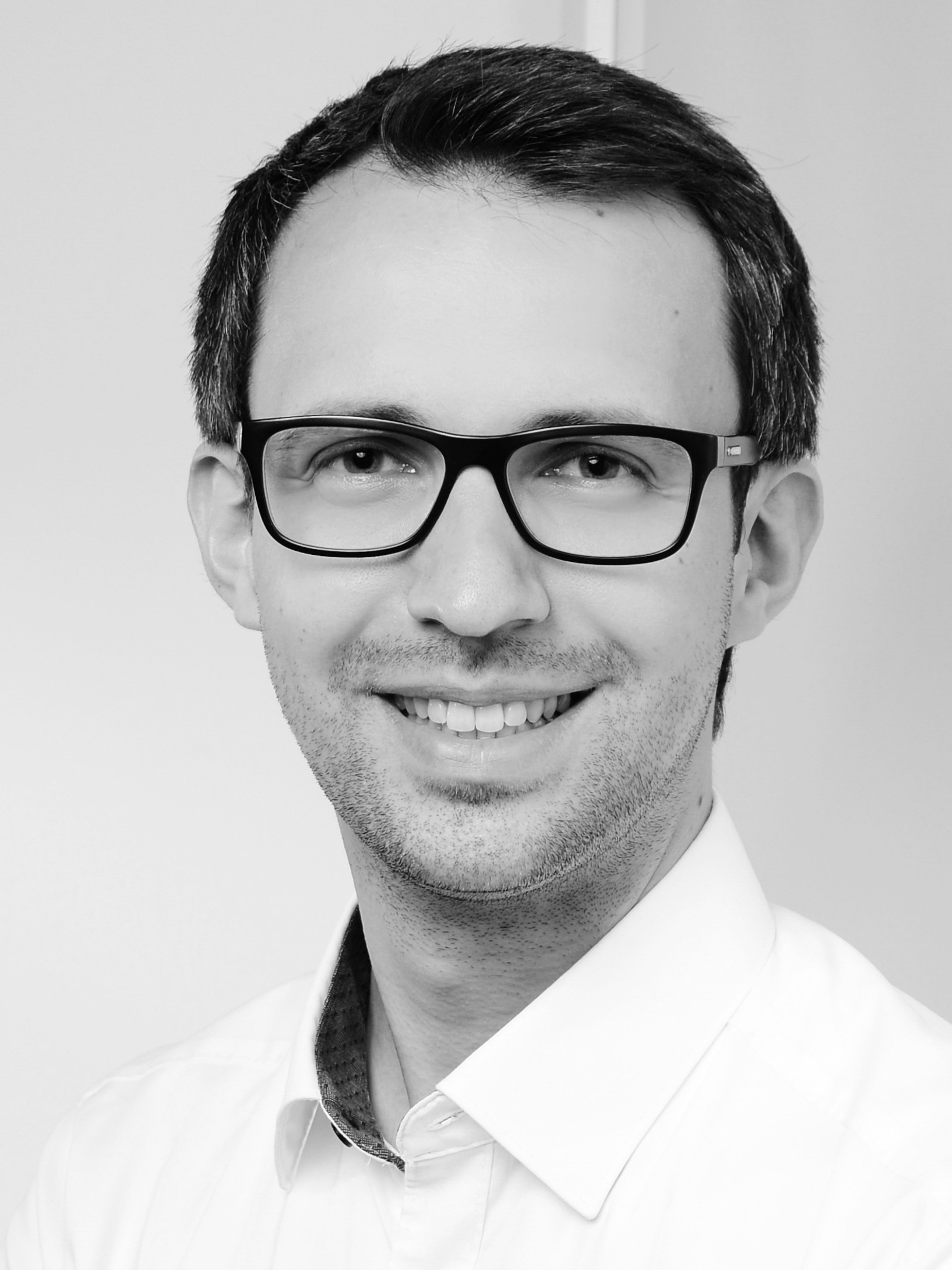}}]
	{Jakob Gabriel} received his M.Sc. degree in Mechatronics in 2016 from Friedrich-Alexander-Universität Erlangen-Nürnberg (FAU), Germany.
	Afterwards, he worked as a research assistant in the group of Prof. Deutscher at the Chair of Automatic Control at FAU and until December 2022 at the Institute of Measurement, Control and Microtechnology, Ulm University, Germany.
	His research interests are the control of hyperbolic systems using backstepping and cooperative output regulation.
	Since February 2023 he is with Primetals Technologies GmbH, Erlangen, Germany, as a developer for automation and control of metal rolling mills.
\end{IEEEbiography}
\begin{IEEEbiography}[{\includegraphics[width=1in,height=1.25in,clip,keepaspectratio]{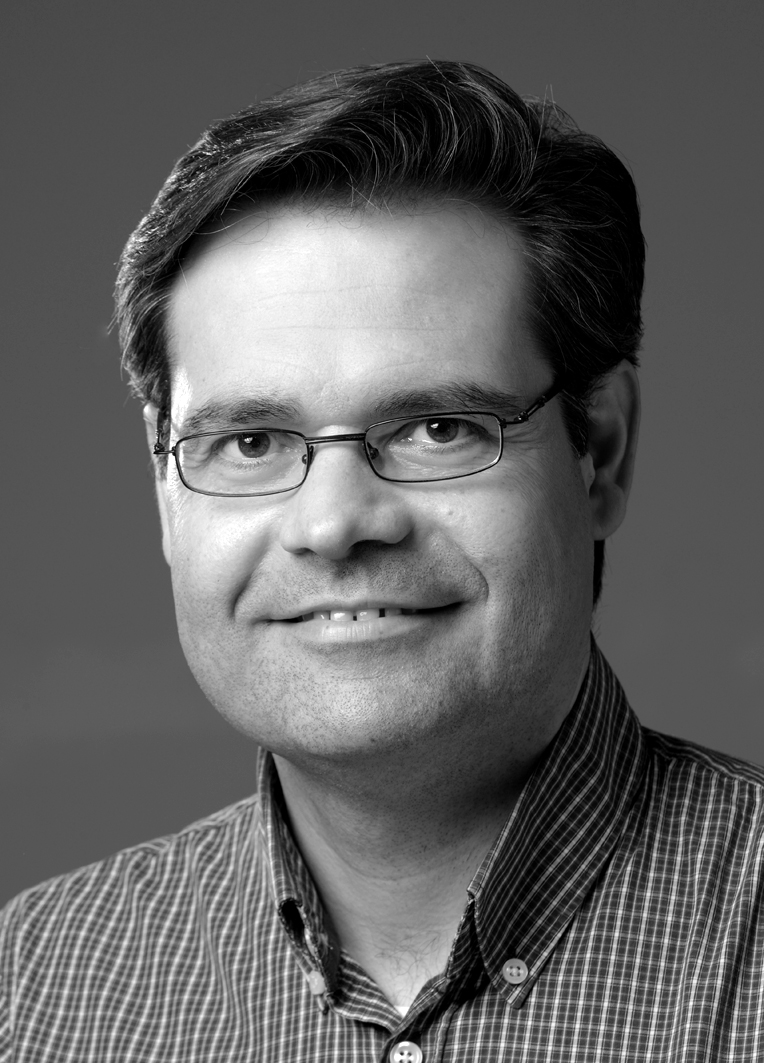}}]
	{Joachim Deutscher} (M '18)
	received the Dipl.-Ing. (FH) degree in Electrical Engineering from the University of Applied Sciences W\"urzburg-Schweinfurt-Aschaffenburg, Germany, in 1996, the Dipl.-Ing. Univ. degree in Electrical Engineering, the Dr.-Ing. and the Dr.-Ing. habil. degrees both in Automatic Control from the Friedrich-Alexander Universit\"at Erlangen-N\"urnberg (FAU), Germany, in 1999, 2003 and 2010, respectively.
	
	From 2003--2010 he was a Senior Researcher at the Chair of Automatic Control (FAU), in 2011 he was appointed Associate Professor and in 2017 Professor at the same university. Since April 2020 he is a Full Professor at the Institute of Measurement, Control and Microtechnology, Ulm University.
	His research interests include control of distributed-parameter and multi-agent systems as well as control theory for nonlinear lumped-parameter systems with applications in mechatronics and robotics. Dr. Deutscher has co-authored a book on state feedback control for linear lumped-parameter systems: Design of Observer-Based Compensators (Springer, 2009) and is author of the book: State Feedback Control of Distributed-Parameter Systems (in German) (Springer, 2012). At present he serves as Associate Editor for Automatica. 
\end{IEEEbiography}
\vfill%

\end{document}